\documentclass[12pt,twoside]{article}
\usepackage[english]{babel}
\usepackage[latin1]{inputenc}
\usepackage{amsmath}
\usepackage{amsthm}
\usepackage{amssymb,amsfonts}
\usepackage{comment}
\usepackage{tikz-cd}
\usepackage{tikz}
 \usetikzlibrary {shapes.geometric} 
\usepackage{orcidlink}

\usepackage{graphicx}                   
\usepackage{hyperref}

\newtheorem{assn}{Assumption}

\begin{document}
\raggedbottom
\pagestyle{myheadings}
\title
{\textsc{Asymptotic geometry of non-abelian Hodge theory and Riemann--Hilbert correspondence, rank three $\widetilde{E}_6$ case}}
\author{by \textsc{Mikl\'os Eper}\footnote{\textsl{Department of Algebra and Geometry, Institute of Mathematics, Faculty of Natural Sciences, Budapest University of Technology and
Economics, M\H uegyetem rkp. 3., Budapest H-1111, Hungary}, e-mail: \href{mailto:epermiklos@gmail.com}{epermiklos@gmail.com}} and \textsc{Szil\'ard Szab\'o}\footnote{Institute of Mathematics, Faculty of Science, E\"otv\"os Lor\'and University, P\'azm\'any P\'eter s\'et\'any 1/C, Budapest, Hungary, H-1117; Alfr\'ed R\'enyi Institute of Mathematics,
Re\'altanoda utca 13-15., Budapest 1053, Hungary, e-mails:\href{mailto:szilard.szabo@ttk.elte.hu}{szilard.szabo@ttk.elte.hu} and \href{mailto:szabo.szilard@renyi.hu}{szabo.szilard@renyi.hu}}
\orcidlink{https://orcid.org/my-orcid?orcid=0009-0008-7435-2021}}
\maketitle


\begin{abstract}
We prove the Geometric P=W conjecture in rank 3 on the three-punctured sphere. 
We describe the topology at infinity of the related character variety. 
We use asymptotic abelianization of harmonic bundles away from the ramification divisor and an equivariant approach near the branch points to find the WKB (also known as Liouville--Green or phase-integral) expansion of the involved maps. 
We analyze the Stokes phenomenon governing their behavior. 
\end{abstract}



\newtheorem{theorem}{Theorem}[section]
\newtheorem{corollary}[theorem]{Corollary}
\newtheorem{conjecture}{Conjecture}[section]
\newtheorem{lemma}[theorem]{Lemma}
\newtheorem{exmple}[theorem]{Example}
\newtheorem{defn}[theorem]{Definition}
\newtheorem{prop}[theorem]{Proposition}
\newtheorem{rmrk}[theorem]{Remark}
\newtheorem{claim}[theorem]{Claim}

\newenvironment{definition}{\begin{defn}\normalfont}{\end{defn}}
\newenvironment{remark}{\begin{rmrk}\normalfont}{\end{rmrk}}
\newenvironment{example}{\begin{example}\normalfont}{\end{example}}
\newenvironment{acknowledgement}{{\bf Acknowledgement:}}

\newcommand\restr[2]{{
  \left.\kern-\nulldelimiterspace 
  #1 
  \vphantom{\big|} 
  \right|_{#2} 
  }}

\section{Introduction and statement of the main result}

The main characters of this paper are, on the one hand, the moduli space $\mathcal{M}_{\operatorname{Dol}}(\alpha )$ of gauge-equivalence classes of rank $3$ Higgs bundles on the complex projective line $\mathbb{C}P^1$ with three logarithmic points (called the Dolbeault moduli space) and, on the other hand, the space $\mathcal{M}_{\operatorname{B}}(\textbf{c})$ of representations of the fundamental group of the thrice-punctured $\mathbb{C}P^1$ in $\operatorname{SL}(3, \mathbb{C})$ up to overall conjugation (called the Betti moduli space or character variety). 
These two spaces are special instances corresponding to the extended root system $\widetilde{E}_6$ of a more general setup, where the genus of the underlying curve and the structure group of the objects may be arbitrary. 
It is known in general that the Dolbeault and Betti moduli spaces are diffeomorphic to each other via a composition of the non-abelian Hodge and Riemann--Hilbert correspondences. 
However, this diffeomorphism is highly transcendental, and extensive literature is devoted to understanding its properties, as well as to using it to translate known facts about one of these moduli spaces to new statements about the other one. 
For instance, it is known that these moduli spaces carry natural complex algebraic variety structures (of dimension $2$ in the $\widetilde{E}_6$ case), but the diffeomorphism between them is not compatible with their complex structures. 
Our goal is to investigate the asymptotic behavior of this diffeomorphism $\psi$ at infinity in the special case  $\widetilde{E}_6$ specified above (i.e., genus $0$, rank $3$, and with suitably chosen parameters to be specified in Section~\ref{sec:choices}), in the hope that this particular case will shed light on some phenomena that continue to hold in the general case too. 
We may summarize our main result as:

\begin{theorem}\label{thm:main}
    The Geometric P=W conjecture holds for rank 3 tame harmonic bundles 
    over the three-punctured sphere. 
\end{theorem}
Along the way, we also give a self-contained proof for the following:
\begin{theorem}[Proposition 3.1]\label{thm:secondary}
    The $\operatorname{GL}(3,\mathbb{C})$ character variety of the three-punctured sphere admits a smooth compactification by a curve of type $I_1$. 
    In particular, the body of its nerve complex is of homotopy type $S^1$. 
\end{theorem}
We note that P.~Etingof, A.~Oblomkov, E.~Rains have obtained essentially the same result about the compactifying divisor of type $I_1$ using representation theoretical  techniques~\cite[Proposition~6.6]{EtOblRa}, including in the cases corresponding to the affine root systems $\widetilde{E}_7$ and $\widetilde{E}_8$ (rather than just $\widetilde{E}_6$ studied here). 
We now turn to describing the context and precise meaning of Theorem~\ref{thm:main}. 

One motivation of this study comes from Hitchin's WKB problem~\cite{KatNollPandSimp}, which roughly reads as follows: consider a $\mathbb{C}^{\times}$ orbit in the Hitchin base and a smooth lift to the Dolbeault space. Consider the family of flat connections corresponding to this lift, and determine the behaviour of the associated transport matrices, as the point of the Hitchin base converges to infinity along the $\mathbb{C}^{\times}$-orbit. 
For more about the WKB approximation theory of the Schr\"odinger operator, see~\cite{Vor}. 
More recently,~\cite[Section~13]{GMN} and~\cite[Section~2.3]{Moc} carried out a WKB analysis of Hitchin's equations. 
It remains an actively studied area, related to different other fields, e.g. resurgence~\cite[Section~6]{KonSoiResurgence}, Borel resummation~\cite[Section~5]{Nikolaev}, Teichm\"uller theory~\cite{OSWW}, cluster algebras~\cite{IwaNaka}, and so on. 
As we will point out in the next paragraph, this article adds Hodge theory to the list.

Another source of inspiration is the so-called P=W conjecture. 
The Betti space is known to be an affine algebraic variety, and as such its cohomology spaces carry a mixed Hodge structure~\cite{Del}. 
On the other hand, the Hitchin map endows the cohomology spaces of the Dolbeault space with a perverse Leray filtration. 
The conjecture, formulated by M.~de~Cataldo, T.~Hausel and L.~Migliorini~\cite{CatHauMig}, states that the diffeomorphism $\psi$ between the spaces respects these filtrations. 
Recently this has been an intensely investigated area, with strong ties to other fields such as 
Cohomological Hall Algebras, the geometry of the affine Springer fiber, and Donaldson--Thomas theory. 
In~\cite{CatHauMig} it was proved in the rank two case for compact curves. 
The complete proof for the P=W conjecture came recently from two different sources, using different techniques \cite{HauMellMinSch} and \cite{MauSh}.

A geometric counterpart of the conjecture was proposed in~\cite[Conjecture~1.1]{KatNollPandSimp} and~\cite[Conjecture~11.1]{Simp4}. 
So far, this Geometric $P=W$ conjecture has received significantly less attention than the original cohomological version. 
Mauri, Mazzon and Stevenson~\cite{MauMazStev} proved it in genus $1$ and type $A$. 
The second author dealt with the conjecture in the case of rank 2 Higgs bundles with irregular singularities over $\mathbb{C}P^1$, belonging to the Painlev\'e cases~\cite{SzSz2}, and with rank 2 logarithmic Higgs bundles over the five-punctured sphere, by establishing the WKB-analysis of coordinates of the character variety~\cite{SzSz}. 
A.~N\'emethi and the second author provided a different proof for these Painlev\'e cases using low-dimensional topology techniques~\cite{NemSz}. 
As far as the authors know, these are the only cases where the full assertion of the Geometric P=W is proved.
The Geometric $P=W$ conjecture provides an understanding of the Hitchin map on the Betti side, away from a compact subset. 
More precisely, it asserts the existence of the following homotopy commutative diagram:
\begin{equation}\label{diagram:PW}
\begin{tikzcd}
\mathcal{M}_{\operatorname{Dol}}\setminus H^{-1}(B_R(0)) \arrow{r}{\psi} \arrow[swap]{d}{H} & \mathcal{M}_{\operatorname{B}}\setminus \psi(H^{-1}(B_R(0))) \arrow{d}{S} \\%
\mathcal{H}\setminus B_R(0) \arrow{r}{}& \lvert\mathcal{D}\partial\mathcal{M}_B\rvert
\end{tikzcd}
\end{equation}
where 
$$
    \psi= \operatorname{RH}\circ\operatorname{NAHC} 
$$ 
is the composition of the non-abelian Hodge correspondence 
$$
    \operatorname{NAHC} \colon \mathcal{M}_{\operatorname{Dol}}(\alpha) \to \mathcal{M}_{\operatorname{dR}}(\beta,\tau) 
$$
and the Riemann--Hilbert correspondence 
$$
    \operatorname{RH} \colon \mathcal{M}_{\operatorname{dR}}(\beta,\tau) \to \mathcal{M}_{\operatorname{B}}(\textbf{c})
$$
and $H$ is the Hitchin fibration over the Hitchin base $\mathcal{H}$. 
The Hitchin base will turn out to be a one-dimensional affine space in our case (see Section 2.3 for more details), and $\mathcal{H}\setminus B_R(0)$ is a neighbourhood of infinity in the base, where $R\gg 1$. 
Moreover $\lvert\mathcal{D}\partial\mathcal{M}_B\rvert$ denotes the body of the dual (or nerve) complex of the compactifying divisor of the Betti space, and $S$ is Simpson's natural map from a neighborhood of infinity in $\mathcal{M}_B$ to the body of the nerve complex (see Section 2.1 and 6.3 for more details). 

In~\cite{KatNollPandSimp} the conjecture about the existence of a commutative diagram up to homotopy was stated in higher generality. 
In particular, the homotopy type of the body of the nerve complex is expected to be always that of a sphere. 
This homotopy sphere assertion was proved by C. Simpson in~\cite{Simp4} over $\mathbb{C}P^1$, with an arbitrary finite number of punctures in the rank $2$ case. 
The investigation of the homotopy type of the topological space of the dual boundary complex of the character variety is a basic step to deal with the Geometric $P=W$ conjecture, and numerous results belong to this. 
A.~Komyo~\cite{Kom} proved the assertion for some $2$ and $4$ dimensional tame cases, and it was generalized by C.~Simpson~\cite{Simp4}, who showed that in the rank $2$ case for an arbitrary number $n$ of logarithmic points on $\mathbb{C}P^1$, the homotopy type of the dual boundary complex is that of $S^{2n-7}$. 
Another result from M. Mauri, E. Mazzon and M. Stevenson shows that the dual boundary complex of a log-Calabi-Yau compactification of the $\operatorname{GL}(n,\mathbb{C})$ character variety of a 2-torus is homeomorphic to $S^{2n-1}$, see~\cite{MauMazStev}. 

In order to prove our result, we will employ the asymptotic abelianization approach used in~\cite{SzSz2} to establish the Geometric P=W conjecture in rank 2 corresponding to the Painlev\'e cases, suitably adapted to rank 3. 
An important technical difference with the rank $2$ case is that in~\cite{SzSz},~\cite{SzSz2} the second author made use of certain local models called fiducial solutions~\cite{FMSW} that are only available in rank $2$. 
In the present article, however, we have found a way to get around finding specific model solutions around the branch points by making use of recent results of T.~Mochizuki and the second author~\cite{MocSz} for the analysis. 
This simplifies the presentation and clarifies the picture, thus paving the way for potential higher-rank and higher-dimensional generalizations of our viewpoint. 
As a consequence of the analysis, we find that the asymptotic behavior of $\psi$ depends on a decomposition into sectors of the Hitchin base, namely it is determined by different exponential terms in each sector around infinity. 
To be more precise, the absolute values of the trace coordinates are of exponential growth in the parameter, with the expression of their constant factors depending on the sector. 
Over all sectors, these constant factors are the integrals of the Liouville $1$-form along some loop on the spectral curve. As for the phases of the trace coordinates, we find that they are asymptotically equal on each sector to the holonomy of the spectral sheaf along certain loops that we determine. 
This then gives an unconditional proof of Theorem~\ref{thm:main}.  
The behaviour of the trace coordinates is an instance of the Stokes phenomenon, that we analyze in detail in this particular case.  
Unfortunately, the analysis breaks down in the Stokes directions. 
However, assuming that the same asymptotic formulas are valid in these directions too (Assumption~\ref{assn:Stokes}), we can fully understand the asymptotic behaviour of the diffeomorphism, leading to a proof of Theorem~\ref{thm:main} subject to this condition. 

This paper is organized as follows. 
In Section~\ref{sec:prep}, we recall the background material necessary to explain our arguments. 
In Section~\ref{sec:char}, we first describe the $\operatorname{GL}(3,\mathbb{C})$ (which is equivalent in this case to $\operatorname{SL}(3,\mathbb{C})$) character variety of the three-punctured sphere in general, and prove Theorem~\ref{thm:secondary}. 
We then analyze trace coordinates on the character variety, going back to classical work of R. Fricke and F. Klein \cite{FrKl}, and which were generalized for $\operatorname{GL}(3,\mathbb{C})$ character varieties by S.~Lawton~\cite{Law1},~\cite{Law2}. 
In Section~\ref{sec:Dol}, we use T.~Mochizuki's asymptotic abelianization technique~\cite{Moc} and equivariant method of T.~Mochizuki and the second author~\cite{MocSz} to give the large-scale analysis of harmonic bundles. 
In Section~\ref{sec:RH} we use T.~Mochizuki's solution of the Hitchin WKB problem on non-critical paths~\cite[Section~2.3]{Moc} to describe the parallel transport matrices, and apply Riemann--Hilbert correspondence to the previous setup. 
Finally in Section~\ref{sec:proof}, we investigate the asymptotic behaviour of the trace coordinates under the Riemann--Hilbert correspondence, analyze the emerging Stokes phenomenon. 
We prove Theorem~\ref{thm:main} in Section~\ref{sec:proof_main2} without assuming any further condition, and in Section~\ref{sec:proof_main} relying on Assumption~\ref{assn:Stokes}. 
We believe that the conditional proof of the result may be of independent interest, because it conjecturally explains the global behaviour of the diffeomorphism. 
This is the reason we include Section~\ref{sec:proof_main} too.

\begin{acknowledgement}
The second author would like to thank T.~Mochizuki for collaboration and inspiring discussions about harmonic bundles. 
\end{acknowledgement}

\section{Preparatory material}\label{sec:prep}

First, let us introduce the material that we will need to establish and prove our result. 
The structure follows more or less~\cite{SzSz} and~\cite{FMSW}. 

\subsection{Basic notations, definitions and results}

Consider $X=\mathbb{C}P^1$ with the standard Riemannian-metric, and with coordinate charts $z$ and $w=z^{-1}$. For the distinct points $0,1,\infty\in\mathbb{C}P^1$, let $D$ be the simple effective divisor $D=0+1+\infty$ (and denote by $D$ the support set of the divisor as well). Consider furthermore a smooth vector bundle $E$ of rank $3$ and degree $0$ on $\mathbb{C}P^1$. We denote by $K$ and $\mathcal{O}$ the sheaves of holomorphic $1$-forms and functions on $\mathbb{C}P^1$, and by $\Omega^{1,0}$ and $\Omega^{0,1}$ the smooth $(1,0)$- and $(0,1)$-forms on $\mathbb{C}P^1$. 
Then, a $1$-form valued $\mathcal{O}$-linear vector bundle morphism $\theta$ is called a Higgs field: 
\begin{equation*}
    \theta:E\rightarrow E\otimes \Omega^{1,0},
\end{equation*}
Moreover, we consider partial $(0,1)$-connections $\overline{\partial}_E$ on $E$ over $\mathbb{C}P^1$. Together with a Hermitian metric $h$ on $E$, the basic objects of our investigation will be Higgs bundles $(E,\theta,\overline{\partial}_E)$, which satisfy Hitchin's equations:
\begin{equation*}
    \begin{cases}
      \overline{\partial}_E\theta=0
      \\ F_h + [\theta,\theta^{\dag_h}]=0
    \end{cases}
\end{equation*}
where $F_h$ denotes the curvature form of the Chern connection $\nabla_h^+$, associated with $\overline{\partial}_E$ and $h$, and $\theta^{\dag_h}$ denotes the adjoint of the Higgs field with respect to $h$ (i.e. $\theta^{\dag_h}:E\rightarrow E\otimes \Omega^{0,1}$). If the Hitchin's equations are satisfied, then the bundle is called harmonic, and $h$ is called Hermitian-Einstein metric. We also get a holomorphic structure $\overline{\partial}_{\operatorname{det}E}$ on the complex line bundle $\operatorname{det}E$, induced by $\overline{\partial}_E$, and a Hermitian metric $h_{\operatorname{det}E}$ on it, induced by $h$. Now denote by $\mathcal{E}$ the holomorphic vector bundle $(E,\overline{\partial}_E)$ on $\mathbb{C}P^1\setminus D$. The holomorphic vector bundle $\mathcal{E}$ is also defined over $D$, and we assume that $\theta$ 
has logarithmic singularities at the points of $D$, that is we are considering logarithmic Higgs bundles.

Let us define the parabolic structure of Higgs bundles, based on~\cite{MehSesh},~\cite{BodYok} and~\cite{Mats}. Fix a weight vector for all $p\in D$: $\underline{\alpha}_P=(\alpha_P^1,\alpha_P^2,\alpha_P^3)$, where $\alpha_P^j$'s lie in a unit interval for all $j=1,2,3$, and $\alpha_P^1<\alpha_P^2<\alpha_P^3$. 
A quasi-parabolic structure on $\mathcal{E}$ with divisor $D$ is a filtration on the fiber of $\mathcal{E}$ over each point of $D$:
\begin{equation}\label{eq:filtration}
    0= l_P^3 \subset l_P^2 \subset l_P^1\subset l_P^0=\restr{\mathcal{E}}{P}.
\end{equation}
  A parabolic structure is a quasi-parabolic structure together with a choice of weight vectors $\underline{\alpha}_P$ at each $P$. 
 It gives rise to an $\mathbb{R}$-filtration $\mathcal{P}_*$ of $\mathcal{E}$. 
 We always assume that the Higgs field is weakly parabolic, meaning that $\theta:l_P^i\rightarrow l_P^{i}\otimes K(D)$ at each $P\in D$. 
 The Higgs field is called strongly parabolic if $\theta:l_P^i\rightarrow l_P^{i+1}\otimes K(D)$ at each $P\in D$, i.e. the residue is nilpotent with respect to the filtration in that the action on the graded pieces $\operatorname{gr}_l^{\boldsymbol{\cdot}}\mathcal{E}$ of $l_P^{\bullet}$ is trivial. 
 
As usual, under stability of a Higgs bundle $(E,\theta,\overline{\partial}_E)$ we mean that for any proper holomorphic subbundle $F\subset E$ which satisfies $\theta:F\rightarrow F\otimes K(D)$, the inequality $\mu(F)<\mu(E)$ holds, where $\mu(E)=\frac{\operatorname{deg}E}{\operatorname{rank}E}$ is the slope of the bundle ($\mu(F)$ defined similarly). In the parabolic setting, we speak about $\alpha$-stability, which depends on the weight vectors $\underline{\alpha}_P$, and means that for all $F$ satisfying the above conditions
\begin{equation*}
    \frac{\operatorname{pdeg}_{\underline{\alpha}}E}{\operatorname{rank}E}>\frac{\operatorname{pdeg}_{\underline{\alpha}}F}{\operatorname{rank}F},
\end{equation*}
where the parabolic degree of the parabolic bundle (and subbundle) is:
\begin{gather}
    \operatorname{pdeg}_{\underline{\alpha}}E = \operatorname{deg}E + \sum_{P\in D}\sum_{j=1}^3 \alpha_P^j \label{stab1}
    \\ \operatorname{pdeg}_{\underline{\alpha}}F = \operatorname{deg}F + \sum_{P\in D}\sum_{j=1}^3 \alpha_P^j\cdot \operatorname{dim}((\restr{F}{P}\cap l_P^{j-1})/(\restr{F}{P}\cap l_P^{j})) \label{stab2}
\end{gather}
The Higgs bundle is called $\alpha$-polystable, if it is the direct sum of lower rank $\alpha$-stable Higgs bundles, with the same parabolic slope as $(E,\theta,\overline{\partial}_E)$. By the results of Hitchin~\cite{Hit} and Simpson~\cite{Simp2}, it is known that a Higgs bundle admits a unique Hermitian-Einstein metric $h$ with $\operatorname{det}h=h_{\operatorname{det}E}$ if and only if it is polystable.

Let $\operatorname{SL}(E)$ be the principal bundle of automorphisms of $E$ which induce the identity on $\operatorname{det}E$. Then the group of complex gauge transformations, denoted by $\mathcal{G}$, is the group of sections of $\operatorname{SL}(E)$. Moreover its Lie algebra consists of the sections of $\mathfrak{sl}(E)$, the vector bundle of traceless endomorphisms of $E$. The gauge group $\mathcal{G}$ acts on the Higgs bundles via
\begin{equation*}
    g.(E,\theta,\overline{\partial}_E)=(E,g^{-1}\theta g,g^{-1}\overline{\partial}_E g), \hspace{0.5cm} \forall g\in\mathcal{G}.
\end{equation*}

\begin{defn}
    The moduli space of harmonic, $\alpha$-stable, strongly parabolic, meromorphic $\operatorname{SL}(3,\mathbb{C})$-Higgs bundles, with at most logarithmic singularities, with given weight vectors $\underline{\alpha}_P$, up to the complex gauge action, is called the Dolbeault moduli space, denoted by $\mathcal{M}_{\operatorname{Dol}}(\alpha)$.
\end{defn}
See~\cite{Konno} for a differential geometric construction of this space, and~\cite{Nitsure} for an algebraic geometric one. 

In the case of logarithmic Higgs bundles, for $\theta$ and $h$ the so called tameness condition is satisfied, that is at each $P\in D$, $h$ admits a lift along any ray to $P$, which grows at most polynomially in the standard metric. (Here we consider $h$ as an equivariant harmonic map from the universal cover of the Riemann surface to the Hermitian symmetric space $\operatorname{GL}(3,\mathbb{C})/\operatorname{U}(3)$). For such a tame, harmonic bundle $(E,\theta,\overline{\partial}_E,h)$ the connection
\begin{equation}
    \label{connection}
    \nabla=\nabla_h^+ + \theta + \theta^{\dag_h}
\end{equation}
is integrable, and $\nabla^{1,0}$ has regular singularities. Fix again for all $P\in D$ some $\underline{\beta}_P=(\beta_P^1,\beta_P^2,\beta_P^3)$ parabolic weight vectors and the $\underline{\tau}_P=(\tau_P^1,\tau_P^2,\tau_P^3)$ eigenvalues of the residue of the connection. The definition of $\beta$-stability and parabolic structure of the integrable connection is just the same as for the Higgs field (the parabolic structure of the underlying vector bundle is already given, see also \cite{Mats}). We again require that~\eqref{connection} is compatible with the filtration, that is $(\operatorname{res}_P\nabla-\tau_P^j\operatorname{id})(l_P^j)\subset l_P^{j+1}$, for all $P\in D$ and $j=0,1,2$. If the eigenvalues $\tau_P^j$ are pairwise different that this implies that $\operatorname{res}_P \nabla$ is diagonal with respect to some basis compatible with the filtration. The complex gauge group action on the space of connections is also inherited from the action on the space of Higgs bundles. 

\begin{defn}
    The moduli space of $\beta$-stable, parabolic, integrable $\operatorname{SL}(3,\mathbb{C})$-connections, with regular singularities at the punctures, with given weight vectors $\underline{\beta}_P$ and given residues $\underline{\tau}_P$ at each $P\in D$, up to the complex gauge action, is called the de Rham moduli space, denoted by $\mathcal{M}_{\operatorname{dR}}(\beta,\tau)$.
\end{defn}

The third main object of our research is the Betti moduli space, also known as character variety. 
Under the stability condition,~\eqref{connection} is an irreducible integrable connection.  
For any choice of base point $x_0 \notin D$, analytic continuation of solutions provides a representation 
$$\rho:\pi_1(\mathbb{C}P^1\setminus D, x_0 )\rightarrow \operatorname{SL}(3,\mathbb{C})
$$ 
that is well-defined up to simultaneous conjugation by elements of $\operatorname{PGL}(3,\mathbb{C})$ (corresponding to different choices of a basis of solutions at $x_0$). 
The eigenvalues, denoted by $\underline{c}_P=(c_P^1,c_P^2,c_P^3)$, of the local monodromy around $P\in D$ are determined by $(\underline{\beta}_P, \underline{\tau}_P)$.
(As a matter of fact, the local system admits a filtration and corresponding weights too, but we will not need this extra structure here.)

\begin{defn}\label{def:Betti}
    The moduli space of the above described representations is called the Betti moduli space or character variety, denoted by $\mathcal{M}_{\operatorname{B}}(\textbf{c})$.
\end{defn}

It is known that the Betti space is a smooth, affine algebraic variety for generic parameters. 
There exists a compactification of the Betti space by a simple normal crossing divisor $D_B$ (see the results of Nagata and Hironaka \cite{Nag}, \cite{Hir}). In our case $D_B$ is a complex curve. 
As customary, we define its dual complex $\mathcal{D}D_B$ as the simplicial complex whose vertices are the irreducible components of $D_B$, and whose edges corresponds to the intersections of the components. We want to apply this to the compactification of the Betti moduli space, therefore the resulting simplicial complex will be called dual boundary complex, denoted by $\mathcal{D}\partial\mathcal{M}_B(\textbf{c})$.

It is known from Simpson \cite{Simp3}
that there is a connection between the above defined parameters. With the eigenvalues of the residues of the Higgs-field being equal to 0, it simplifies to 
\begin{equation*}
    \alpha_P^i=\beta_P^i=\tau_P^i, \hspace{0.2cm} \operatorname{and}\hspace{0.2cm} c_P^i=e^{-2\pi \sqrt{-1}\alpha_P^i}, \forall P\in D,\forall i\in\{1,2,3\}.
\end{equation*}

Moreover, the following theorem holds.

\begin{theorem}\label{naht_RH}
    Assume that the parabolic degree of $\mathcal{E}$ is 0. 
    \begin{enumerate}
        \item\cite{BB}\label{naht} 
        The spaces $\mathcal{M}_{\operatorname{Dol}}(\alpha)$, $\mathcal{M}_{\operatorname{dR}}(\beta,\tau)$ and $\mathcal{M}_{\operatorname{B}}(\textbf{c})$ are $\mathbb{C}$-analytic manifolds, and there exists a diffeomorphism 
    \begin{equation*}
        \operatorname{NAHC}: \mathcal{M}_{\operatorname{Dol}}(\alpha)\rightarrow \mathcal{M}_{\operatorname{dR}}(\beta,\tau)
    \end{equation*}
    called the non-abelian Hodge correspondence. 
    \item\cite[Theorem~7.1]{IIS}\label{RH} There exists a complex bianalytic isomorphism 
    \begin{equation*}
        \operatorname{RH}: \mathcal{M}_{\operatorname{dR}}(\beta,\tau)\rightarrow \mathcal{M}_{\operatorname{B}}(\textbf{c}),
    \end{equation*}
    called the Riemann--Hilbert correspondence.
    \end{enumerate}
\end{theorem}

\subsection{Choice of parameters}\label{sec:choices}

Now let us make particular choices of the parameters introduced in the previous subsection so that they simplify some of the later computations. For all $P\in D$, set 
\begin{gather*}
    \alpha_P^1=\beta_P^1=\tau_P^1=-\frac{1}{3}
    \\ \alpha_P^2=\beta_P^2=\tau_P^2=0
     \\ \alpha_P^3=\beta_P^3=\tau_P^3=\frac{1}{3}
\end{gather*}
Consequently 
\begin{equation*}
    c_P^2=1,\hspace{0.2cm} c_P^3=-\frac{1}{2}+\frac{\sqrt{3}}{2}i=\varepsilon, \hspace{0.2cm} c_P^1=-\frac{1}{2}-\frac{\sqrt{3}}{2}i=\varepsilon^2,
\end{equation*}
where $i=\sqrt{-1}$ and $\varepsilon$ stands for a primitive cubic root of unity.

\begin{lemma}\label{lem:vanishing_char_coeffs}
    With these parameter values 
    \begin{itemize}
        \item[i)] the parabolic degree of $\mathcal{E}$ is zero,
        \item[ii)] the traces of $\theta$ and $\theta^2$ are identically zero, 
        \item[iii)] the trace of the monodromy around $P$, in any orientation, vanishes for all $P$: 
        \[
            c_P^1 + c_P^2 + c_P^3 = 0 = \frac 1{c_P^1} + \frac 1{c_P^2} + \frac 1{c_P^3}. 
        \]
    \end{itemize}
\end{lemma}

\begin{proof}
    \begin{itemize}
        \item[i)] 
    Since at all $P\in D$, the sum of the parabolic weights equal to 0, it follows from equation \eqref{stab1}, that the parabolic degree of $\mathcal{E}$ is 0.
    \item[ii)] The residues of the Higgs field are traceless at all $P\in D$.

    On $\mathbb{C}P^1$ we have for the sheaf of holomorphic 1-forms $K\cong \mathcal{O}(-2)$, and with the divisor $D=0+1+\infty$, we have $K(D)\cong\mathcal{O}(1)$, via the identification $\frac{\operatorname{d}\! z}{z(z-1)}\leftrightarrow 1$. Then
\begin{gather*}
    \operatorname{tr}\theta\in H^0(\mathbb{C}P^1,K(D))\cong H^0(\mathbb{C}P^1,\mathcal{O}(1))\cong \mathbb{C}^2
    \\ \operatorname{tr}\theta^2\in H^0(\mathbb{C}P^1,K(D)^{\otimes 2})\cong H^0(\mathbb{C}P^1,\mathcal{O}(2))\cong \mathbb{C}^3,
\end{gather*}
We have three independent vanishing conditions for both $\operatorname{tr}\theta$ and $\operatorname{tr}\theta^2$ at the points $P\in D$, therefore both must be zero globally. (One condition is even redundant in the case of $\operatorname{tr}\theta$.)
    \end{itemize}

    \item[iii)] Straightforward verification. 
\end{proof}

\subsection{The Hitchin fibration}

The characteristic coefficients of a Higgs bundle of rank $3$ over $\mathbb{C}P^1$ with logarithmic singularities at $D$ belong to the vector space 
\begin{gather*}
    \mathcal{B}=H^0(\mathbb{C}P^1,K(D))\oplus H^0(\mathbb{C}P^1,K(D)^{\otimes 2})\oplus H^0(\mathbb{C}P^1,K(D)^{\otimes 3})\cong
    \\ \cong H^0(\mathbb{C}P^1,\mathcal{O}(1))\oplus H^0(\mathbb{C}P^1,\mathcal{O}(2))\oplus H^0(\mathbb{C}P^1,\mathcal{O}(3))\cong \mathbb{C}^2\oplus\mathbb{C}^3\oplus\mathbb{C}^4
\end{gather*}
According to Lemma~\ref{lem:vanishing_char_coeffs}, with the choices made in Section~\ref{sec:choices}, the first two components vanish. 
The third characteristic coefficient of $\theta$ is $\operatorname{det}\theta\in H^0(\mathbb{C}P^1,K(D)^{\otimes 3})$. Let us use the notation $L=K(D)$, with the natural projection from the total space of $L$, $p_L:\operatorname{Tot}L\rightarrow \mathbb{C}P^1$. Define $\zeta\frac{\operatorname{d}\! z}{z(z-1)}$ to be the canonical section of $p_L^*L$ over $p_L^{-1}(\mathbb{C})$, i.e. away from the infinity section. With this, the characteristic polynomial of $\theta$ is
\begin{equation*}
    \operatorname{det}(\zeta\operatorname{id}_{\mathcal{E}}-\theta)=\zeta^3\operatorname{id}^{\otimes 3}_{\mathcal{E}}+H_{\theta}
\end{equation*}
where $H_{\theta}$ lies in the last direct summand of $\mathcal{B}$. The third coefficient $\operatorname{det}\theta=H_{\theta}$ has 4 free parameters, but with the three independent vanishing relations at $P\in D$, this reduces to a one-parameter family. One can see \cite[Appendix A]{Mats}, that it has the form
\begin{equation*}
    \operatorname{det}\theta=H_{\theta}=(tz(z-1)+p_2(z))\frac{\operatorname{d}\! z^{\otimes 3}}{z^3(z-1)^3}
\end{equation*}
where $p_2(z)=az^2+bz+c$ is a quadratic polynomial, and $a,b,c,t$ are the 4 parameters, from which $a,b,c$ are fixed to be $0$, and $t\in\mathbb{C}$ is the only free parameter. The 1-dimensional subspace 
\begin{equation}\label{eq:Hitchin_base}
      \mathcal{H} = \left\{ t \frac{\operatorname{d}\! z^{\otimes 3}}{z^2(z-1)^2}  \colon t\in \mathbb{C} \right\} \subset\mathcal{B}
\end{equation}
where $\operatorname{det}\theta$ may take its values is called the Hitchin base of $\mathcal{M}_{\operatorname{Dol}}(\alpha)$, and the map 
\begin{equation*}
    H:\mathcal{M}_{\operatorname{Dol}}(\alpha)\rightarrow \mathcal{H},\hspace{0.3cm} (\mathcal{E},\theta)\mapsto \operatorname{det}\theta
\end{equation*}
is called the Hitchin fibration. 
Clearly, for any $\tau \in \mathbb{C}^{\times}$ and $(\mathcal{E},\theta)$ we have 
\[
    H((\mathcal{E},\tau \theta_1)) = \tau^3 H((\mathcal{E}, \theta)). 
\]
We choose the preferred point 
\begin{equation}\label{eq:q1}
        q_1 = \frac{\operatorname{d}\! z^{\otimes 3}}{z^2(z-1)^2}\in \mathcal{H}. 
\end{equation}
For every $t\in\mathbb{C}$ the smooth curve called the spectral curve is defined via 
\begin{equation}
    \label{spectral}
    \Sigma_t=\{(z,\zeta)|\zeta^3+tz(z-1)=0\}\subset \operatorname{Tot}L.
\end{equation}
Notice that $\Sigma_t$ has maximal ramification over $0,1$ and a smooth compactification, denoted by $\widetilde{\Sigma}_t$, in $\operatorname{Tot}L$ at $z=\infty$ with $(w,\zeta)=(0,0)$, where it also admits cyclic ramification. One can easily compute from Riemann--Hurwitz formula that its genus is equal to 1. 

\subsection{Ramification of the spectral curve and the Jacobian}

The equation in (\ref{spectral}) defining $\widetilde{\Sigma}_t$ has three roots over every point of $\mathbb{C}P^1\setminus D$:
\begin{gather*}
    \xi_1=R^{1/3}e^{i\varphi/3}z^{1/3}(z-1)^{1/3},\hspace{0.3cm} \xi_2=\varepsilon R^{1/3}e^{i\varphi/3}z^{1/3}(z-1)^{1/3}, 
    \\ \xi_3=\varepsilon^2 R^{1/3}e^{i\varphi/3}z^{1/3}(z-1)^{1/3},
\end{gather*}
where we recall that $\varepsilon$ is a cubic root of unity and we switched to polar coordinates $t=Re^{i\varphi}$. That is, $\xi_{1,2,3}$ is a three-valued holomorphic function on $\mathbb{C}P^1\setminus D$, and $p_L$ induces a projection map $p_{R,\varphi}:\widetilde{\Sigma}_{R,\varphi}\rightarrow\mathbb{C}P^1$. This is indeed a ramified triple cover over $\mathbb{C}P^1$, with ramification points $z=0,z=1$ on chart $z$, and $w=0$ on chart $w$, independently of the value of $t>0$. Thus we introduce the ramification divisor $\Delta=\{0,1,\infty\}=D$, and denote the lift of the ramification divisor by $\widetilde{\Delta}$ (or $\widetilde{D}$), whose points are the branch points on $\widetilde{\Sigma}_{R,\varphi}$.
We summarize properties of the spectral curve. 
\begin{prop}
    $\widetilde{\Sigma}_{R,\varphi}$ is a smooth genus 1 curve, with ramification index at all 3 points of $D$ equal to 3, i.e. all its ramifications are cyclic. 
\end{prop}

We will use the polar coordinates  
\begin{equation}\label{eq:polar}
        \tau^3  = t = R e^{i\varphi}. 
\end{equation}
For any $\theta_1$ satisfying $H(\theta_1 ) = q_1$ let us use the notation for the cubic meromorphic differentials, with double poles at the punctures 
\begin{equation}\label{eq:qt}
t q_1 = R e^{i\varphi} q_1 :=\operatorname{det}(\tau \theta_1 ) = R e^{i\varphi}\frac{\operatorname{d}\! z^{\otimes 3}}{z^2(z-1)^2} \in H^0(\mathbb{C}P^1,K^{\otimes 3}(2D)). 
\end{equation}
Viewed as a section of $L^{\otimes 3}$, $q_1$ has a simple zero. 
Then the sheets of $\widetilde{\Sigma}_{R,\varphi}$ are just the cubic roots of $t q_1$: 
\begin{align}
    Q_{1,R,\varphi} & = \tau \frac{\operatorname{d}\! z}{z^{2/3}(z-1)^{2/3}} = R^{1/3}e^{i\varphi/3}\frac{\operatorname{d}\! z}{z^{2/3}(z-1)^{2/3}} \notag \\
    Q_{2,R,\varphi} & =\varepsilon \tau \frac{\operatorname{d}\! z}{z^{2/3}(z-1)^{2/3}} = \varepsilon R^{1/3}e^{i\varphi/3}\frac{\operatorname{d}\! z}{z^{2/3}(z-1)^{2/3}}, \label{Q}
    \\ 
    Q_{3,R,\varphi} & =\varepsilon^2 \tau \frac{\operatorname{d}\! z}{z^{2/3}(z-1)^{2/3}} = \varepsilon^2 R^{1/3}e^{i\varphi/3}\frac{\operatorname{d}\! z}{z^{2/3}(z-1)^{2/3}}. \notag
\end{align}
According to~\cite{BNR}, the three corresponding eigenspaces determine a line bundle $\mathcal{L}_{R,\varphi}\rightarrow\widetilde{\Sigma}_{R,\varphi}$, whose pushforward $(p_{R,\varphi})_*\mathcal{L}_{R,\varphi}$ is isomorphic to $\mathcal{E}$. The degree of $\mathcal{L}_{R,\varphi}$ can be computed via \cite[2.3.3]{LM} or~\cite[Theorem~5.4]{Sz_BNR}:
\begin{equation*}
    \operatorname{deg}\mathcal{L}=\operatorname{deg}\mathcal{E}+r(1-r)\left( 1-g-\frac{n}{2} \right) =3, 
\end{equation*}
where $r=3$ is the rank, $n=3$ is the number of ramification points, and $g=0$ is the genus of $\mathbb{C}P^1$. Thus an element of the Dolbeault moduli space determines a spectral curve and a line bundle of degree 3 on it, and vice versa. So, the fiber of the Hitchin fibration over the point parameterized by a fixed $t\in \mathcal{H}$ is $\operatorname{Pic}^3(\widetilde{\Sigma}_t)$, that is a 2-torus, namely a torsor over $\operatorname{Jac}(\widetilde{\Sigma}_t)$. In particular, the Hitchin fibration is an elliptic fibration. 

Let us discuss this correspondence between the Hitchin fibers and the Jacobian of the spectral curve a bit more in detail. Consider the following period lattice $\Lambda_t\subset H^{0,1}(\widetilde{\Sigma}_t)\cong \mathbb{C}$, provided by the image $\operatorname{Im}(p^{0,1}\circ \iota)$, where
\begin{gather*}
    \iota:H^1(\widetilde{\Sigma}_t,2\pi i\mathbb{Z})\rightarrow H^1(\widetilde{\Sigma}_t,\mathbb{C})
    \\ p^{0,1}:H^1(\widetilde{\Sigma}_t,\mathbb{C})\rightarrow H^{0,1}(\widetilde{\Sigma}_t)
\end{gather*}
are the coefficient inclusion on the first cohomology class, and the projection of harmonic forms to their antiholomorphic part respectively. There exists a $\mathbb{C}$-analytic isomorphism $\operatorname{Jac}(\widetilde{\Sigma}_t)\cong  H^{0,1}(\widetilde{\Sigma}_t)/\Lambda_t$. Namely, any class in $H^{0,1}(\widetilde{\Sigma}_t)/\Lambda_t$ can be represented by a $\mu\in\Omega^{0,1}(\widetilde{\Sigma}_t)$, because of the abelian Hodge correspondence. Then the connection form $u=\mu-\overline{\mu}=2i\operatorname{Im}\mu\in\Omega^1(\widetilde{\Sigma}_t)$ defines a flat $U(1)$-connection.
For the line bundle $\mathcal{L}\in \operatorname{Jac}(\widetilde{\Sigma}_t)$ given by $\mu$ and a $1$-cycle $X\in H_1(\widetilde{\Sigma}_t,\mathbb{Z})$ we call  
\begin{equation}\label{eq:holonomy}
    \operatorname{hol}_X (\mathcal{L} ) = e^{\oint_X u}    
\end{equation}
the holonomy of $\mathcal{L}$ along $X$. 
Fix $X,Y$ 1-cycles generating $H_1(\widetilde{\Sigma}_t,\mathbb{Z})$. 
The abelian version of Theorem~\ref{naht} can then be expressed as saying that the holonomy map 
\begin{equation*}
    \operatorname{hol}\colon \operatorname{Jac}(\widetilde{\Sigma}_t)\rightarrow T^2=S^1\times S^1, \hspace{0.3cm} \mu\mapsto (\operatorname{hol}_X (\mathcal{L} ), \operatorname{hol}_Y (\mathcal{L} )) = \left( e^{\oint_X u},e^{\oint_Y u} \right), 
\end{equation*}
is an isomorphism between the Jacobian and the 2-torus. 
See \cite[Section 4]{GX} for more details. 

\subsection{The Hitchin section}

There is a preferred line bundle $\mathcal{L}_0$ over $\widetilde{\Sigma}_t$ giving rise to a section of $H$ analogous to the Hitchin section.
Namely, as it is well-known, 
\[
(p_L)_* \mathcal{O}_{\widetilde{\Sigma}_t} \cong \mathcal{O}_{\mathbb{C} P^1} \oplus K_{\mathbb{C} P^1}(D)^{-1} \oplus K_{\mathbb{C} P^1}(D)^{-2}\cong \mathcal{O}_{\mathbb{C} P^1}\oplus\mathcal{O}_{\mathbb{C} P^1}(-1)\oplus\mathcal{O}_{\mathbb{C} P^1}(-2), 
\]
the direct summands being generated by $1,\zeta , \zeta^2$ respectively. 
The preferred choice of spectral sheaf is then $\mathcal{L}_0 =p_L^*L\otimes\mathcal{O}_{\widetilde{\Sigma}_t}$. 
Notice that $\mathcal{O}_{\widetilde{\Sigma}_t}\cong K_{\widetilde{\Sigma}_t}$, because $\widetilde{\Sigma}_t$ is an elliptic curve. 
Moreover, from the local form at $z=0$ of the equation defining $\widetilde{\Sigma}_t$ we get 
\[
    \frac{\operatorname{d}\! z}z = h(\zeta ) \frac{\operatorname{d}\! \zeta }{\zeta }
\]
for some local holomorphic function $h$ with $h(0)\neq 0$, and similarly for the other poles $P\in D$. 
This means that 
\[
    p_L^* K_{\mathbb{C} P^1} (D )\otimes\mathcal{O}_{\widetilde{\Sigma}_t} = K_{\widetilde{\Sigma}_t} (\widetilde{D}). 
\]
We infer 
\begin{equation*}    
\mathcal{L}_0 =p_L^*L\otimes\mathcal{O}_{\widetilde{\Sigma}_t}\cong K_{\widetilde{\Sigma}_t}(\widetilde{D}) \cong \mathcal{O}_{\widetilde{\Sigma}_t}(\widetilde{D}).
\end{equation*}
By the projection formula we then have 
\[
\mathcal{E}_0 = (p_L)_* \mathcal{L}_0 \cong K(D)\oplus\mathcal{O}\oplus K(D)^{-1}, 
\] 
in particular, the degree of $\mathcal{E}_0$ is equal to zero, as required. 
Stable Higgs fields 
\begin{equation}
    \label{E}  
    \\ \theta_t: K(D)\oplus\mathcal{O}\oplus K(D)^{-1} \rightarrow K(D)^2\oplus K(D)\oplus \mathcal{O} 
    \end{equation}
over $\mathcal{E}_0$  are of the form 
\begin{equation}
\label{companion}
    \theta_t=
    \begin{pmatrix}
        0 & 0 & t q_1 \\
        1 & 0 & 0\\
        0 & 1 & 0
    \end{pmatrix}
\end{equation}
where $q_1$ is given in~\eqref{eq:q1}. 
Applying a constant (i.e., depending only on $t$) gauge transformation 
\[
\begin{pmatrix}
    t^{-\frac 13} & 0 & 0 \\
    0 & 1 & 0 \\
    0 & 0 & t^{\frac 13}
\end{pmatrix}
= 
\begin{pmatrix}
    \tau^{-1} & 0 & 0 \\
    0 & 1 & 0 \\
    0 & 0 & \tau
\end{pmatrix}
\]
the Higgs field gets transformed into 
\[
\tau \theta_1 =  \tau 
    \begin{pmatrix}
        0 & 0 & q_1 \\
        1 & 0 & 0\\
        0 & 1 & 0
    \end{pmatrix}.
\]
Since $q_1$ has a zero viewed as a section of $L^{\otimes 3}$, we get that 
\begin{equation*}
    \operatorname{res}_P(\tau \theta_1 )=
    \begin{pmatrix}
        0 & 0 & 0\\
        1 & 0 & 0\\
        0 & 1 & 0
    \end{pmatrix}.
\end{equation*}
The strongly parabolic condition is that the residues of $\theta_t$ at the points of $D$ are nilpotent with respect to the quasi-parabolic filtration. 
Now, there exists a unique filtration of $\mathcal{E}_0$ over the points of $D$ compatible with the above residue, namely 
\begin{equation*}
    l^2_P=\mathbb{C}\cdot\zeta^2,\hspace{0.3cm} l^1_P=\mathbb{C}\cdot\zeta\oplus \mathbb{C}\cdot\zeta^2
\end{equation*}
 The parabolic weights corresponding to the generators $1,\zeta , \zeta^2$ are therefore respectively equal to $-\frac{1}{3}, 0, \frac{1}{3}$. 
 See also the results from \cite{FredNei}, \cite{HallOuPed}.

\section{Description of the Betti moduli space}\label{sec:char}

In this chapter we consider the Betti space (or character variety), parameterizing the irreducible representations of the fundamental group of $\mathbb{C}P^1\setminus D$ in $\operatorname{GL}(3,\mathbb{C})$, up to the simultaneous conjugation by $\operatorname{PGL}(3,\mathbb{C})$. 
We pick $z_0\in \mathbb{C}P^1\setminus \{0,1,\infty\}$ once and for all, and all occurrences of fundamental group will mean with base point $z_0$.  
Since the fundamental group of $\mathbb{C}P^1\setminus \{0,1,\infty\}$ is isomorphic to the free group generated by two elements, this amounts to considering maps
\begin{equation*}
    \rho:\pi_1\left(\mathbb{C}P^1\setminus \{0,1,\infty\}\right)\cong\langle a,b \rangle\rightarrow \operatorname{GL}(3,\mathbb{C})
\end{equation*}
under the constraint that the eigenvalues of $\rho(a),\rho(b)$ and $\rho(ab)$ are fixed: $\{\lambda_1,\lambda_2,\lambda_3\}$, $\{\mu_1,\mu_2,\mu_3\}$ and $\{\nu_1,\nu_2,\nu_3\}$ respectively (previously denoted by the vectors $\underline{c}_P$). 
Because of the $\operatorname{PGL}(3,\mathbb{C})$ action, we have the freedom to choose $\rho(a)$ to be diagonal with elements $\{\lambda_1,\lambda_2,\lambda_3\}$. Having achieved this, there remains the action of the maximal torus $\left(\mathbb{C}^{\times}\right)^2$ of $\operatorname{PGL}(3,\mathbb{C})$. 
The action of $(t_1,t_2)\in\left(\mathbb{C}^{\times}\right)^2$ on a matrix $B=[b_{ij}]$ is the standard conjugation 
\begin{equation*}
    (t_1,t_2).B=
    \begin{pmatrix}
        b_{11} & t_1b_{12} & t_1t_2b_{13}\\
        t_1^{-1}b_{21} & b_{22} & t_2b_{23}\\
        t_1^{-1}t_2^{-1}b_{31} & t_2^{-1}b_{32} & b_{33}
    \end{pmatrix}.
\end{equation*}

With the notations $\rho(a)=A,\rho(b)=B$, the constraints on the eigenvalues are equivalent to constraints on the traces of $A^j,B^j$ and $(AB)^j$ for $j=1,2,3$. Thus the Betti space can be written as

\begin{gather*}
    \mathcal{M}_B=\{A,B\in\operatorname{GL}(3,\mathbb{C})|A=\operatorname{Diag}[\lambda_1,\lambda_2,\lambda_3],B=[b_{i,j}],
    \\ \operatorname{tr}(B^j)=\sigma_j(\underline{\mu}), \operatorname{tr}((AB)^j)=\sigma_j(\underline{\nu}),j=1,2,3\}/\left(\mathbb{C}^{\times}\right)^2
\end{gather*}
where $\sigma_j$ is the degree $j$ homogeneous symmetric polynomial in 3 variables. 
Because of the irreducibility of the representations, we are given that both $b_{21}$ and $b_{31}$ can not vanish simultaneously. 
Possibly passing to a Zariski open subset (that does not alter validity of our arguments), we may assume that they are both nonzero. 
We may then use the $\left(\mathbb{C}^{\times}\right)^2$ action to remove these coefficients. 
Since the assumption on $(AB)^3$ is redundant, this gives 
\begin{gather*}
    \mathcal{M}_B=\{ B=
    \begin{pmatrix}
        b_{11} & b_{12} & b_{13}\\
        1 & b_{22} & b_{23}\\
        1 & b_{32} & b_{33}
    \end{pmatrix}
    |\operatorname{tr}(B^j)=\sigma_j(\underline{\mu}), j=1,2,3, 
    \\ \operatorname{tr}((AB)^j)=\sigma_j(\underline{\nu}),j=1,2 \}.
\end{gather*}
Now, the conditions $\operatorname{tr}(B)=b_{11}+b_{22}+b_{33}=\sigma_1(\underline{\mu})$, and $\operatorname{tr}(AB)=\lambda_1b_{11}+\lambda_2b_{22}+\lambda_3b_{33}=\sigma_1(\underline{\nu})$  can be used to express $b_{22}$ and $b_{33}$ in terms of $b_{11}$:
\begin{gather*}
    b_{22}=\frac{\lambda_3-\lambda_1}{\lambda_2-\lambda_3}b_{11}+c_1(\underline{\lambda},\underline{\mu},\underline{\nu})=:Q(b_{11})
    \\ b_{33}=\frac{\lambda_1-\lambda_2}{\lambda_2-\lambda_3}b_{11}+c_2(\underline{\lambda},\underline{\mu},\underline{\nu})=:P(b_{11}),
\end{gather*}
where $c_1,c_2$ are constants depending only on $\underline{\lambda},\underline{\mu},\underline{\nu}$, while $P$ and $Q$ are degree 1 polynomials in $b_{11}$. Switching to the notation
\begin{gather*}
    B=
    \begin{pmatrix}
        b_{11} & b_{12} & b_{13}\\
        1 & b_{22} & b_{23}\\
        1 & b_{32} & b_{33}
    \end{pmatrix} =
    \begin{pmatrix}
        X & Y & Z\\
        1 & Q(X) & V\\
        1 & W & P(X)
    \end{pmatrix}
\end{gather*}
implies 
\begin{gather*}
    \mathcal{M}_B=\{(X,Y,Z,V,W)\in\mathbb{C}^5|\operatorname{tr}(B^2)=\sigma_2(\underline{\mu}), \operatorname{tr}(B^3)=\sigma_3(\underline{\mu}),\operatorname{tr}((AB)^2)=\sigma_2(\underline{\nu})\}.
\end{gather*}
The remaining three conditions read as:
\begin{gather*}
    \operatorname{tr}(B^2)=X^2+2Y+2Z+Q^2(X)+P^2(X)+2VW=\sigma_2(\underline{\mu})
    \\ \operatorname{tr}((AB)^2)=\lambda_1^2X^2+2\lambda_1\lambda_2Y+2\lambda_1\lambda_3Z+\lambda_2^2Q^2(X)+\lambda_3^2P^2(X)+2\lambda_2\lambda_3VW=\sigma_2(\underline{\nu})
    \\ \operatorname{tr}(B^3)=X^3+P^3(X)+Q^3(X)+3ZW+3XZ+3YV+3XY
    \\ +3Q(X)VW+3Q(X)Y+3P(X)VW+3P(X)Z=\sigma_3(\underline{\mu})
\end{gather*}
Now, the first two equations allow us to express $Y$ and $Z$, and eliminate these variables from the third equation.
We thus obtain a description of the Betti space as a cubic surface in $\operatorname{Spec}\mathbb{C}[X,V,W]$. 
Then, we can consider the homogenisation of the resulting equation.
This procedure provides us the compactifying curve of $\mathcal{M}_B$ as a homogeneous cubic curve in $\mathbb{C} P^2$, with equation\footnote{The correctness of the algebraic derivations was checked by symbolic computation using Python's SymPy library.} 
\begin{equation}
  \begin{gathered}
  \label{divisor}
    \frac{-3(\lambda_1-\lambda_2)^2(\lambda_1-\lambda_3)^2(\lambda_2+\lambda_3)}{\lambda_1(\lambda_2-\lambda_3)^4}X^3 + \frac{3(\lambda_1-\lambda_3)^2(\lambda_1\lambda_3-\lambda_2^2)}{\lambda_1(\lambda_2-\lambda_3)^3}X^2V 
    \\ +\frac{3(\lambda_1-\lambda_2)^2(\lambda_3^2-\lambda_1\lambda_2)}{\lambda_1(\lambda_2-\lambda_3)^3}X^2W + \frac{-3(\lambda_1(\lambda_2-\lambda_3)^2+\lambda_2(\lambda_1-\lambda_3)^2+\lambda_3(\lambda_1-\lambda_2)^2)}{\lambda_1(\lambda_2-\lambda_3)^2}XVW 
    \\ + \frac{3\lambda_2(\lambda_1-\lambda_3)}{\lambda_1(\lambda_3-\lambda_2)}VW^2+\frac{3\lambda_3(\lambda_1-\lambda_2)}{\lambda_1(\lambda_2-\lambda_3)}V^2W=0
  \end{gathered}
\end{equation}

\subsection{Topology of the compactifying divisor}

Although we have made special choices for the parameters $\underline{\lambda},\underline{\mu},\underline{\nu}$, here we will see that up to homeomorphism, the compactifying curve is the same as with general choices. Namely, it is the genus $0$ curve with nodal singularity, also called fishtail, and denoted by $I_1$ in Kodaira's list of singular elliptic curves \cite{Kod}.
\begin{prop}
    The curve $C$ determined by equation (\ref{divisor}) in $\mathbb{C}P^2=\{[X:V:W]\}$, is of type $I_1$.
\end{prop}

\begin{proof}
    The proof has a similar idea as Proposition 2.3.3 in \cite{GSt}. Consider the pencil of projective lines passing through the point $P=\left[\frac{\lambda_2-\lambda_3}{\lambda_3-\lambda_1}:\frac{\lambda_1-\lambda_2}{\lambda_3-\lambda_1}:1\right]$ in $\mathbb{C}P^2$, and parameterize them by $[t_0:t_1]\in\mathbb{C}P^1$, via
    \begin{equation*}
        L_{[t_0:t_1]}=\{[X:V:W]\in\mathbb{C}P^2|t_0\left(X-\frac{\lambda_2-\lambda_3}{\lambda_3-\lambda_1}\right)=t_1\left(V-\frac{\lambda_1-\lambda_2}{\lambda_3-\lambda_1}\right)\}.
    \end{equation*}
    Determine the intersection points of line $L_{[t_0:t_1]}$  with $C$: if ${[t_0:t_1]}=[1:0]$, then $X=\frac{\lambda_2-\lambda_3}{\lambda_3-\lambda_1}$, and substituting this into (\ref{divisor}):
    \begin{equation*}
  \begin{gathered}
    \frac{-(\lambda_1-\lambda_2)^2(\lambda_2+\lambda_3)}{(\lambda_3-\lambda_1)(\lambda_2-\lambda_3)} + \frac{(\lambda_1\lambda_3-\lambda_2^2)}{\lambda_2-\lambda_3}V +\frac{(\lambda_1-\lambda_2)^2(\lambda_3^2-\lambda_1\lambda_2)}{(\lambda_3-\lambda_1)^2(\lambda_2-\lambda_3)}W  
    \\ +\frac{-(\lambda_1(\lambda_2-\lambda_3)^2+\lambda_2(\lambda_1-\lambda_3)^2+\lambda_3(\lambda_1-\lambda_2)^2)}{(\lambda_3-\lambda_1)(\lambda_2-\lambda_3)}VW 
    \\ + \frac{\lambda_2(\lambda_1-\lambda_3)}{(\lambda_3-\lambda_2)}VW^2+\frac{\lambda_3(\lambda_1-\lambda_2)}{(\lambda_2-\lambda_3)}V^2W=0
  \end{gathered}
\end{equation*}
    Here, if $W=0$, then $\frac{(\lambda_1\lambda_3-\lambda_2^2)}{\lambda_2-\lambda_3}V=\frac{(\lambda_1-\lambda_2)^2(\lambda_2+\lambda_3)}{(\lambda_3-\lambda_1)(\lambda_2-\lambda_3)}$ has a unique solution for $V$. If $W\neq 0$, then we can choose $W=1$, and the equation simplifies:
    \begin{equation*}
        V^2-\frac{2(\lambda_1-\lambda_2)}{\lambda_3-\lambda_1}V+\frac{(\lambda_1-\lambda_2)^2}{(\lambda_3-\lambda_1)^2}=0,
    \end{equation*}
    which is a complete square, therefore has a unique solution for $V$. That is, $L_{[1:0]}$ and $C$ have two intersection points:
    \begin{equation*}
        P=\left[\frac{\lambda_2-\lambda_3}{\lambda_3-\lambda_1}:\frac{\lambda_1-\lambda_2}{\lambda_3-\lambda_1}:1\right], \hspace{0.3cm} Q=\left[\frac{\lambda_2-\lambda_3}{\lambda_3-\lambda_1}:\frac{(\lambda_1-\lambda_2)^2(\lambda_2+\lambda_3)}{(\lambda_3-\lambda_1)(\lambda_1\lambda_3-\lambda_2^2)}:0\right]
    \end{equation*}
    The other case, if $t_1\neq 0$ (assume $t_1=1$), then $V=t_0\left(X-\frac{\lambda_2-\lambda_3}{\lambda_3-\lambda_1}\right)+\frac{\lambda_1-\lambda_2}{\lambda_3-\lambda_1}$. If $W\neq 0$ ($W=1$), then substitute the above equation for $V$ into (\ref{divisor})
    \begin{gather*}
        \frac{1}{\lambda_2-\lambda_3}\left(X-\frac{\lambda_2-\lambda_3}{\lambda_3-\lambda_1}\right)^2\left(X\left(\frac{(\lambda_1-\lambda_3)^2(\lambda_1\lambda_3-\lambda_2^2)}{(\lambda_2-\lambda_3)^2}t_0 \right.\right.
        \\ \left.+\frac{-(\lambda_1-\lambda_2)^2(\lambda_1-\lambda_3)^2(\lambda_2+\lambda_3)}{(\lambda_2-\lambda_3)^3)}\right)+\lambda_3(\lambda_1-\lambda_2)t_0^2
        \\ +\frac{-\lambda_1^2\lambda_2-2\lambda_1^2\lambda_3+6\lambda_1\lambda_2\lambda_3-2\lambda_2^2\lambda_3-\lambda_2\lambda_3^2}{\lambda_2-\lambda_3}t_0 
        \\ \left. + \frac{(\lambda_1-\lambda_2)(\lambda_1^2\lambda_2+\lambda_1^2\lambda_3-4\lambda_1\lambda_2\lambda_3+\lambda_2^2\lambda_3+\lambda_2\lambda_3^2)}{(\lambda_2-\lambda_3)^2} \right)=0
    \end{gather*}
    This has solution $X=\frac{\lambda_2-\lambda_3}{\lambda_3-\lambda_1}$, which provides $P$. The remaining factor is linear in $X$, so for fixed $t_0$, it has a unique solution for $X$, except if the coefficient of $X$ is $0$, i.e.: $t_0'=\frac{(\lambda_1-\lambda_2)^2(\lambda_2+\lambda_3)}{(\lambda_2-\lambda_3)(\lambda_1\lambda_3-\lambda_2^2)}$, this case the only intersection point of $C$ and $L_{[t_0:t_1]}$ is $P$. In other cases, there are one more intersection point besides $P$, except when the root of the linear factor provides $P$ again. This happens if the following quadratic equation satisfies for $t_0$:
    \begin{gather*}
        \lambda_3(\lambda_1-\lambda_2)t_0^2 + \frac{(\lambda_1-\lambda_2)(-\lambda_1\lambda_2-3\lambda_1\lambda_3+3\lambda_2\lambda_3+\lambda_3^2)}{\lambda_2-\lambda_3}t_0
        \\ + \frac{(\lambda_1-\lambda_2)(2\lambda_1^2\lambda_2+2\lambda_1^2\lambda_3-\lambda_1\lambda_2^2-6\lambda_1\lambda_2\lambda_3-\lambda_1\lambda_3^2+2\lambda_2^2\lambda_3+2\lambda_2\lambda_3^2)}{(\lambda_2-\lambda_3)^2}=0
    \end{gather*}
    This has solutions $t_0^+=\frac{2\lambda_1-\lambda_2-\lambda_3}{\lambda_2-\lambda_3}$ and $t_0^-=\frac{\lambda_1\lambda_2+\lambda_1\lambda_3-2\lambda_2\lambda_3}{\lambda_3(\lambda_2-\lambda_3)}$. One can check that for $t_0^+$, the line $L_{[t_0^+:1]}$ passes through $[0:1:0]$, so it has two intersection points with $C$. But $C$ has only two common points with the line $W=0$, namely $[0:1:0]$ and $Q$, thus $L_{[t_0^-:1]}$ intersects $C$ only at $P$. We deduce that there are exactly two parameters $[t_0':1]$ and $[t_0^+:1]$ for which $L_{[t_0:t_1]}$ has only one intersection point with $C$, namely $P$. This means that the map $\mathbb{C}P^1\rightarrow C$ sending $[t_0:t_1]$ to the other intersection of $L_{[t_0:t_1]}$ and $C$ (besides $P$) is one-to-one, except for $[t_0':1]$ and $[t_0^+:1]$. That is, $C$ is homeomorphic with a $\mathbb{C}P^1$ with $[t_0':1]$ and $[t_0^+:1]$ identified. 
\end{proof}

\subsection{The trace coordinates}\label{sec:trace_coordinates}

We will use the so called trace coordinates on the Betti moduli space, introduced by Lawton \cite{Law1}, \cite{Law2}. Let $\rho$ be at the SL(3,$\mathbb{C}$) character variety of a rank 2 free group $\langle a,b \rangle$, and consider the character map $\mathcal{M}_B\rightarrow\mathbb{C}^9$: 
\begin{equation}
    \begin{gathered}
    \label{trace}
    \rho\mapsto \left(\operatorname{tr}(\rho(a)),\operatorname{tr}(\rho(b)),\operatorname{tr}(\rho(a)\rho(b)),\operatorname{tr}(\rho(a)^{-1}),\operatorname{tr}(\rho(b)^{-1}),\operatorname{tr}((\rho(a)\rho(b))^{-1}), \right.
    \\ \left. \operatorname{tr}(\rho(a)\rho(b)^{-1}),\operatorname{tr}(\rho(a)^{-1}\rho(b)),\operatorname{tr}(\rho(a)\rho(b)\rho(a)^{-1}\rho(b)^{-1}) \right)=:(x_1,x_2,...,x_9).
    \end{gathered}
\end{equation}
In this way, the above map gives coordinates on $\mathcal{M}_B$, under the condition $x_9^2-p(\underline{x})x_9+q(\underline{x})=0$, where $p$ and $q$ are two polynomials in the variables $\{x_i\}_{i=1}^9$, see \cite[Section 4]{Law1}. On the three-punctured sphere, where $\gamma_1,\gamma_2,\gamma_3$ are simple loops around the punctures ($0,1,\infty$), with a common base point, $[\gamma_1]$ and $[\gamma_2]$ generate the fundamental group of the curve, while $[\gamma_3]=([\gamma_1][\gamma_2])^{-1}$. With our special choice of eigenvalues for $\rho([\gamma_1]), \rho([\gamma_2])$, described in Section 2.2: 
\begin{prop}\label{lem:vanishing_coordinates}
We have $x_1=x_2=x_3=x_4=x_5=x_6=0$. 
\end{prop} 
\begin{proof}
According to Lemma~\ref{lem:vanishing_char_coeffs}, the trace of the monodromy around the punctures vanishes in either orientation. 
Now, the coordinate $x_1$ is just the trace of the monodromy along a positively orientated around loop around $0$, and $x_4$ is the trace of the monodromy along a negatively orientated around loop around $0$. 
Similarly, $x_2$ is just the trace of the monodromy along a positively orientated around loop around $1$, and $x_5$ is the trace of the monodromy along a negatively orientated around loop around $1$.
Finally, $x_3$ is just the trace of the monodromy along a negatively orientated around loop around $\infty$, and $x_6$ is the trace of the monodromy along a positively orientated around loop around $\infty$. 
\end{proof}
We see that $x_7,x_8,x_9$ are the only nonzero coordinates.
Indeed, the above character map gives a $\mathcal{M}_B(\textbf{c})\rightarrow\mathbb{C}^3$ morphism, and the polynomials simplify to $p(\underline{x})=x_7x_8-3$, and $q(\underline{x})=9-6x_7x_8+x_7^3+x_8^3$. Thus the condition $x_9^2-p(\underline{x})x_9+q(\underline{x})=0$ reads as
\begin{equation*}
    0=x_9^2-x_7x_8x_9+3x_9+9-6x_7x_8+x_7^3+x_8^3.
\end{equation*}
The homogenisation of this will again provide the equation of the curve at infinity: $0=-x_7x_8x_9+x_7^3+x_8^3$.
\begin{lemma}\label{lem:compactifying_curve}
    This curve $0=-x_7x_8x_9+x_7^3+x_8^3$ is also of type $I_1$.
\end{lemma}
\begin{proof}
    One can easily see that $(x_7,x_8)=(0,0)$ is its only singular point. We may set near that point $x_9=1$, and the equation becomes $x_7x_8\approx 0$ up to terms of degree $3$. 
    This shows that the singularity is a node. Therefore the singular cubic curve must be a singular elliptic curve, and we can apply Kodaira's classification, which shows that an elliptic curve with one nodal singularity must be an $I_1$ curve.
\end{proof}

\begin{figure}[ht]
\centering
\includegraphics[width=7.0cm]{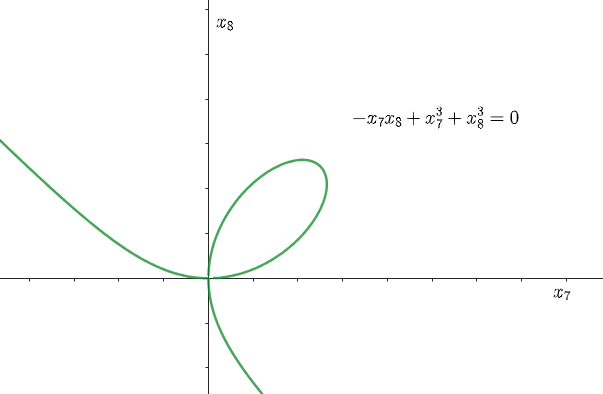}
\caption{The $-x_7x_8+x_7^3+x_8^3=0$ nodal (fishtail) curve.}
\end{figure}

\begin{corollary}
    The fundamental group of the body of the dual boundary complex of the compactifying divisor $D_B$ is cyclic, so  $\lvert \mathcal{D}\partial\mathcal{M}_B(\textbf{c})\rvert$ is of homotopy type $S^1$.
\end{corollary}

Notice that we used two different types of coordinates, but the two pictures coincide, as expected: for our special choice of parameters, the compactifying divisor of the Betti space is topologically the same as in the general description. According to Simpson \cite{Simp1}, the Riemann--Hilbert correspondence provides equivalence between vector bundles with integrable connection and representations of the fundamental group, via the monodromies of the connection around the punctures. Therefore our aim is to apply the trace coordinates to the monodromies of the connection, and investigate their asymptotic behaviour. 

\section{Asymptotic analysis of non-abelian Hodge correspondence}\label{sec:Dol}

Consider a Higgs bundle $(\mathcal{E},\theta_1 )\in \mathcal{M}_{\operatorname{Dol}}(\alpha)$ such that $H(\mathcal{E},\theta_1)=q_1$ where $q_1$ is defined in~\eqref{eq:q1}. 
For any nonzero complex number $\tau$, the Higgs bundle $(\mathcal{E},\overline{\partial}_E,\tau \theta_1)$ is also stable with the same parabolic structure, therefore it gives another element of $\mathcal{M}_{\operatorname{Dol}}(\alpha)$. 
For this Higgs bundle we clearly have $H(\mathcal{E},\tau \theta_1)=\tau^3 q_1 = t q_1$, and the real Hitchin's equation reads as
\begin{equation*}
    F_{h_{\tau}} + |\tau |^2[\theta_1,\theta_1^{\dag_{h_{\tau}}}]=0,
\end{equation*}
where $h_{\tau}$ is the unique Hermitian--Einstein metric solving the equation. 
We consider the 1-parameter family of Higgs bundles $(\mathcal{E},\overline{\partial}_E,\tau \theta_1)$, where $| \tau | \rightarrow\infty$, or equivalently the family $(\mathcal{E},\overline{\partial}_E,\theta_t)$, where $|t q| \rightarrow\infty$ in the Hitchin base $\mathcal{H}$ (for $\theta_t$, see~\eqref{companion}). 
In this section, we analyze the limiting behavior of the correspondence given in Theorem~\ref{naht_RH}~\ref{naht}. 

\subsection{The limit of the Hermitian--Einstein metric}

In~\cite{MocSz}, T.~Mochizuki and the second author developed tools to understand the limiting behaviour of the solution $h_{\tau}$ of Hitchin's real  equation, as $| \tau | \to\infty$. 
The aim of this section is to describe the asymptotic behaviour of $h_{\tau}$ under this limit in our particular case, based on~\cite{MocSz}. 

\subsubsection{Reduction to the case $t>0$}

Let us work in polar coordinates $t = R e^{i\varphi}$, see~\eqref{eq:polar}. 
Let us set\footnote{To match the assumption of~\cite{MocSz} completely, we should take a complex manifold for $\mathcal{S}$; an open annulus with trivial radial dependence of the objects is appropriate.}
\[
\mathcal{S} = S^1 = \{ e^{i\varphi}\vert \; \varphi \in [0, 2\pi ] \}
\]
Set $X=\mathbb{C}P^1$ and $\mathcal{D} = \mathcal{S} \times D$. 
Set 
\[
    \widetilde{\mathcal{X}} = \{ (e^{i\varphi}, (z, \zeta ))\vert \; \zeta^3 = e^{i\varphi} q(z) \} \subset \mathcal{S} \times T^* X. 
\]
The family version~\cite[Theorem~1.8]{MocSz} applied to this family $\widetilde{\mathcal{X}}$ and parameter $R>0$ (more precisely, $|\tau | = R^{1/3}$) converging to $\infty$ (instead of $\tau \in\mathbb{C}^{\times}$) shows that the results stated below hold for the parameter $\tau$ too. 
In the sequel, we will replace limits $R\to\infty$ by $|\tau |\to\infty$. 

\subsubsection{Convergence to a decoupled harmonic metric}

A Hermitian metric $h$ on $E$ is said to be decoupled harmonic for $(\mathcal{E},\theta)$ if the equations 
\[
    F_{h} = 0 = [\theta,\theta^{\dag_{h}}] 
\]
hold. 
It is clear that $h$ is decoupled harmonic for $(\mathcal{E},\theta)$ if and only if it is decoupled harmonic for $(\mathcal{E}, \tau \theta )$ for some (equivalently, any) $\tau \neq 0$.
\begin{theorem}~\cite[Theorem~1.4]{MocSz}\label{thm:asymptotic_decoupling}
    Fix a stable, generically semi-simple Higgs bundle $(\mathcal{E},\overline{\partial}_E,\theta)$ of $0$ parabolic degree and let $h_{\tau}$ be its Hermitian--Einstein metric. 
    If $\mathcal{E}$ is the direct image of a line bundle on the normalization of $\Sigma_t$, then the sequence $h_{\tau}$ converges to a decoupled harmonic metric $h_{\infty}$ locally on $\mathbb{C}P^1 \setminus D$ in the $C^{\infty}$ topology,  as $|\tau |\to \infty$. 
\end{theorem}

\begin{remark}
    The generic semi-simplicity and direct image of a line bundle conditions trivially hold in our particular case.
    A priori, the limit $h_{\infty}$ could depend on the argument $\varphi$ of $t$. 
    However, the interpretation in terms of canonical decomposable filtration (see the next subsection) shows that this is not the case. 
\end{remark}

\subsubsection{Canonical decomposable filtration}

A crucial notion for Theorem~\ref{thm:asymptotic_decoupling} is that of a canonical decomposable filtration $\mathcal{P}^{\star}_*$ for the Higgs bundle, see~\cite[Definition~5.10]{MocSz} (and~\cite[Claim~2.1]{Sz_BNR} for a similar condition in the unramified irregular situation). 
If the Higgs bundle is generically semi-simple (which clearly holds in the case at hand), then in general there exists a degree $\ell$ covering $\varphi_{\ell}$ of $X=\mathbb{C}P^1$ ramified over $D$ such that locally near the points $D$ there exists a direct sum decomposition 
\[
    \varphi_{\ell}^* (\mathcal{E}, \theta) \cong \bigoplus_{i=1}^r (\mathcal{E}_i , \theta_i )
\]
with $\mathcal{E}_i$ of rank $1$. (In our case, we have $\ell = r= 3$.) 
A filtration $\mathcal{P}_*$ is said to be decomposable if for any $\alpha\in\mathbb{R}$ we have 
\[
    \varphi_{\ell}^* \mathcal{E}_{\alpha} = \bigoplus_{i=1}^r \mathcal{E}_{i,\alpha}, 
\]
where we use $\mathcal{E}_{\alpha}$ for the associated $\mathbb{R}$-parabolic sheaves. 

For our purposes, we only need the case of a maximally ramified spectral curve, which implies that it is locally connected, i.e. $S=1$ in~\cite[Section~5.2.2]{MocSz}. 
Let $\mathcal{P}_* \operatorname{det}\, (\mathcal{E})$ be the filtration induced by~\eqref{eq:filtration} on $\operatorname{det}\, (\mathcal{E})$. 
In this case,~\cite[Proposition~5.13]{MocSz} states that there exists a unique decomposable filtered bundle $\mathcal{P}^{\star}_* (\mathcal{E})$ such that  $\operatorname{det}\, \mathcal{P}^{\star}_*(\mathcal{E}) = \mathcal{P}_* \operatorname{det}\, (\mathcal{E})$. 
It is easy to see that with our choices of weights $\underline{\alpha}_P$, $\mathcal{P}_* \operatorname{det}\, (\mathcal{E})$ is the trivial filtration (simply because the weights at $P$ add up to $0$). 
We call $\mathcal{P}^{\star}_* (\mathcal{E})$ the canonical decomposable filtration.

\subsubsection{The induced Hermitian metric on the determinant line bundle}

We study the Hermitian metric $h_{\tau, \operatorname{det}\, E}$ on $\operatorname{det}\, E$ induced by $h_{\tau}$. 

\begin{lemma}
    The induced parabolic bundle $\operatorname{det}\, \mathcal{E}$ is the trivial line bundle with trivial parabolic filtration. 
    With respect to a suitable global trivialization of $\operatorname{det}\, \mathcal{E}$, we have $h_{\operatorname{det}\, \mathcal{E}}\equiv 1$. 
\end{lemma}

\begin{proof}
We observe that since $\operatorname{tr}\, \theta\equiv 0$, i.e. the Higgs field takes values in $\mathfrak{sl}(3,\mathbb{C})$, the metric $h_{\tau, \operatorname{det}\mathcal{E}}$ turns $\operatorname{det}E$ into a flat unitary parabolic line bundle. 
Moreover, the parabolic weight of $\operatorname{det}\mathcal{E}$ at $P\in D$ is the sum of the parabolic weights of $E$ at $P$, so it vanishes. 
Now, $\operatorname{det}\mathcal{E}$ is a line bundle of degree $0$ over $\mathbb{C} P^1$, therefore it is isomorphic to $\mathcal{O}_{\mathbb{C} P^1}$. 
Therefore, by uniqueness of the Hermitian--Einstein metric, we see that with respect to a global trivialization of $\mathcal{O}$, the metric $h_{\operatorname{det}\mathcal{E}}$ is some constant. 
By rescaling the trivialization of $\mathcal{O}$, we can arrange  $h_{\operatorname{det}\mathcal{E}}\equiv 1$.
\end{proof}

\subsubsection{Limiting metric}

Let us be given a Hermitian metric on a tame harmonic bundle $(\mathcal{E}, \theta )$ with parabolic divisor $D$. 
Let $P\in D$ and $z$ be a holomorphic chart centered at $P$. 
Simpson's theory~\cite{Simp3} induces a metric filtration on $\mathcal{E}$ given by 
\[
    \mathcal{P}^h_a \mathcal{E} (U ) = \left\{ f\in \mathcal{E} (U \setminus P ) \vert \; \mbox{for any} \; \varepsilon > 0, \;  |f|_h = O\left( |z|^{-a-\varepsilon}\right) \right\} .
\]
By~\cite[Lemma~7.4, Theorem~7.5]{MocSz}, the metric $h_{\infty}$ appearing in Theorem~\ref{thm:asymptotic_decoupling} is the unique decoupled harmonic metric satisfying the properties: 
\begin{enumerate}
    \item $\mathcal{P}^{h_{\infty}}_*$ agrees with the canonical decomposable filtration $\mathcal{P}^{\star}_*$ on $\mathcal{E}$, 
    \item $\operatorname{det} h_{\infty} = h_{\operatorname{det}\, \mathcal{E}}$. 
\end{enumerate}

In Lemma~\ref{lem:canonical_decomposable_filtration} we will determine $\mathcal{P}^{\star}_* \mathcal{E}$.

\subsubsection{Hermitian metric on the spectral line bundle}\label{sec:Hermitian_line}

We use $\zeta \mapsto \tau \zeta$ with $\tau^3 = t$ to identify $\Sigma_t$ with $\Sigma_1$, compatibly with $p_L$. 
As previously discussed, let $\mathcal{L}$ be the line bundle such that $\mathcal{E}=p_*\mathcal{L}$. 

Now, let us equip $\mathcal{L}\rightarrow\Sigma_1$ with a parabolic structure. 
For all the lifted points $\widetilde{P}\in\widetilde{D}$ on the spectral curve let us set the parabolic weights to be $\widetilde{\alpha}_{\widetilde{P}}=-1$.
Then $\operatorname{pdeg}\, \mathcal{L}=0$, because $\operatorname{deg}\, \mathcal{L}=3$. 
By~\cite{Simp3},~\cite{Biq}, there exists a unique metric $h_{\mathcal{L}}$ that  
\begin{enumerate}
    \item is adapted to the above parabolic structure 
    \item solves the abelian Hitchin's equation, and 
    \item induces the metric $h_{\operatorname{det}\mathcal{E}}$. 
\end{enumerate} 

\begin{remark}
    For $\mathcal{L} = \mathcal{L}_0 = \mathcal{O}_{\widetilde{\Sigma}_t}(\widetilde{D})$, $h_{\mathcal{L}}$ induces on 
    \[
    \mathcal{L}_0 \otimes \mathcal{O}_{\widetilde{\Sigma}_t}(- \widetilde{D}) \cong \mathcal{O}_{\widetilde{\Sigma}_t} 
    \]
    a flat metric compatible with the trivial parabolic structure.
    It follows that this induced metric is $h\equiv 1$ with respect to a global trivialization. 
\end{remark}

We define the orthogonal pushforward $p_* h_{\mathcal{L}}$ of the metric $h_{\mathcal{L}}$ as the unique metric on $p_* \mathcal{L}$ for which
\begin{enumerate}
\item the eigenspaces of $\theta$ are orthogonal to each other, i.e. for any local sections $l_1, l_2$ of $\mathcal{L}$ supported on different sheets of $\widetilde{\Sigma}_t$, we have 
$$
  p_* h_{\mathcal{L}} ( (p_L)_* l_1, (p_L)_* l_2 ) =  0 ,
$$ 
\item for any local section $l$ of $\mathcal{L}$ supported on a single sheet of $\widetilde{\Sigma}_t$, we have 
$$
  p_* h_{\mathcal{L}} ( (p_L)_* l, (p_L)_* l ) =   h_{\mathcal{L}} (l,l) 
$$ 
\end{enumerate}

By~\cite[Proposition~2.3]{MocSz}, there exists some flat Hermitian metric on $\mathcal{L}$ whose orthogonal pushforward is $h_{\infty}$. 
We next show that these definitions of $h_{\infty}$ and $h_{\mathcal{L}}$ are compatible. 

\begin{lemma}\label{lem:pushforward_metric}
The metric $h_{\infty}$ agrees with the orthogonal pushforward $p_* h_{\mathcal{L}}$ of $h_{\mathcal{L}}$. 
\end{lemma} 

\begin{proof}
The filtration induced on $\mathcal{E}$ by the parabolic filtration of $\mathcal{L}$ defined above is the canonical decomposable one (see Lemma~\ref{lem:canonical_decomposable_filtration}). 
The statement therefore essentially reduces to~\cite[Proposition~5.34]{MocSz}. 
\end{proof}

\subsubsection{Cover of the underlying curve}

We now set $N=\overline{B}_{\delta}(P)$ for $P\in D$ and some $0 < \delta \ll 1$ (see Figure~\ref{fig:annulus}). 
In the sequel, we will make use of a cover of the surface $\mathbb{C}P^1\setminus D$ by open regions $U_1, U_2$ where  
\begin{align}
   U_1 & = \mathbb{C}P^1\setminus \bigcup_{P\in D} \overline{B}_{\delta}(P) \label{eq:U1} \\
   U_2 & = \bigcup_{P\in D} B_{2\delta}(P) . \label{eq:U2}
\end{align}
We pick a loop $\gamma_P$ in the annulus $U_1 \cap U_2 = B_{2\delta}(P) \setminus \overline{B}_{\delta}(P)$ winding around $P$ once in the positive direction. 
On the other hand, we pick a path $\eta_P$ in $U_1$ from the base point $z_0$ to the starting point of $\gamma_P$, and also a path $\sigma_P$ from $P$ to the starting point of $\gamma_P$, see Figure~\ref{fig:annulus}.  

\begin{figure}[ht]
\centering
\includegraphics[width=16.0cm]{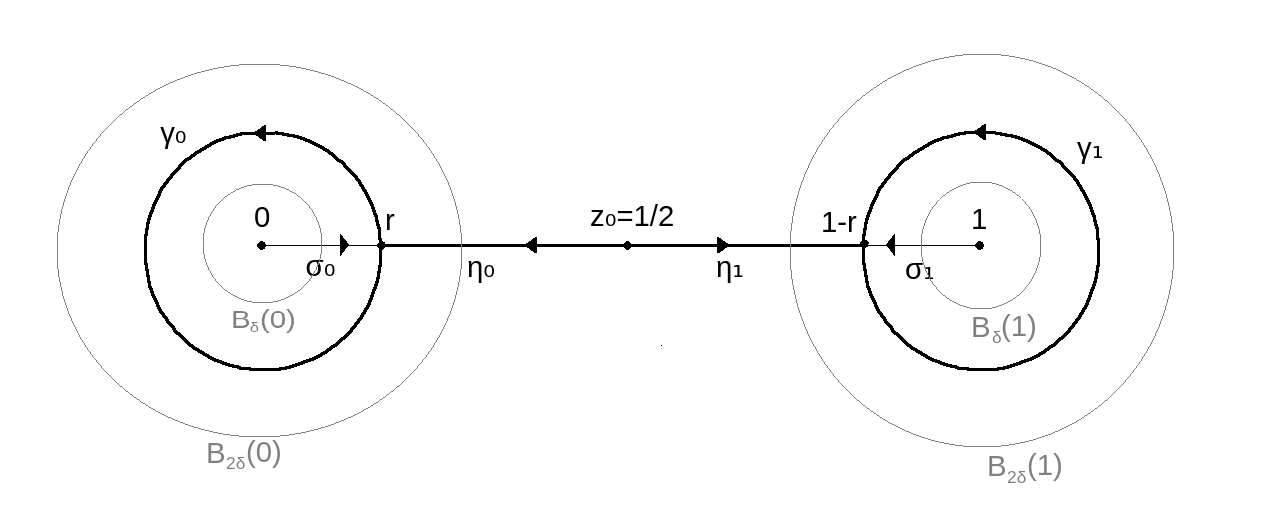}
\caption{The base point $z_0=\frac{1}{2}$ and path $\eta_P$ approaching the point $P\in D$, and the setup near $P$: the loop $\gamma_P$ lying in the annulus $B_{2\delta}(P)\setminus \overline{B}_{\delta}(P)$, and $\sigma_P$ connecting $P$ and the starting point of $\gamma_P$, where $P\in\{0,1\}$.}
\label{fig:annulus}
\end{figure}

\subsection{Analysis away from the ramification divisor}

In this section we will work over the set~\eqref{eq:U1}. 
Let $U \subset U_1$ be any simply connected open subset such that 
\[
p_L^{-1}(U) \cap \widetilde{\Sigma}_t = V_1 \sqcup V_2 \sqcup V_3, 
\] 
the restriction of $p_L$ to each $V_i$ being a homeomorphism onto $U$. 

\subsubsection{Induced unitary frame of $\mathcal{E}$}
Over $U$, the relation $(p_L)_* \mathcal{L} = \mathcal{E}$ reads as  
$$
    \restr{\mathcal{E}}{U} = \restr{\mathcal{L}}{V_1}\oplus \restr{\mathcal{L}}{V_2}\oplus \restr{\mathcal{L}}{V_3}. 
$$
Theorem~\ref{thm:asymptotic_decoupling} states in particular that $[\theta, \theta^{\dag_{h_{\infty}}} ] =0$. 
Let $\mathbf{l}_i$ be a $h_{\mathcal{L}}$-unitary trivialization of $\restr{\mathcal{L}}{V_i}$ and define 
\[
    \mathbf{f}_i = (p_L)_* \mathbf{l}_i . 
\]
By Lemma~\ref{lem:pushforward_metric}, 
\begin{equation}
    \label{frame2}
    (\mathbf{f}_1,\mathbf{f}_2,\mathbf{f}_3)
\end{equation}
is a smooth $h_{\infty}$-unitary trivialization of $\restr{E}{U}$. 
We will use this frame to express $\tau \theta_1 , \nabla_{\tau}$ throughout this section. 
In particular, $\theta, \theta^{\dag_{h_{\infty}}}$ both have diagonal matrices, and the matrix of $h_{\infty}$ is the $3\times 3$ identity.

\subsubsection{Flat unitary connection induced by $h_{\infty}, \mathcal{E}$}\label{sec:flat_unitary}

The holomorphic structure on the line bundle $\mathcal{L}$ is provided by the operator $\overline{\partial}_{\mathcal{L}}=\overline{\partial}+ \mu$, for some $\mu\in\Omega^{0,1}(\widetilde{\Sigma}_t)$. 
With respect to~\eqref{frame2}, the holomorphic structure of $\restr{\mathcal{E}}{U}$ reads as 
\begin{equation*}
    \overline{\partial}_E=\overline{\partial}+
    \begin{pmatrix}
        \mu_1 & 0 & 0\\
        0 & \mu_2 & 0\\
        0 & 0 & \mu_3
    \end{pmatrix},
\end{equation*}
with $\mu_i$ the restriction of $\mu$ to $V_i$. 
The metric $h_{\mathcal{L}}$ is Hermitian--Einstein, thus the associated unitary connection $\nabla_{h_{\mathcal{L}}}^{+}$ is flat on each $V_i$. 
Let us denote its connection form by 
\[
u_i=\mu_i-\overline{\mu}_i\in \Omega^1(V_i ,\sqrt{-1} \mathbb{R}) .
\]
Then the flat $\operatorname{U(1)}^{\times 3}$-connection form on $\restr{\mathcal{E}}{U}$, associated with $h_{\infty}$ is
\begin{equation*}
    \nabla_{h_{\infty}}^{+}=
    \begin{pmatrix}
        u_1 & 0 & 0\\
        0 & u_2 & 0\\
        0 & 0 & u_3
    \end{pmatrix}.
\end{equation*}

\subsubsection{Family of flat connections induced by $h_{\infty}, \theta$}

As we have already mentioned, the Higgs field is diagonal in the frame~\eqref{frame2}, with its eigenvalues~\eqref{Q} as diagonal elements. 
We are ready to express the integrable connection associated to the decoupled metric $h_{\infty}$ over $U$ with respect to the unitary frame~\eqref{frame2}:
\begin{equation}
    \begin{gathered}
    \label{conn1}
    \nabla_{\tau}^{\operatorname{appr}} = \nabla_{h_{\infty}}^{+} + \tau \theta_1 + (\tau \theta_1)^{\dag_{h_{\infty}}} = \nabla_{h_{\infty}}^{+} + \tau \theta_1 + h_{\infty}^{-1}\overline{\tau \theta_1}^T h_{\infty} =
    \\ = \operatorname{d} + 
    \left( \begin{smallmatrix}
        u_1 & 0 & 0\\
        0 & u_2 & 0\\
        0 & 0 & u_3
    \end{smallmatrix}\right)  + 
    \left( \begin{smallmatrix}
        \tau Q_1+\overline{\tau Q_1} & 0 & 0\\
        0 & \tau Q_2+\overline{\tau Q_2} & 0\\
        0 & 0 & \tau Q_3+\overline{\tau Q_3}
    \end{smallmatrix}\right) =
    \\ = \operatorname{d} + \left( \begin{smallmatrix}
        u_1 + 2\operatorname{Re}( \tau Q_1) & 0 & 0\\
        0 & u_2 + 2\operatorname{Re}(\varepsilon \tau Q_1) & 0\\
        0 & 0 & u_3 + 2\operatorname{Re}(\varepsilon^2 \tau Q_1)
    \end{smallmatrix} \right)
    \end{gathered}
\end{equation}
Notice that the formula 
\[
    u_i + 2\operatorname{Re}(\varepsilon^{i-1} Q_t)
\]
is a decomposition into imaginary and real part. 

\begin{remark}
Over the Hitchin section, i.e. for $ \mathcal{L} = \mathcal{L}_0 = \mathcal{O}_{\widetilde{\Sigma}_t}(\widetilde{D})$, we have $\mu\equiv 0$ for the $(0,1)$-partial connection defining $\mathcal{L}_0(- \widetilde{D})$. 
This then implies $u_i \equiv 0$ for all $1\leq i \leq 3$, and the limiting connection simplifies to 
\[
      \operatorname{d} + 2\operatorname{Re}
      \left( \begin{smallmatrix}
        \tau Q_1 & 0 & 0\\
        0 & \varepsilon \tau Q_1 & 0\\
        0 & 0 & \varepsilon^2 \tau Q_1
    \end{smallmatrix} \right).
\]
\end{remark}

\subsubsection{Estimate for the associated integrable connection}

Let $h_{\tau}$ denote the metric solving the Hitchin's equations for $(\mathcal{E}, \tau \theta_1 )$, and $\nabla_{\tau}$ the associated integrable connection~\eqref{connection}. 
On the other hand, we have the approximate flat connection~\eqref{conn1}. 
We consider the Sobolev norm $\Vert \cdot \Vert_{L^2_l(U, h_{\infty})}$ on sections of $End(E)\otimes \Omega^1$ defined using the standard Riemannian metric $|\operatorname{d} \! z|^2$ on $\mathbb{C}P^1$, the fiber metric $h_{\infty}$ and the Killing form. 
\begin{prop}
    Fix $l\in\mathbb{N}$. 
    Over any simply connected open subset $U$ of $U_1$, there exist $C_1>0, \epsilon_1, R_1 >0$ depending only on $l$ and $\delta>0$ (used to define $U_1$), such that for all $|\tau |>R_1$ 
    \begin{equation*}
        \Vert \nabla_{\tau} - \nabla_{\tau}^{\operatorname{appr}} \Vert_{L^2_l(U, h_{\infty})} \leq C_1 e^{-\epsilon_1 |\tau |}.
    \end{equation*}
    In particular, there exists $C_2>0, \epsilon_2, R_1>0$ depending only on $\delta>0$ such that for all $|\tau |>R_1$ 
     \begin{equation*}
        \Vert \nabla_{\tau} - \nabla_{\tau}^{\operatorname{appr}} \Vert_{C^0(U, h_{\infty})} \leq C_2 e^{-\epsilon_2 |\tau |}.
    \end{equation*}
\end{prop}

\begin{proof}
Let $s(h_{\infty}, h_{\tau})$ be defined by $h_{\tau} = h_{\infty}\cdot s(h_{\infty}, h_{\tau})$. 
We then have 
\begin{align*}
    \Vert \nabla_{h_{\tau}}^{+} - \nabla_{h_{\infty}}^{+} \Vert_{L^2_l(U)} & = \Vert \partial \left( \log h_{\tau} - \log h_{\infty} \right) \Vert_{L^2_l(U)}\\
    & = \Vert \partial \log s(h_{\infty}, h_{\tau}) \Vert_{L^2_l(U)}\\ 
    & \leq 2 \Vert \partial ( s(h_{\infty}, h_{\tau}) - \operatorname{id}) \Vert_{L^2_l(U)}\\ 
    & \leq 2 \Vert  s(h_{\infty}, h_{\tau}) - \operatorname{id} \Vert_{L^2_{l+1}(U)}\\ 
    & \leq C_1 e^{-\epsilon_1 |\tau |}, 
\end{align*}
where $\log$ stands for matrix logarithm and the last estimate follows from~\cite[Theorem~7.14]{MocSz}. 

For the other term, up to increasing $C_1$ and decreasing $\epsilon_1$ we have
\begin{align*}
    \Vert  h_{\infty}^{-1}\overline{\tau \theta_1}^T h_{\infty} - h_{\tau}^{-1}\overline{\tau \theta_1}^T h_{\tau} \Vert_{L^2_l(U)}^2 & = \Vert - ( \operatorname{Ad}( s(h_{\infty}, h_{\tau})^{-1} ) - \operatorname{id}) \operatorname{Ad} (h_{\infty}^{-1}) (\overline{\tau \theta_1}^T ) \Vert_{L^2_l(U)}^2 \\
    & \leq 4 \Vert s(h_{\infty}, h_{\tau}) - \operatorname{id} \Vert_{L^2_l(U)} \cdot \Vert  \operatorname{Ad} (h_{\infty}^{-1}) (\overline{\tau \theta_1}^T ) \Vert_{L^2_l(U)} \\
    & \leq C_1 e^{-\epsilon_1 |\tau |} \Vert  \operatorname{Ad} (h_{\infty}^{-1}) (\overline{\tau \theta_1}^T ) \Vert_{L^2_l(U)}
\end{align*}
by Cauchy--Schwartz. 
Now, the second factor of the last line can be bounded from above by $C_1' |\tau |$ for some $C_1'>0$, so it can be absorbed by the first term (possibly up to increasing $C_1$ again). 

The last statement immediately follows from Sobolev's embedding theorem. 
\end{proof}

\subsection{Near the ramification divisor}

In this section we describe asymptotic behaviour of solutions of Hitchin's equations on~\eqref{eq:U2}. 
Obviously, it is sufficient to deal with $B_{2\delta}(P)$ for a given $P\in D$, see Figure~\ref{fig:annulus}, where $0<\delta\ll 1$. 
Let us consider the cyclic triple cover 
\[
    \varphi_3\colon \widetilde{B}_{2\delta}(P)\to B_{2\delta}(P)
\]
defined by $\widetilde{\Sigma}_1$, i.e. 
\[
\zeta^3 = z(z-1), 
\]
where we use the trivialization 
\begin{equation}\label{eq:triviL}
    \frac{\operatorname{d}\!z}{z(z-1)}    
\end{equation}
of $L = K(D)$. 

\subsubsection{Frames and parabolic structure}

Let us fix a trivialization $\mathcal{L}\cong \mathcal{O}_{\widetilde{B}_{2\delta}(P)}$, i.e. a nowhere vanishing holomorphic section $\mathbf{l}$ of $\mathcal{L}$ over $\widetilde{B}_{2\delta}(P)$. 
Then, over $B_{2\delta}(P)\subset \mathbb{C}P^1$, $\mathcal{E}$ can be decomposed as
\begin{equation}\label{eq:ramified_decomp}
    \restr{\mathcal{E}}{B_{2\delta}(P)}\cong \mathcal{O}_{B_{2\delta}(P)}  \oplus \mathcal{O}_{B_{2\delta}(P)} \cdot \zeta \oplus \mathcal{O}_{B_{2\delta}(P)} \cdot \zeta^2 
\end{equation}
(see~\eqref{E}), and $\theta$ is the operation induced by multiplication by $\zeta$. 
Its residue is nilpotent with respect to the filtration 
\[
    0 \subset \mathcal{O}_{B_{2\delta}(P)} \cdot \zeta^2 \subset \mathcal{O}_{B_{2\delta}(P)} \cdot \zeta \oplus \mathcal{O}_{B_{2\delta}(P)} \cdot \zeta^2 \subset \restr{\mathcal{E}}{B_{2\delta}(P)}. 
\]

A metric $h$ is compatible with the parabolic structure if and only if  
\[
    |f_1(z) + f_2(z) \cdot \zeta + f_3(z) \cdot \zeta^2 |_h^2 \approx |f_1(z)|^2 |z|^{-2/3} + |f_2(z)|^2 + |f_3(z)|^2 |z|^{2/3}. 
\]
Recall the definition of the metric $h_{\mathcal{L}}$ from Section~\ref{sec:Hermitian_line}. 
Let us set 
\begin{equation}\label{eq:h}
    h\colon \widetilde{B}_{2\delta}(P) \to \mathbb{R}, \quad h(\zeta ) = h_{\mathcal{L}} (\mathbf{l}(\zeta ) , \mathbf{l}(\zeta ) ). 
\end{equation}
Recall that 
\[
    h \approx |\zeta |^{-2}
\]
because the parabolic weight of $\mathcal{L}$ is $-1$. 
With respect to the decomposition~\eqref{eq:ramified_decomp}, the metric $h_{\infty}$ is then given by the matrix 
\[
    h \begin{pmatrix}
        1 & \zeta & \zeta^2 \\
        \bar{\zeta} & \bar{\zeta} \zeta & \bar{\zeta} \zeta^2 \\
        \bar{\zeta}^2 & \bar{\zeta}^2 \zeta & \bar{\zeta}^2 \zeta^2
    \end{pmatrix}. 
\]

\subsubsection{Diagonalizing trivialization}

For $0<\delta\ll 1$, with respect to the trivialization~\eqref{eq:triviL} of $L$ let us use the approximations up to higher order terms on $\gamma_0,\gamma_1$, near the points $0,1\in\mathbb{C}P^1$: 
\begin{equation*}
    \restr{\theta_t}{\gamma_0}\approx  
    \begin{pmatrix}
        0 & 0 & t z\\ 
        1 & 0 & 0\\
        0 & 1 & 0
    \end{pmatrix}, \hspace{0.3cm} \restr{\theta_t}{\gamma_1}\approx  
    \begin{pmatrix}
        0 & 0 & t (z-1)\\ 
        1 & 0 & 0\\
        0 & 1 & 0
    \end{pmatrix}.
\end{equation*}
We focus on the first one, i.e. the neighbourhood of 0, but a similar analysis holds near $1$ too. 

The aim of this section is to introduce a smooth approximately unitary frame
\begin{equation}
    \label{frame}
    (\mathbf{e}_1,\mathbf{e}_2,\mathbf{e}_3) 
\end{equation}
defined over $\widetilde{B}_{2\delta}(P)$ in which the Higgs field is diagonal. 
We will expand the vectors~\eqref{frame} in coordinates with respect to~\eqref{eq:ramified_decomp}. 
They can be determined uniquely (up to phase) as the unit length eigenvectors of $\restr{\theta_t}{\gamma_0}$.  
A diagonalizing frame is given by 
\begin{align*}
        \mathbf{e}_1' & = 
        \begin{pmatrix}
        (tz)^{2/3} \\ 
        (tz)^{1/3} \\ 
        1 
    \end{pmatrix} \\
    \mathbf{e}_2' & = 
        \begin{pmatrix}
        \varepsilon^2 (tz)^{2/3} \\ 
        \varepsilon (tz)^{1/3} \\ 
        1 
    \end{pmatrix} 
    \\ 
    \mathbf{e}_3' & = 
        \begin{pmatrix}
        \varepsilon (tz)^{2/3} \\ 
        \varepsilon^2 (tz)^{1/3} \\ 
        1 
    \end{pmatrix} 
\end{align*}
where we have used the cubic root of unity $\varepsilon = e^{2\pi i/3}$. 
Then, the expression of the frame~\eqref{frame} with respect to~\eqref{eq:ramified_decomp} is:
\begin{align*}
        \mathbf{e}_1 & = \frac{1}{|\mathbf{e}_1'|_{h_{\infty}}} 
        \begin{pmatrix}
        (tz)^{2/3} \\ 
        (tz)^{1/3} \\ 
        1 
    \end{pmatrix} \\
    \mathbf{e}_2 & = \frac{1}{|\mathbf{e}_2'|_{h_{\infty}}} 
        \begin{pmatrix}
        \varepsilon^2 (tz)^{2/3} \\ 
        \varepsilon (tz)^{1/3} \\ 
        1 
    \end{pmatrix} 
    \\ 
    \mathbf{e}_3 & = \frac{1}{|\mathbf{e}_3'|_{h_{\infty}}} 
        \begin{pmatrix}
        \varepsilon (tz)^{2/3} \\ 
        \varepsilon^2 (tz)^{1/3} \\ 
        1 
    \end{pmatrix} .
\end{align*}

\subsubsection{Identifying the canonical decomposable filtration}
We are ready to proving: 
\begin{lemma}\label{lem:canonical_decomposable_filtration}
With our choices, the canonical decomposable filtration $\mathcal{P}^{\star}_* \mathcal{E}$ is the trivial filtration
\[
    \mathcal{P}^{\star}_a \mathcal{E} = z^{-[a]} \cdot \mathcal{E}, 
\]
where 
\[
    [a] = \max \{ n \in \mathbb{Z} \vert \quad n\leq a \} 
\]
is the usual floor function. 
\end{lemma}

\begin{proof}
We focus on $P=0$. 
Set $U = B_{2\delta}(0), U^{(3)} = \widetilde{B}_{2\delta}(0)$ and let 
\[ 
    \varphi_3 \colon U^{(3)}\to U 
\]
be the covering $\zeta^3 = z (z-1)$, as before. 
Let 
\[
    U^{(3),*} = \varphi_3^{-1} ( U^* ) = \varphi_3^{-1} ( U - \{ 0 \} ). 
\]
By the formulas of the previous section, we have 
\[
    \mathbf{e}_1 \wedge \mathbf{e}_2 \wedge \mathbf{e}_3 = z g(z )
\]
in $U$ for some holomorphic function $g$ with $g(0)\neq 0$. 

The sum of the parabolic weights being $0$, the induced filtration $\mathcal{P}_* \operatorname{det}\, \mathcal{E}$ is trivial. 
So, we have 
\[
    \mathbf{e}_1 \wedge \mathbf{e}_2 \wedge \mathbf{e}_3 \in \varphi_3^* (\mathcal{P}_1 \operatorname{det}\, \mathcal{E} ) 
\]
and 
\[
    \mathbf{e}_1 \wedge \mathbf{e}_2 \wedge \mathbf{e}_3 \notin \varphi_3^* (\mathcal{P}_b \operatorname{det}\, \mathcal{E} ) 
\]
for any $b<1$. 
The formula after~\cite[Lemma~5.2]{MocSz} specializes to 
\[
    \mathcal{P}_a (\varphi_3^* \operatorname{det}\, \mathcal{E} ) = \sum_{b\in\mathbb{R}, k\in\mathbb{Z}, rb+k\leq a } \zeta^{-k} \varphi_3^* (\mathcal{P}_b \operatorname{det}\mathcal{E} ). 
\]
Let us apply it with $k=-3, b=1$ to obtain that 
\[
    \mathbf{e}_1 \wedge \mathbf{e}_2 \wedge \mathbf{e}_3 \in \mathcal{P}_0 (\varphi_3^* \operatorname{det}\mathcal{E} ), 
\]
and 
\[
    \mathbf{e}_1 \wedge \mathbf{e}_2 \wedge \mathbf{e}_3 \notin \mathcal{P}_a (\varphi_3^* \operatorname{det}\mathcal{E} )
\]
for any $a<0$. 
This gives the value $0$ for the constant $b$ in the proof of~\cite[Proposition~5.13]{MocSz}. 
The canonical decomposable filtration on the $i$-th component is then defined by 
\[
    \mathcal{P}^{\star}_a ( \mathcal{O}_{U^{(3),*}} \mathbf{e}_i )  = \zeta^{-[a]} \mathcal{O}_{U^{(3)}} \mathbf{e}_i .
\]
Said differently, the canonical decomposable filtration on the $i$-th component is the trivial filtration. 
Next, this is extended to $\varphi_3^* \mathcal{E}$ by direct sum: 
\[
    \mathcal{P}^{\star}_a ( \varphi_3^* \mathcal{E} ) = \bigoplus_{i=1}^3 \mathcal{P}^{\star}_a ( \mathcal{O}_{U^{(3),*}} \mathbf{e}_i ) . 
\]
It is immediate that this is the trivial filtration, because we have the trivial filtration on all summands. 
At last, we define the canonical decomposable filtration on $E$ as the Galois descent of this filtration. 
The statement follows. 
\end{proof}

\subsubsection{Equivariance and scaling properties}

We discuss some further consequences of the formula for~\eqref{frame}. 

First, for given $|t|=R$ the choice of radius $r=R^{-1}$ guarantees that the coordinates of~\eqref{frame} remain bounded, namely if we switch to local polar coordinates on the circle $\gamma_0 (t)$ of radius $r = \frac {3 \delta}{2}$ centered at $0$, 
with angle coordinate $\vartheta\in [0,2\pi ]$, they are given by the following (recalling again that $\varepsilon = e^{2\pi \sqrt{-1}/3}$) 
\begin{align*}
        \mathbf{e}_1 & \approx \frac{1}{\sqrt{3}}
        \begin{pmatrix}
        e^{2 \sqrt{-1}(\vartheta + \varphi)/3} \\ 
        e^{\sqrt{-1}(\vartheta + \varphi)/3} \\ 
        1 
    \end{pmatrix} \\
    \mathbf{e}_2 & \approx \frac{1}{\sqrt{3}}
        \begin{pmatrix}
        e^{2 \sqrt{-1}(\vartheta + \varphi + 2\pi)/3} \\ 
        e^{\sqrt{-1}(\vartheta + \varphi + 2\pi)/3} \\ 
        1 
    \end{pmatrix} 
    \\ 
    \mathbf{e}_3 & \approx \frac{1}{\sqrt{3}}
        \begin{pmatrix}
        e^{2 \sqrt{-1}(\vartheta + \varphi + 4\pi)/3} \\ 
        e^{\sqrt{-1}(\vartheta + \varphi + 4\pi)/3} \\  
        1 
    \end{pmatrix} 
\end{align*}

Our second observation relates to the deck transformation group $\operatorname{Gal}(\varphi_3)\cong \mu_3$ of $\varphi_3$. 
Let $\gamma\in \operatorname{Gal}(\varphi_3)$ be a generator. 
\begin{prop}
Let $\mathcal{L} = \mathcal{L}_0$, i.e. we assume that $(\mathcal{E}, \theta )$ lies on the Hitchin section. 
Then, the metric $h_{\infty}$ is $\operatorname{Gal}(\varphi_3)$-invariant, namely $h(\gamma \cdot \zeta ) = h(\zeta )$ for every $\zeta\in U^{(3)}$. 
\end{prop}

\begin{proof}
There exists a canonical isomorphism $\gamma^* \mathcal{O}_{\widetilde{\Sigma}_t} \cong \mathcal{O}_{\widetilde{\Sigma}_t}$. 
It induces a canonical isomorphism $\gamma^* \mathcal{O}_{\widetilde{\Sigma}_t}(\widetilde{D}) \cong \mathcal{O}_{\widetilde{\Sigma}_t}(\widetilde{D})$ because $\widetilde{D}$ is fixed by $\operatorname{Gal}(\varphi_3)$. 
The metric $\gamma^* h_{\mathcal{L}_0}$ is Hermitian--Einstein on $\mathcal{L}_0$, and compatible with the same parabolic structure. 
By uniqueness, we get $\gamma^* h_{\mathcal{L}_0} = c \cdot h_{\mathcal{L}_0}$ for some $c\in \mathbb{R}_+$. 
Because $\gamma^3 = 1$, we have $c^3 = 1$. 
The conclusion follows. 
\end{proof}

\section{Asymptotic analysis of the Riemann--Hilbert correspondence}\label{sec:RH}

In this section, our aim is to apply the Riemann--Hilbert map to the integrable connection associated by the non-abelian Hodge correspondence to a Higgs bundle in a Hitchin fiber close to infinity. 
The asymptotic connection form of such integrable connections was obtained in Section~\ref{sec:Dol}. 
The procedure consists in determining the monodromy on a positively oriented loop around the punctures. 
More precisely, let us fix some base point $z_0\in\mathbb{C}P^1 \setminus D$, and consider the concatenations of paths $\eta_P*\gamma_P*\eta_P^{-1}$ where, as usual, $\eta_P^{-1}$ is the path $\eta_P$ with opposite orientation (see Fig.2). 
This is clearly a simple loop around the puncture $P\in D$ separating it from the other two punctures. 
We need to determine the parallel transport maps along these paths  of the flat connections whose approximations were given in Section~\ref{sec:Dol}, and determine the monodromies by integrating the connection forms on the paths and exponentiating. 
We will do so separately on the paths $\gamma_P$ and $\eta_P$, and the mondoromy on the concatenation will be the triple product of the parallel transport maps corresponding to the above decomposition into paths. 
From now on, the notation "$\approx$" will mean "agrees up to exponentially small terms", or more precisely: 
\begin{equation*}
    A\approx B \iff AB^{-1}=\operatorname{id}+O(e^{-\epsilon R^{1/3}})
\end{equation*}
for some $\epsilon>0$. 
This implies in particular that if $B$ has an asymptotic expansion as $R\to\infty$, then the same holds for $A$ too with the same expansion. 
Notice that by Baker--Campbell--Dynkin--Hausdorff, $A\approx B$ is equivalent to requiring 
\begin{equation*}
B^{-1}A=\operatorname{id}+O(e^{-\epsilon R^{1/3}}). 
\end{equation*}

\subsection{Hitchin WKB problem}\label{sec:WKB}

In this section we recall results of~\cite[Section~2.4]{Moc}. 
We use the simplifying convention that the values of $C,\epsilon > 0$ may change from one occurrence to the other. 
Moreover, in this section for simplicity we replace $R^{1/3}$ by $R$. 

\subsubsection{The dilation spectrum}

Let $(V_0,h_0), (V_1,h_1)$ be finite dimensional Hermitian vector spaces, of dimension $r$ over $\mathbb{C}$. 
Let $\Pi\colon V_0\to V_1$ be a linear isomorphism. 
According to the spectral theorem applied to the positive symmetric matrix $h_1(e_i (0), e_j(0))$ where $e_i(0)$ is any orthonormal basis of $(V_0,h_0)$, there exist orthonormal bases $(q_1(s), \ldots q_r(s))$ of $(V_s,h_s)$ ($s \in \{ 0, 1\}$) such that 
\[
    \Pi (q_j (0)) = e^{\beta_j} q_j (1)\quad 1\leq j\leq r
\]
for some real numbers 
\[
    \beta_1 \geq \beta_2 \geq \cdots  \beta_r .
\]
The vector 
\[
    (\beta_1 , \ldots, \beta_r) \in \mathbb{R}^r 
\]
is called the \emph{dilation spectrum} of $\Pi$, see~\cite[Definition~4.2,~Lemma~4.5]{KatNollPandSimp}.

\subsubsection{Non-critical paths}\label{sec:noncritical}
Let $X$ be any Riemann surface (in our case, an open subset of $ \mathbb{C}P^1 \setminus D$). 
Let us fix any smooth path 
\[
    \gamma\colon [0,1]\to X. 
\]
Consider maps  
\begin{align*}
   A & \colon X \to \mathfrak{gl}(n, \mathbb{C}) \\
   U & \colon X \to \mathfrak{gl}(n, \mathbb{R})
\end{align*}
satisfying $[A(z), U(z)]=0$ for every $z\in X$. 
Then, locally on  $[0,1]$ there exists a common diagonalizing frame of $\gamma^* A$ and $\gamma^* U$. 
We may assume that this is the standard basis, so that $\gamma^*A, \sqrt{-1} \gamma^*U$ are diagonal matrices. 
Let 
\[
a_1(t) \operatorname{d}\! t, \ldots, a_n(t) \operatorname{d}\! t\in \Omega^1([0,1], \mathbb{C})
\]
stand for the eigenvalues of $A(t)$, 
\[
u_1 (t)\operatorname{d}\! t, \ldots , u_n(t)\operatorname{d}\! t \in \Omega^1([0,1], \sqrt{-1} \mathbb{R})
\]
stand for those of $\sqrt{-1} U(t)$ with respect to this basis. 
We will write $\mathfrak{t}$ for the Cartan subalgebra of diagonal matrices and $T$ for the corresponding maximal torus. 

We make the assumption that $\gamma$ is non-critical in the following sense. 
\begin{assn}\label{assn:regular}
For any $t\in [0,1]$ the matrix $\operatorname{Re} A(\gamma (t))$ is regular semi-simple, i.e. for any $1\leq j < j' \leq n$ we have 
\[
    \operatorname{Re} a_j (t ) \neq \operatorname{Re} a_{j'} (t ) .
\]
\end{assn}
Up to a permutation, we may and will assume 
\[
    \operatorname{Re} a_j (t ) < \operatorname{Re} a_{j'} (t )  \quad \mbox{for all} \; 1\leq j < j' \leq n .
\]
Let us set 
\[
    \alpha_j = - \int_0^1 \operatorname{Re} a_j(t) \operatorname{d}\! t \in \mathbb{R}. 
\]

\subsubsection{Dilation exponents of parallel transport of harmonic bundles}

We consider the Higgs bundle over $X$ of the form 
\begin{equation}\label{eq:direct_sum_Higgs_bundle}
        (\mathcal{E}, R \theta ) \cong \bigoplus_{j=1}^r (\mathcal{E}_j, R \phi_j ) 
\end{equation}
where $\mathcal{E}_j$ is a holomorphic line bundle. 
We assume 
\[
    \gamma^* \phi_j = a_j (t) \operatorname{d}\! t 
\]
and that Assumption~\ref{assn:regular} holds. 
Let $h_{R}$ be a harmonic metric for~\eqref{eq:direct_sum_Higgs_bundle} and $\nabla_{R}$ be the associated integrable connection. 
Let 
\[
  B \colon [1, \infty ] \times X \to \mathfrak{gl}(n, \mathbb{C}) 
\]
be the connection matrix of $\nabla_{R}$ with respect to trivializations of $\mathcal{E}_j$. 
We define the fundamental solution 
\[
    Y\colon \mathbb{C}^{\times} \times [0,1] \to \operatorname{GL}(r, \mathbb{C}) 
\]
as the solution of the initial value problem 
\[
    \frac{\operatorname{d}\! Y(R, t)}{\operatorname{d}\! t} = B(R,\gamma(t)) Y(R, t), \quad Y(R, 0) = \operatorname{id}. 
\]
We define 
\[
    \Pi_{\gamma} (R ) = Y(R , 1) \in \operatorname{GL}(r, \mathbb{C}).
\]
We define $(V_s,h_s) = (E_{\gamma(s)}, (h_R)_{\gamma(s)})$ for $s\in \{ 0,1 \}$. 
Then, $\Pi_{\gamma} (R )$ can be thought of as a linear map in $\operatorname{Hom}(V_0, V_1)$. 
Let 
\[
    (\beta_1 (R ) , \ldots, \beta_r (R )) 
\]
be the dilation spectrum of $\Pi_{\gamma} (R)$. 
\begin{theorem}~\cite[Theorem~2.17]{Moc}\label{thm:dilation_spectrum}
    There exist $R_0\in \mathbb{R}_+$ such that for all $R>R_0$ we have 
    \[
       \left\vert \frac 1R (\beta_1 (R) , \ldots, \beta_r (R)) -  (\alpha_1, \ldots , \alpha_r) \right\vert \leq C e^{-\epsilon R}. 
    \]
\end{theorem}

\subsubsection{Explicit description of the parallel transform}

The proof of Theorem~\ref{thm:dilation_spectrum} is obtained via an explicit description of the parallel transform $\Pi_{\gamma} (R)$. 
Recall the diagonal connection form~\eqref{conn1}. 
It is proven in~\cite[Theorem~2.17]{Moc} that there exist 
\begin{enumerate}
    \item an orthonormal frame $\mathbf{p}_1 (0), \ldots, \mathbf{p}_r (0)$ of $\restr{\mathcal{E}}{\gamma(0)}$, connected to 
    \[
    \mathbf{f}_1 (\gamma(0)),\cdots ,\mathbf{f}_r (\gamma(0))
    \]
    (see~\eqref{frame2} for the particular case $r=3$) by some matrix $L(0)$, 
    \item an orthonormal frame $\mathbf{p}_1 (1), \ldots, \mathbf{p}_r (1)$ of $\restr{\mathcal{E}}{\gamma(1)}$, connected to 
    \[
    \mathbf{f}_1 (\gamma(1)),\cdots ,\mathbf{f}_r (\gamma(1))
    \]
    by some matrix $L(1)$,
    \item $H\in \Omega^1 ([0,1], \mathfrak{t})$, with $j$-th diagonal entry denoted by $h_j (t) \operatorname{d}\! t$, 
\end{enumerate}
all these quantities depending on $R$, and fulfilling the following properties: 
\begin{enumerate}
    \item $\Vert L(s) - \operatorname{id} \Vert \leq C e^{-\epsilon R}$ for $s\in \{ 0,1 \}$, 
    \item $\Vert H(t) \Vert \leq C e^{-\epsilon R}$ for all $t\in [ 0,1 ]$,
    \item the matrix of $\Pi_{\gamma} (R )$ with respect to the bases $\mathbf{p}_1 (s), \ldots, \mathbf{p}_r (s)$ is given by the diagonal matrix 
    \[
        \operatorname{diag} \left( \exp \left( R \alpha_j - \int_0^1 u_j (t) + h_j (t) \operatorname{d}\! t \right)  \right)_{j=1}^r, 
    \]
    (see Section~\ref{sec:noncritical} for $\alpha_j$). 
\end{enumerate}

\subsection{Monodromy of the diagonalizing frame}

Using the results of Section~\ref{sec:WKB}, we now compute the monodromy along tiny loops winding around the punctures once in the positive direction. 
We will again only treat the case of the loop $\gamma_0$ (see Figure~\ref{fig:annulus}), but the same argument works for $\gamma_1$ too.

\begin{prop}
    The monodromy matrix $M_{0,R,\varphi}$ of $\nabla_t$ with respect to the frame~\eqref{frame2} is given by $M_{0,R,\varphi}= T N_{0,R,\varphi}$, where 
    \begin{equation*}
    N_{0,R,\varphi}=\left(\begin{smallmatrix}
        \exp(2(Rr)^{1/3}\operatorname{Re}(e^{i\varphi/3}3(\varepsilon -1))) & 0 & 0\\
        0 & \exp(2(Rr)^{1/3}\operatorname{Re}(\varepsilon e^{i\varphi/3}3(\varepsilon -1))) & 0\\
        0 & 0 & \exp(2(Rr)^{1/3}\operatorname{Re}(\varepsilon^2 e^{i\varphi/3}3(\varepsilon -1)))
    \end{smallmatrix}\right) 
\end{equation*}
and 
\begin{equation}\label{eq:T}
    T=\begin{pmatrix}
        0 & 0 & 1\\
        1 & 0 & 0\\
        0 & 1 & 0
    \end{pmatrix}. 
\end{equation}
\end{prop}

\begin{proof}
Obviously, the pull-back of the Higgs bundle to the spectral curve decomposes as a direct sum of Higgs line bundles 
\[
    \varphi_3^* (\mathcal{E}, \tau \theta_1  ) \cong \bigoplus_{i=1}^3 ( \mathcal{O}_{U^{(3),*}} \mathbf{e}_i , Q_{i,R,\varphi} ), 
\]
the summands being the restrictions of $\mathcal{L}$ to the sheets of $\Sigma_t$. 
Recall the filtration $\mathcal{P}^{h_{\infty}}_*$ induced by $h_{\infty}$. 
According to Lemma~\ref{lem:canonical_decomposable_filtration}, $\mathcal{P}^{h_{\infty}}_*$ is the trivial filtration. 
According to Lemma~\ref{lem:pushforward_metric}, $h_{\infty}$ is the direct image of $h_{\mathcal{L}}$. 
It follows that $h_{\mathcal{L}}$ induces the trivial filtration on $\mathcal{L}$. 
We infer that the pull-back by $\varphi_3$ of the parallel transport of $\nabla_t$ is diagonal. 
In order to find its eigenvalues, we need to find a fundamental solution of the local system. 
Equivalently, we need to integrate the connection form~\eqref{conn1} on $\gamma_0$, that is:
\begin{equation*}
    \left(\begin{smallmatrix}
        2R^{1/3}\operatorname{Re}(e^{i\varphi/3}\int_{\gamma_0}z^{-2/3}\operatorname{d}\! z) & 0 & 0\\
        0 & 2R^{1/3}\operatorname{Re}(\varepsilon e^{i\varphi/3}\int_{\gamma_0}z^{-2/3}\operatorname{d}\! z) & 0\\
        0 & 0 & 2 R^{1/3}\operatorname{Re}(\varepsilon^2 e^{i\varphi/3}\int_{\gamma_0}z^{-2/3}\operatorname{d}\! z)
    \end{smallmatrix}\right) + U_0.
\end{equation*}
where $U_0$ is the diagonal matrix with diagonal entries $\int_{\gamma_0} u_i$. 
The integrals appearing here can be determined explicitly, because $\gamma_0$ is just the boundary of a disc of radius $r$ around $0\in\mathbb{C}P^1$:
\begin{equation*}
    \int_{\gamma_0}z^{-2/3}\operatorname{d}\! z=\int_0^{2\pi} \frac{rie^{i\vartheta}\operatorname{d}\! \vartheta}{r^{2/3}e^{2i\vartheta/3}}= 3r^{1/3}(\varepsilon -1)
\end{equation*}
Now, to get the parallel transport map of $\varphi_3^* \nabla_t$, we take the exponential of this matrix, and we obtain $N_0$. 
However this is not yet the monodromy of the flat connection $\nabla_t$ on $\gamma_0$ with respect to any basis, because when we transport the frame~\eqref{frame} around the loop in the positive direction, the effect on the frame is a cyclic permutation. Indeed, one can easily check that 
\begin{equation}\label{eq:cyclic}
    \begin{gathered}
    \mathbf{e}_1(\vartheta=0)=\mathbf{e}_3(\vartheta=2\pi)
    \\ \mathbf{e}_3(\vartheta=0)=\mathbf{e}_2(\vartheta=2\pi)
    \\ \mathbf{e}_2(\vartheta=0)=\mathbf{e}_1(\vartheta=2\pi)
    \end{gathered}
\end{equation}
Thus the monodromy of the diagonalizing frame $(\mathbf{e}_1,\mathbf{e}_2,\mathbf{e}_3)$ is the permutation matrix $T$, and we get the desired result 
\[
    M_{0,R,\varphi}=TN_{0,R,\varphi}e^{U_0}.
\]
\end{proof}
Notice that the monodromy $M_{1,R,\varphi}$ on $\gamma_1$ near the puncture $1\in\mathbb{C}P^1$ is just the same, because the integrals coincide:
\begin{equation*}
    \int_{\gamma_0}z^{-2/3}\operatorname{d}\! z=\int_{\gamma_1}(z-1)^{-2/3}\operatorname{d}\! z
\end{equation*}

\subsection{Parallel transport to the punctures and the whole monodromy}\label{sec:monodromy}

\subsubsection{Choices of non-critical paths}

We now fix our choice of the base point: we set $z_0 = \frac 12 \in\mathbb{C}$, as well as the paths $\eta_0, \eta_1$: we let $\eta_0 (t) = (1-t)\frac 12 + t r$ and $\eta_1 (t) = (1-t)\frac 12 + t (1-r)$. 
In words, we let $\eta_0, \eta_1$ be the straight line segments connecting $z_0$ to $r$ and $1-r$, respectively. 
Recall the notion of a non-critical path from Assumption~\ref{assn:regular}. 
Recall from the polar coordinates~\eqref{eq:polar} that $\operatorname{arg}(\tau ) = \varphi/3$. 

\begin{lemma}\label{lem:noncritical}
    These paths $\eta_0, \eta_1$ are non-critical for $\tau \theta_1$ if and only if $\varphi = k \pi$ for some $k\in \mathbb{Z}$. 
\end{lemma}

\begin{proof}
    Recall from~\eqref{Q} that the eigenvalues of $\tau \theta_1$ are equal to 
    \[
    \varepsilon^{j-1} R^{1/3}e^{i\varphi/3}\frac{\operatorname{d}\! z}{z^{2/3}(z-1)^{2/3}}. 
    \]
    Clearly, $\eta_0^* \operatorname{d}\! z , \eta_1^* \operatorname{d}\! z$ are some nonzero real multiples of $\operatorname{d}\! t$. 
    Moreover, $t(t-1) < 0$ for $0 < t < 1$ implies $t^{-2/3}(t-1)^{-2/3}>0$. 
    The constant $R^{1/3}$ is real too. 
    So, we need to solve the equation 
    \[
        \operatorname{Re}\left( \varepsilon^{j-1} e^{i\varphi/3} \right) =  \operatorname{Re}\left( \varepsilon^{j'-1} e^{i\varphi/3} \right)
    \]
    for $0\leq j\neq j' \leq 2$. 
    This happens if and only if $\varphi/3$ is an integer multiple of $\pi/3$. 
\end{proof}

Therefore, Assumption~\ref{assn:regular} for $\gamma = \eta_0, \eta_1$ translates into: 
\begin{equation}\label{eq:critical_angles}
    \varphi \notin \mathbb{Z}\cdot \pi. 
\end{equation}
We say that the angles $\varphi \in \mathbb{Z} \cdot \pi$ are \emph{critical angles of the first kind}. 

\subsubsection{Product of parallel transport maps}

Under the assumption~\eqref{eq:critical_angles}, it follows from Theorem~\ref{thm:dilation_spectrum} that the matrices of the parallel transport maps of~\eqref{conn1} along $\eta_0$ and $\eta_1$ with respect to the frame~\eqref{frame2} are respectively asymptotic to  
\begin{equation*}
    \begin{gathered}
        L_{0,t}=\left(\begin{smallmatrix}
       \exp(\int_{\eta_0}u_1 + 2\operatorname{Re}\int_{\eta_0}Q_t) & 0 & 0\\
        0 & \exp(\int_{\eta_0}u_2 + 2\operatorname{Re}(\varepsilon \int_{\eta_0}Q_t)) & 0\\
        0 & 0 &  \exp(\int_{\eta_0}u_3 + 2\operatorname{Re}(\varepsilon^2 \int_{\eta_0}Q_t))
    \end{smallmatrix}\right)
    \\ L_{1,t}=\left(\begin{smallmatrix}
       \exp(\int_{\eta_1}u_1 + 2\operatorname{Re}\int_{\eta_1}Q_t) & 0 & 0\\
        0 & \exp(\int_{\eta_1}u_2 + 2\operatorname{Re}(\varepsilon \int_{\eta_1}Q_t)) & 0\\
        0 & 0 &  \exp(\int_{\eta_1}u_3 + 2\operatorname{Re}(\varepsilon^2 \int_{\eta_1}Q_t))
    \end{smallmatrix}\right)
    \end{gathered}
\end{equation*}

Over $\eta_0$ we had smooth unitary frame~\eqref{frame2}, and over $\gamma_0$ we had another smooth unitary frame~\eqref{frame}. 
These two frames both diagonalize $\theta$, and are connected by a unitary transformation. 
Given that $\theta(\eta_0(1))$ is regular semi-simple, we may therefore assume that these two frames agree at $\eta_0(1) = r>0$. 
Denoting the monodromy of $\nabla_{\tau}$ along $\eta_0*\gamma_0*\eta_0^{-1}$ and $\eta_1*\gamma_1*\eta_1^{-1}$ by $A_t$ and $B_t$ respectively, we have 
\begin{equation}
    \label{multi}
    \begin{gathered}
        A_t=A_{R,\varphi}=L_{0,t}^{-1}TN_{0,R,\varphi}e^{U_0}L_{0,t}
        \\ B_t=B_{R,\varphi}=L_{1,t}^{-1}TN_{1,R,\varphi}e^{U_1}L_{1,t}, 
    \end{gathered}
\end{equation}
where $T$ is the cyclic permutation matrix introduced in~\eqref{eq:T}.
(Recall that $T$ is needed because of the cyclic permutation~\eqref{eq:cyclic} of the sheets around the given puncture.) 
After we execute the matrix multiplications, we get that both matrices have the following form similar to a companion matrix: 
\begin{equation*}
    A_t=A_{R,\varphi}=\begin{pmatrix}
        0 & 0 & A_3(t)\\
        A_1(t) & 0 & 0\\
        0 & A_2(t) & 0
    \end{pmatrix}, \hspace{0.2 cm}
    B_t=B_{R,\varphi}=\begin{pmatrix}
        0 & 0 & B_3(t)\\
        B_1(t) & 0 & 0\\
        0 & B_2(t) & 0
    \end{pmatrix}. 
\end{equation*}

\subsubsection{Asymptotic behaviour of monodromy maps and period integrals}

\begin{prop}\label{prop:asymptotic_monodromy}
For the nonzero elements of $A_t$, as $t\rightarrow\infty$ we have
\begin{equation*}
    \begin{gathered}
        A_1(t)\approx\operatorname{exp}\left(R^{1/3}\operatorname{Re}(e^{i\varphi/3}\pi_0)+\int_{\eta_0}(u_1-u_2)+\int_{\gamma_0}u_1\right)
        \\ A_2(t)\approx\operatorname{exp}\left(R^{1/3}\operatorname{Re}(e^{i\varphi/3}\varepsilon\pi_0)+\int_{\eta_0}(u_2-u_3)+\int_{\gamma_0}u_2\right)
        \\ A_3(t)\approx\operatorname{exp}\left(R^{1/3}\operatorname{Re}(e^{i\varphi/3}\varepsilon^2\pi_0)+\int_{\eta_0}(u_3-u_1)+\int_{\gamma_0}u_3\right)
    \end{gathered}
\end{equation*}
where $\pi_0=2(\varepsilon-1)\int^{z_0}_0\frac{\operatorname{d}\! z}{z^{2/3}(z-1)^{2/3}}$ along $\sigma_0*\eta_0^{-1}$. Similarly for $B_t$ we have
\begin{equation*}
    \begin{gathered}
        B_1(t)\approx\operatorname{exp}\left(R^{1/3}\operatorname{Re}(e^{i\varphi/3}\pi_1)+\int_{\eta_1}(u_1-u_2)+\int_{\gamma_1}u_1\right)
        \\ B_2(t)\approx\operatorname{exp}\left(R^{1/3}\operatorname{Re}(e^{i\varphi/3}\varepsilon\pi_1)+\int_{\eta_1}(u_2-u_3)+\int_{\gamma_1}u_2\right)
        \\ B_3(t)\approx\operatorname{exp}\left(R^{1/3}\operatorname{Re}(e^{i\varphi/3}\varepsilon^2\pi_1)+\int_{\eta_1}(u_3-u_1)+\int_{\gamma_1}u_3\right)
    \end{gathered}
\end{equation*}
where $\pi_1=2(\varepsilon-1)\int^{z_0}_1\frac{\operatorname{d}\! z}{z^{2/3}(z-1)^{2/3}}$ along $\sigma_1*\eta_1^{-1}$. 
\end{prop}

\begin{proof}
    From the matrix multiplication~\eqref{multi} we get 
    \begin{gather*}
        A_1(t)=\exp\left(-\int_{\eta_0}u_2-2\operatorname{Re}\left(\varepsilon\int_{\eta_0}Q_t\right)+\int_{\eta_0}u_1+2\operatorname{Re}\left(\int_{\eta_0}Q_t\right)\right)\cdot
        \\ \cdot\, \exp\left(2R^{1/3}r^{1/3}\operatorname{Re}(e^{i\varphi/3}3(\varepsilon-1))+\int_{\gamma_0}u_1\right)
    \end{gather*}
Now, using the notation $Q_t=R^{1/3}e^{i\varphi/3}\frac{\operatorname{d}\! z}{z^{2/3}(z-1)^{2/3}}$ from \eqref{Q}, we can compute the improper integral:
    \begin{equation*}
        \int_{\sigma_0} Q_t \approx R^{1/3}e^{i\varphi/3} \int_0^r \frac{\operatorname{d}\! z}{z^{2/3}}=3r^{1/3}
    \end{equation*}
So in the above formula, the real part of the exponent of $A_1(t)$ reads as:
    \begin{gather*}
        -2\operatorname{Re}\left(\varepsilon\int_{\eta_0}Q_t\right)+2\operatorname{Re}\left(\int_{\eta_0}Q_t\right)+2R^{1/3}\operatorname{Re}(e^{i\varphi/3}3r^{1/3}(\varepsilon-1))\approx
        \\ \approx R^{1/3}2\operatorname{Re}\left(e^{i\varphi/3}(\varepsilon-1)\int_{\eta_0^{-1}}\frac{\operatorname{d}\! z}{z^{2/3}(z-1)^{2/3}}+e^{i\varphi/3}(\varepsilon-1)\int_{\sigma_0}\frac{\operatorname{d}\! z}{z^{2/3}(z-1)^{2/3}}\right),
    \end{gather*}
where we used that $(1-\varepsilon)\int_{\eta_0}Q_t=(\varepsilon-1)\int_{\eta_0^{-1}}Q_t$. 
Now, using the notation $\pi_0$ given in the assertion and the fact that the concatenation $\sigma_0*\eta_0^{-1}$ is a path form $0$ to $z_0$, the formula for $A_1(t)$ becomes 
    \begin{gather*}
        A_1(t)\approx\operatorname{exp}\left(R^{1/3}\operatorname{Re}(e^{i\varphi/3}\pi_0)+\int_{\eta_0}(u_1-u_2)+\int_{\gamma_0}u_1\right)
    \end{gather*}
Similarly
    \begin{gather*}
        A_2(t)=\exp\left(-\int_{\eta_0}u_3-2\operatorname{Re}\left(\varepsilon^2\int_{\eta_0}Q_t\right)+\int_{\eta_0}u_2+2\operatorname{Re}\left(\varepsilon\int_{\eta_0}Q_t\right)\right)\cdot
        \\ \cdot\, \exp\left(2R^{1/3}r^{1/3}\operatorname{Re}(e^{i\varphi/3}3\varepsilon(\varepsilon-1))+\int_{\gamma_0}u_2\right)\approx
        \\ \approx\operatorname{exp}\left(R^{1/3}\operatorname{Re}(e^{i\varphi/3}\varepsilon\pi_0)+\int_{\eta_0}(u_2-u_3)+\int_{\gamma_0}u_2\right)
        \\ A_3(t)=\exp\left(-\int_{\eta_0}u_1-2\operatorname{Re}\left(\int_{\eta_0}Q_t\right)+\int_{\eta_0}u_3+2\operatorname{Re}\left(\varepsilon^2\int_{\eta_0}Q_t\right)\right)\cdot
        \\ \cdot\, \exp\left(2R^{1/3}r^{1/3}\operatorname{Re}(e^{i\varphi/3}3\varepsilon^2(\varepsilon-1))+\int_{\gamma_0}u_3\right)\approx
        \\ \approx \operatorname{exp}\left(R^{1/3}\operatorname{Re}(e^{i\varphi/3}\varepsilon^2\pi_0)+\int_{\eta_0}(u_3-u_1)+\int_{\gamma_0}u_3\right)
    \end{gather*}
The same computations go along for the elements of the $B_t$ matrix, with integrals on $\sigma_1*\eta_1^{-1}$. 
Notice that the results do not depend on the parameter $r$, as we obtained the complete improper integrals from $0$ and $1$ to $z_0$.
\end{proof}

\section{WKB asymptotics of the trace coordinates and proof of the Geometric P=W conjecture}\label{sec:proof}

Finally, in this section we will compute the asymptotic behaviour of the trace coordinates introduced in Section~\ref{sec:trace_coordinates} based on the approximations for the monodromy matrices determined in Section~\ref{sec:monodromy}. 
Curiously, we will find that a structure reminiscent of the Stokes phenomenon governs their behaviour. 
Namely, depending on sectors in the Hitchin base, different exponential terms of the coordinates dominate.
In Section~\ref{sec:reformulation}, the proof of Theorem~\ref{thm:main} is reduced to either Proposition~\ref{prop:main} or Proposition~\ref{prop:main2}. 
We prove Proposition~\ref{prop:main2} in Sections~\ref{sec:asymptotic}--\ref{sec:proof_main2}, thereby establishing Theorem~\ref{thm:main}. 
Section~\ref{sec:proof_main} provides a proof of Proposition~\ref{prop:main} under the extra condition Assumption~\ref{assn:Stokes}. 
It is included purely in order to provide a heuristic picture of the behavior of the Riemann--Hilbert map at infinity.

\subsection{A topological reformulation of the Geometric P=W conjecture}\label{sec:reformulation}
Remember from~\eqref{eq:polar} that the polar coordinates $(R,\varphi) \in \mathbb{R_+}\times S^1$ parameterize the Hitchin base (except for its origin), and the Dolbeault moduli space admits the Hitchin fibration over the Hitchin base. Now consider $1\ll R$ fixed, and $\varphi$ ranging over $[0,2\pi]$, parameterizing the circle $S^1_R$ of radius $R$ in the base. 
The Hitchin section provides an algebraic lift 
\[
    \sigma\colon S^1_R \to \mathcal{M}_{\operatorname{Dol}}(\alpha)
\]
of the loop $R e^{i\varphi}$ to the Dolbeault space for the Hitchin fibration $H$. 

\begin{prop}\label{prop:topological_reformulation}
 Assume that $\operatorname{dim}_{\mathbb{R}}\mathcal{M}_{\operatorname{Dol}}(\alpha) = 4$. Assume that for some $R\gg 1$ either 
\begin{enumerate}
    \item $S \circ \psi \circ \sigma$ induces an isomorphism 
    \[
    \mathbb{Z}\cong \pi_1 (S^1_R ) \to \pi_1 (\lvert\mathcal{D}\partial\mathcal{M}_B\rvert ) \cong \mathbb{Z} ;
    \]
    or 
    \item \label{item:regular_fibers}
    there exist regular values $z_1, z_2\in S^1$ of $H$ and $S$ respectively such that $H^{-1}(R z_1)$ and $S^{-1}(z_2)$ are both diffeomorphic to $T^2$ and that the inclusion map $H^{-1}(R z_1) \to \mathcal{M}_{\operatorname{Dol}}(\alpha)$ is homotopic to the inclusion map $S^{-1}(z_2)\to \mathcal{M}_{\operatorname{B}}(\textbf{c})$, where the two ambient spaces are identified via $\psi$. 
\end{enumerate}
    Then the diagram~\eqref{diagram:PW} is commutative up to homotopy. 
\end{prop}

\begin{remark}
    Part~\eqref{item:regular_fibers} generalizes to higher dimension via a framed cobordism argument, see~\cite[Theorem~5.4.1]{MauMazStev}. 
\end{remark}

\begin{proof}
    Assuming the first condition, the argument of~\cite[Section~5]{NemSz} applies verbatim. 
    It only depends on the condition that the moduli space is a $4$-manifold, not on the rank nor the type of singularities of the underlying Higgs bundles. 
    
    Assuming the second condition, the proof is again similar. 
    Namely, the spaces $\mathcal{M}_{\operatorname{Dol}}(\alpha) \setminus H^{-1}(B_R(0))$ and $\mathcal{M}_{\operatorname{B}}(\textbf{c})\setminus \psi(H^{-1}(B_R(0)))$ deformation retract to their boundaries $Y_{\operatorname{Dol}}$ and $Y_B$ respectively. 
    These spaces are compact $3$-manifolds, that are $T^2$-fibrations over $S^1$. 
    The map $\psi$ is a homeomorphism between them, so we may refer to either of them by $Y$.
    Now, the set $[Y, S^1]$ of homotopy classes of maps $Y\to S^1$ is in bijection with $H^1(Y, \mathbb{Z})$ because $S^1 = K (\mathbb{Z}, 1)$. 
    On the other hand, by definition 
    \[
        H^1(Y, \mathbb{Z}) \cong \operatorname{Hom}(H_1(Y, \mathbb{Z}), \mathbb{Z}), 
    \]
    so we get 
    \[
        [Y, S^1] = \operatorname{Hom}(H_1(Y, \mathbb{Z}), \mathbb{Z})
    \]
    as sets. 
    This bijection is given as follows: for a map $f\colon Y\to S^1$, take a regular value $z\in S^1$, and then we associate to $[f]$ the map 
    \[
    [a]\in H_1(Y, \mathbb{Z}) \mapsto [a] \cap [f^{-1} (z)] \in \mathbb{Z}, 
    \]
    where $\cap$ stands for algebraic intersection number. 
    This finishes the proof, because the algebraic intersection number is homotopy-invariant. 
\end{proof}

Notice that in this proof we do not even need algebraicity (or even analyticity) of $\sigma$. 
In view of Proposition~\ref{prop:topological_reformulation}, the proof of Theorem~\ref{thm:main} reduces to the either of the following. 

\begin{prop}\label{prop:main2}
There exist regular values $z_1, z_2\in S^1$ of $H$ and $S$ respectively such that $H^{-1}(R z_1)$ and $S^{-1}(z_2)$ are both diffeomorphic to $T^2$ and that the inclusion map $H^{-1}(R z_1) \to \mathcal{M}_{\operatorname{Dol}}(\alpha)$ is homotopic to the inclusion map $S^{-1}(z_2)\to \mathcal{M}_{\operatorname{B}}(\textbf{c})$
\end{prop}

\begin{prop}\label{prop:main}
The map $S \circ \psi \circ \sigma$ induces an isomorphism 
    \[
    \mathbb{Z}\cong \pi_1 (S^1_R ) \to \pi_1 (\lvert\mathcal{D}\partial\mathcal{M}_B\rvert ) \cong \mathbb{Z} 
    \]
    for $R\gg 1$. 
\end{prop}

Notice that the $S^1$ on the left-hand side does have a natural orientation coming from the complex structure of the Hitchin base, but the one on the right hand side does not have a preferred orientation.

We devote the rest of this section to proving these assertions. 
Namely, we will prove Proposition~\ref{prop:main} conditional to an extra Assumption~\ref{assn:Stokes}. 
We also show Proposition~\ref{prop:main2} unconditionally.

\subsection{Asymptotics of the trace coordinates}\label{sec:asymptotic}

\subsubsection{Trace coordinates}

We reformulate Proposition~\ref{lem:vanishing_coordinates} as: 
\begin{prop}
The only nonzero trace coordinates are $\operatorname{tr}(A_tB_t^{-1})$, $\operatorname{tr}(A_t^{-1}B_t)$ and $\operatorname{tr}(A_tB_tA_t^{-1}B_t^{-1})$.
\end{prop}
By the way, this also follows up to exponentially suppressed terms from the shape of the above determined matrices $A_t$ and $B_t$. 

For convenience, instead of $x_7, x_8, x_9$ appearing in Section~\ref{sec:char} we introduce the notations 
\begin{equation*}
    \begin{gathered}
        X_{R,\varphi}=\operatorname{tr}(A_{R,\varphi}B_{R,\varphi}^{-1})=\operatorname{tr}(A_tB_t^{-1})
        \\ Y_{R,\varphi}=\operatorname{tr}(A_{R,\varphi}^{-1}B_{R,\varphi})=\operatorname{tr}(A_t^{-1}B_t)
        \\ Z_{R,\varphi}=\operatorname{tr}(A_{R,\varphi}B_{R,\varphi}A_{R,\varphi}^{-1}B_{R,\varphi}^{-1})=\operatorname{tr}(A_tB_tA_t^{-1}B_t^{-1})
    \end{gathered}
\end{equation*}
(These coordinates are not to be confused with $X,Y,Z$ appearing  in Section~\ref{sec:char}, as those will no longer be used in the sequel.) 
We will refer to $X,Y,Z$ as trace coordinates. 
\begin{prop}\label{prop:analiticity}
$X_{R,\varphi}, Y_{R,\varphi}, Z_{R,\varphi}$ depend $\mathbb{R}$-analytically with $R,\varphi$. 
\end{prop}

\begin{proof}
It is known that $\psi$ is an $\mathbb{R}$-analytic map because it is a composition of the $\mathbb{R}$-analytic map $\operatorname{NAHC}$ and the $\mathbb{C}$-analytic map $\operatorname{RH}$. 
Moreover, the Hitchin section $\sigma$ is $\mathbb{C}$-algebraic, in particular $\mathbb{C}$-analytic too. 
Finally, the trace coordinates are clearly $\mathbb{C}$-algebraic too on the character variety, and we conclude. 
\end{proof}

\begin{prop}
We have the approximations for the three trace coordinates: 
\begin{equation}
    \begin{gathered}
        \label{coordinates}
        X_{R,\varphi} \approx\exp\left( R^{1/3}\operatorname{Re}(e^{i\varphi/3}(\pi_0-\pi_1)) + \int_{\eta_0 - \eta_1} u_3 - u_1 + \int_{\gamma_0 - \gamma_1} u_3 \right) + \\
        +\exp\left( R^{1/3}\operatorname{Re}(e^{i\varphi/3}\varepsilon(\pi_0-\pi_1)) + \int_{\eta_0 - \eta_1} u_1 - u_2 + \int_{\gamma_0 - \gamma_1} u_1 \right) + 
        \\ + \exp\left( R^{1/3}\operatorname{Re}(e^{i\varphi/3}\varepsilon^2(\pi_0-\pi_1)) + \int_{\eta_0 - \eta_1} u_2 - u_3 + \int_{\gamma_0 - \gamma_1} u_2 \right) , 
        \\ Y_{R,\varphi} \approx\exp\left( R^{1/3}\operatorname{Re}(e^{i\varphi/3}(\pi_1-\pi_0)) - \int_{\eta_0 - \eta_1} u_3 - u_1 - \int_{\gamma_0 - \gamma_1} u_3 \right) + \\ 
        +\exp\left( R^{1/3}\operatorname{Re}(e^{i\varphi/3}\varepsilon(\pi_1-\pi_0)) - \int_{\eta_0 - \eta_1} u_1 - u_2 - \int_{\gamma_0 - \gamma_1} u_1 \right) +
        \\ + \exp\left( R^{1/3}\operatorname{Re}(e^{i\varphi/3}\varepsilon^2(\pi_1-\pi_0)) - \int_{\eta_0 - \eta_1} u_2 - u_3 - \int_{\gamma_0 - \gamma_1} u_2 \right) , 
        \\ Z_{R,\varphi} \approx\exp\left( R^{1/3}\operatorname{Re}(e^{i\varphi/3}(1-\varepsilon)(\pi_0-\pi_1)) + \int_{\eta_0 - \eta_1} 2 u_3 - u_1 - u_2 + \int_{\gamma_0 - \gamma_1} u_3  - u_2 \right) +
        \\ + \exp\left( R^{1/3}\operatorname{Re}(e^{i\varphi/3}(1-\varepsilon)\varepsilon(\pi_0-\pi_1)) + \int_{\eta_0 - \eta_1} 2 u_1 - u_2 - u_3 + \int_{\gamma_0 - \gamma_1} u_1  - u_3  \right) + \\ +
         \exp\left( R^{1/3}\operatorname{Re}(e^{i\varphi/3}(1-\varepsilon)\varepsilon^2(\pi_0-\pi_1)) + \int_{\eta_0 - \eta_1} 2 u_2 - u_3 - u_1 + \int_{\gamma_0 - \gamma_1} u_2  - u_1 \right)
    \end{gathered}
\end{equation}
\end{prop}

\begin{proof}
Straightforward matrix multiplications using Proposition~\ref{prop:asymptotic_monodromy}. 
\end{proof}

\begin{remark}\label{rem:argument_period}
\begin{enumerate}
    \item We have 
    \[
    \pi_0-\pi_1 = 2(\varepsilon-1) \int_0^1\frac{\operatorname{d}\! t}{t^{2/3}(t-1)^{2/3}} \neq 0 
    \]
    because the integrand is positive. In addition, we have 
    \[
    \operatorname{arg} (\pi_0-\pi_1) = \operatorname{arg} (\varepsilon-1 ) = \frac{5\pi}{6}. 
    \] 
    \item The terms $\int_{\eta_0}u_{i}-u_{i+1}$, and $\int_{\gamma_0}u_i$ $i=1,2,3$ are purely imaginary, thus they have no contribution to the absolute values of the coordinates. 
    On the other hand, they affect the behaviour of the phases of the coordinates of Betti space. 
\end{enumerate}
\end{remark}

\subsubsection{Maximal dilation exponents of the trace coordinates}

According to Proposition~\ref{prop:topological_reformulation} there remains to answer the question: where does the Riemann--Hilbert correspondence map the loop $S^1_R$, as $R\rightarrow\infty$? 

For fixed $1\ll R$, and $\varphi$ ranging over $[0,2\pi]$ the coordinates $X_{R,\varphi}$, $Y_{R,\varphi}$, $Z_{R,\varphi}$ describe the image of the above mentioned loop $t = Re^{i\varphi}$ in the Hitchin base~\eqref{eq:Hitchin_base} under the Riemann--Hilbert map. To prove the statement of the Geometric P=W conjecture, we next need to investigate the growth order of the coordinates $X_{R,\varphi}$, $Y_{R,\varphi}$, $Z_{R,\varphi}$ with respect to each other.

Introducing the notation 
\begin{equation}\label{eq:x}
    x=e^{i\varphi/3}(\pi_0-\pi_1)=a+b\sqrt{-1}, 
\end{equation} 
we have 
\begin{equation}\label{eq:Real_parts}
    \begin{gathered}
        \operatorname{Re}(x)=a,\hspace{0.4cm} \operatorname{Re}(\varepsilon x)=-\frac{1}{2}a-\frac{\sqrt{3}}{2}b,\hspace{0.4cm} \operatorname{Re}(\varepsilon^2x)=-\frac{1}{2}a+\frac{\sqrt{3}}{2}b
        \\ \operatorname{Re}(-x)=-a,\hspace{0.4cm} \operatorname{Re}(-\varepsilon x)=\frac{1}{2}a+\frac{\sqrt{3}}{2}b,\hspace{0.4cm} \operatorname{Re}(-\varepsilon^2x)=-\frac{1}{2}a-\frac{\sqrt{3}}{2}b
        \\ \operatorname{Re}((1-\varepsilon)x)=\frac{3}{2}a+\frac{\sqrt{3}}{2}b,\hspace{0.2cm} \operatorname{Re}((1-\varepsilon)\varepsilon x)=-\sqrt{3}b,\hspace{0.2cm} \operatorname{Re}((1-\varepsilon)\varepsilon^2x)=-\frac{3}{2}a+\frac{\sqrt{3}}{2}b
    \end{gathered}
\end{equation}
Consider $\lvert X_{R,\varphi}\rvert$ from \eqref{coordinates}. It consists of three exponential terms. We are interested in the dominance properties among these terms as $R\rightarrow\infty$, and similarly for $\lvert Y_{R,\varphi}\rvert, \lvert Z_{R,\varphi}\rvert$. 
Recall from Lemma~\ref{lem:noncritical} that the critical angles of the first kind are integer multiples of $\pi$. 
\begin{prop}\label{prop:dilation}
Fix any $\varphi\in (0, \pi ) \cup (\pi, 2\pi ) $. As $R\to\infty$, we have 
\begin{align*}
   \frac 1R \ln \lvert X_{R,\varphi}\rvert & \to \max (\operatorname{Re}(x) , \operatorname{Re}(\varepsilon x), \operatorname{Re}(\varepsilon^2x)), \\
   \frac 1R \ln \lvert Y_{R,\varphi}\rvert & \to \max (\operatorname{Re}(-x) , \operatorname{Re}(-\varepsilon x) ,  \operatorname{Re}(-\varepsilon^2x)), \\
   \frac 1R \ln \lvert Z_{R,\varphi}\rvert & \to \max (\operatorname{Re}((1-\varepsilon)x) , \operatorname{Re}((1-\varepsilon)\varepsilon x) , \operatorname{Re}((1-\varepsilon)\varepsilon^2x)). 
\end{align*}
\end{prop}

\begin{proof}
This is a straightforward consequence of the formulas~\eqref{coordinates} and~\eqref{eq:Real_parts}. 
\end{proof}

\begin{remark}
    The proposition indicates that the affine coordinates of~\eqref{eq:nodal_curve} near its node, namely the quotients $X/Z$ and $Y/Z$, obey the rules of the tropical (i.e., max-plus) semi-ring. 
    See~\cite[Conjecture~1]{SzSz}. 
\end{remark}

\subsubsection{Sectorial decomposition}
Proposition~\ref{prop:dilation} leads us to introduce 12 open sectors in the complex plane parameterized by $x = a + b\sqrt{-1}$, determined by the terms that achieve the respective maxima, and defined as follows (see also Figure~\ref{figure:sectors}):
\begin{equation*}
    S_j=\left\{x\in\mathbb{C}\colon \quad \frac{\pi(j-1)}{6}<\operatorname{arg}(x)<\frac{\pi j}{6}\right\}, \hspace{0.3cm} j=1,2,\ldots ,12
\end{equation*}
Notice that by~\eqref{eq:x} and Remark~\ref{rem:argument_period}, one has 
\begin{equation}\label{eq:argx}
        \operatorname{arg}(x) = \frac{\varphi}3 + \frac{5\pi}{6}. 
\end{equation}
As we see on Figure~\ref{figure:sectors}, the dividing lines between two neighboring sectors are the lines $b=\pm\sqrt{3}a$, $b=\pm\frac{1}{\sqrt{3}}a$ and the two axes. 
Define $R_j$ to be the ray between sectors $S_j$ and $S_{j+1}$ (where $j+1$ is understood modulo 12). We will refer to these as the \emph{Stokes rays}, and the angles $\varphi$ defining them as \emph{critical angles of the second kind}. 
Namely, for the $j$'th critical angle of the second kind $\varphi_j$ we have 
\[
    \frac{\varphi_j}3 = \frac{\pi j}{6}, 
\]
equivalently 
\[
    \varphi_j= \frac{\pi j}{2}.  
\]
Crucially, any critical of the first kind (see~\eqref{eq:critical_angles}) is also critical of the second kind. 
From now on, \emph{critical} will stand for critical of the second kind. 

\subsubsection{Dominance of terms within the trace coordinates}
Since our aim is to analyze the asymptotic behaviour of the coordinates, we will make use of the usual concept of dominance: we say that function $F(R)$ dominates the function $G(R)$ if and only if $\frac{\lvert G(R)\rvert}{\lvert F(R)\rvert}\rightarrow 0$ as $R\rightarrow\infty$. 
We denote this relation by $F(R) \gg G(R)$. 
Crucially, by Theorem~\ref{thm:dilation_spectrum}, all such convergences to follow are exponentially suppressed with $R^{1/3}$.

\begin{prop}\label{prop:dominance}
  Fix a non-critical $\varphi\in [0, 2\pi ]$, so that $x$ belongs to one of the open sectors $S_j$. 
  The dominant term of the coordinates $X,Y,Z$ in each of the sectors $S_j$ is given by the following table, where the notation $\approx$ implies convergent asymptotic expansions. 
\end{prop}
\begin{center}
\begin{tabular}{c | c} 
 \hline
 Sector & Dominant term of the trace coordinates \\  
 \hline
 $S_1$ & $\lvert X_{R,\varphi}\rvert\approx\exp(R^{1/3}\operatorname{Re}(x))=\exp(R^{1/3}a)$ \\ & $\lvert Y_{R,\varphi}\rvert\approx\exp(R^{1/3}\operatorname{Re}(-\varepsilon x))=\exp(R^{1/3}(\frac{1}{2}a+\frac{\sqrt{3}}{2}b))$ \\ & $\lvert Z_{R,\varphi}\rvert\approx\exp(R^{1/3}\operatorname{Re}((1-\varepsilon) x))=\exp(R^{1/3}(\frac{3}{2}a+\frac{\sqrt{3}}{2}b))$ \\ 
 \hline
 $S_2$ & $\lvert X_{R,\varphi}\rvert\approx\exp(R^{1/3}\operatorname{Re}(x))=\exp(R^{1/3}a)$ \\ & $\lvert Y_{R,\varphi}\rvert\approx\exp(R^{1/3}\operatorname{Re}(-\varepsilon x))=\exp(R^{1/3}(\frac{1}{2}a+\frac{\sqrt{3}}{2}b))$ \\ & $\lvert Z_{R,\varphi}\rvert\approx\exp(R^{1/3}\operatorname{Re}((1-\varepsilon) x))=\exp(R^{1/3}(\frac{3}{2}a+\frac{\sqrt{3}}{2}b))$ \\ 
 \hline
 $S_3$ & $\lvert X_{R,\varphi}\rvert\approx\exp(R^{1/3}\operatorname{Re}(\varepsilon^2x))=\exp(R^{1/3}(-\frac{1}{2}a+\frac{\sqrt{3}}{2}b))$ \\ & $\lvert Y_{R,\varphi}\rvert\approx\exp(R^{1/3}\operatorname{Re}(-\varepsilon x))=\exp(R^{1/3}(\frac{1}{2}a+\frac{\sqrt{3}}{2}b))$ \\ & $\lvert Z_{R,\varphi}\rvert\approx\exp(R^{1/3}\operatorname{Re}((1-\varepsilon) x))=\exp(R^{1/3}(\frac{3}{2}a+\frac{\sqrt{3}}{2}b))$ \\ 
 \hline
 $S_4$ & $\lvert X_{R,\varphi}\rvert\approx\exp(R^{1/3}\operatorname{Re}(\varepsilon^2x))=\exp(R^{1/3}(-\frac{1}{2}a+\frac{\sqrt{3}}{2}b))$ \\ & $\lvert Y_{R,\varphi}\rvert\approx\exp(R^{1/3}\operatorname{Re}(-\varepsilon x))=\exp(R^{1/3}(\frac{1}{2}a+\frac{\sqrt{3}}{2}b))$ \\ & $\lvert Z_{R,\varphi}\rvert\approx\exp(R^{1/3}\operatorname{Re}((1-\varepsilon)\varepsilon^2 x))=\exp(R^{1/3}(-\frac{3}{2}a+\frac{\sqrt{3}}{2}b))$ \\
 \hline
 $S_5$ & $\lvert X_{R,\varphi}\rvert\approx\exp(R^{1/3}\operatorname{Re}(\varepsilon^2x))=\exp(R^{1/3}(-\frac{1}{2}a+\frac{\sqrt{3}}{2}b))$ \\ & $\lvert Y_{R,\varphi}\rvert\approx\exp(R^{1/3}\operatorname{Re}(- x))=\exp(R^{1/3}(-a))$ \\ & $\lvert Z_{R,\varphi}\rvert\approx\exp(R^{1/3}\operatorname{Re}((1-\varepsilon)\varepsilon^2 x))=\exp(R^{1/3}(-\frac{3}{2}a+\frac{\sqrt{3}}{2}b))$ \\  
 \hline
 $S_6$ & $\lvert X_{R,\varphi}\rvert\approx\exp(R^{1/3}\operatorname{Re}(\varepsilon^2x))=\exp(R^{1/3}(-\frac{1}{2}a+\frac{\sqrt{3}}{2}b))$ \\ & $\lvert Y_{R,\varphi}\rvert\approx\exp(R^{1/3}\operatorname{Re}(- x))=\exp(R^{1/3}(-a))$ \\ & $\lvert Z_{R,\varphi}\rvert\approx\exp(R^{1/3}\operatorname{Re}((1-\varepsilon)\varepsilon^2 x))=\exp(R^{1/3}(-\frac{3}{2}a+\frac{\sqrt{3}}{2}b))$ \\  
 \hline
 $S_7$ & $\lvert X_{R,\varphi}\rvert\approx\exp(R^{1/3}\operatorname{Re}(\varepsilon x))=\exp(R^{1/3}(-\frac{1}{2}a-\frac{\sqrt{3}}{2}b))$ \\ & $\lvert Y_{R,\varphi}\rvert\approx\exp(R^{1/3}\operatorname{Re}(- x))=\exp(R^{1/3}(-a))$ \\ & $\lvert Z_{R,\varphi}\rvert\approx\exp(R^{1/3}\operatorname{Re}((1-\varepsilon)\varepsilon^2 x))=\exp(R^{1/3}(-\frac{3}{2}a+\frac{\sqrt{3}}{2}b))$ \\  
 \hline
 $S_8$ & $\lvert X_{R,\varphi}\rvert\approx\exp(R^{1/3}\operatorname{Re}(\varepsilon x))=\exp(R^{1/3}(-\frac{1}{2}a-\frac{\sqrt{3}}{2}b))$ \\ & $\lvert Y_{R,\varphi}\rvert\approx\exp(R^{1/3}\operatorname{Re}(- x))=\exp(R^{1/3}(-a))$ \\ & $\lvert Z_{R,\varphi}\rvert\approx\exp(R^{1/3}\operatorname{Re}((1-\varepsilon)\varepsilon x))=\exp(R^{1/3}(-\sqrt{3}b))$ \\  
 \hline
 $S_9$ & $\lvert X_{R,\varphi}\rvert\approx\exp(R^{1/3}\operatorname{Re}(\varepsilon x))=\exp(R^{1/3}(-\frac{1}{2}a-\frac{\sqrt{3}}{2}b))$ \\ & $\lvert Y_{R,\varphi}\rvert\approx\exp(R^{1/3}\operatorname{Re}(-\varepsilon^2 x))=\exp(R^{1/3}(\frac{1}{2}a-\frac{\sqrt{3}}{2}b))$ \\ & $\lvert Z_{R,\varphi}\rvert\approx\exp(R^{1/3}\operatorname{Re}((1-\varepsilon)\varepsilon x))=\exp(R^{1/3}(-\sqrt{3}b))$ \\  
 \hline
 $S_{10}$ & $\lvert X_{R,\varphi}\rvert\approx\exp(R^{1/3}\operatorname{Re}(\varepsilon x))=\exp(R^{1/3}(-\frac{1}{2}a-\frac{\sqrt{3}}{2}b))$ \\ & $\lvert Y_{R,\varphi}\rvert\approx\exp(R^{1/3}\operatorname{Re}(-\varepsilon^2 x))=\exp(R^{1/3}(\frac{1}{2}a-\frac{\sqrt{3}}{2}b))$ \\ & $\lvert Z_{R,\varphi}\rvert\approx\exp(R^{1/3}\operatorname{Re}((1-\varepsilon)\varepsilon x))=\exp(R^{1/3}(-\sqrt{3}b))$ \\  
 \hline
 $S_{11}$ & $\lvert X_{R,\varphi}\rvert\approx\exp(R^{1/3}\operatorname{Re}( x))=\exp(R^{1/3}(a))$ \\ & $\lvert Y_{R,\varphi}\rvert\approx\exp(R^{1/3}\operatorname{Re}(-\varepsilon^2 x))=\exp(R^{1/3}(\frac{1}{2}a-\frac{\sqrt{3}}{2}b))$ \\ & $\lvert Z_{R,\varphi}\rvert\approx\exp(R^{1/3}\operatorname{Re}((1-\varepsilon)\varepsilon x))=\exp(R^{1/3}(-\sqrt{3}b))$ \\  
 \hline
 $S_{12}$ & $\lvert X_{R,\varphi}\rvert\approx\exp(R^{1/3}\operatorname{Re}( x))=\exp(R^{1/3}(a))$ \\ & $\lvert Y_{R,\varphi}\rvert\approx\exp(R^{1/3}\operatorname{Re}(-\varepsilon^2 x))=\exp(R^{1/3}(\frac{1}{2}a-\frac{\sqrt{3}}{2}b))$ \\ & $\lvert Z_{R,\varphi}\rvert\approx\exp(R^{1/3}\operatorname{Re}((1-\varepsilon) x))=\exp(R^{1/3}(\frac{3}{2}a+\frac{\sqrt{3}}{2}b))$ \\  
 \hline
\end{tabular}
\end{center}

\subsubsection{Dominance between different trace coordinates within sectors}

Next, we need to address the dominance properties among the three trace coordinates. 
\begin{lemma}\label{lem:dominance}
We fix some angle coordinate that is not critical of the first kind, $\varphi\notin \mathbb{Z} \cdot \pi$. 
\begin{itemize}
    \item[i)] $\lvert Z_{R,\varphi}\rvert\gg\lvert X_{R,\varphi}\rvert$ and $\lvert Z_{R,\varphi}\rvert\gg\lvert Y_{R,\varphi}\rvert$. 
    \item[ii)] On sectors $S_j$, $j=1,4,5,8,9,12$ $\lvert X_{R,\varphi}\rvert\gg\lvert Y_{R,\varphi}\rvert$, while on sectors $S_j$, $j=2,3,6, 7,10,11$ $\lvert Y_{R,\varphi}\rvert\gg\lvert X_{R,\varphi}\rvert$.
\end{itemize}
\end{lemma}

\begin{figure}[ht]
\centering
\includegraphics[width=7.0cm]{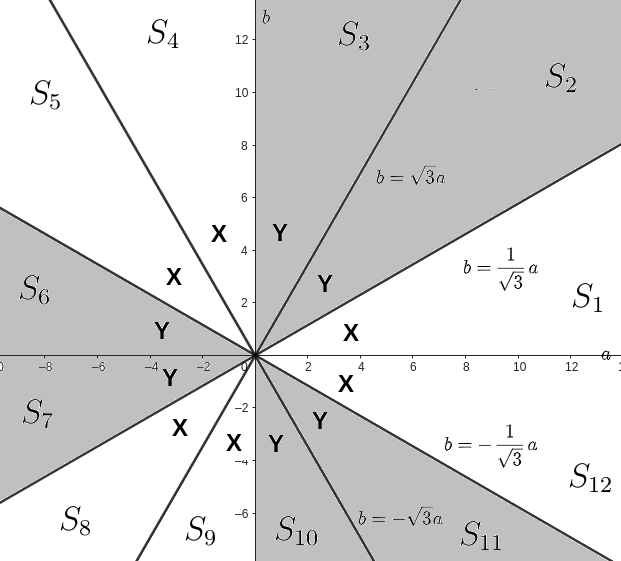}
\caption{The $\tau$ complex plane, sectors $S_j$, dominant coordinate $\lvert X_{R,\varphi}\rvert$ or $\lvert Y_{R,\varphi}\rvert$ depending on the sector.}
\label{figure:sectors}
\end{figure}

\begin{proof}
    A straightforward comparison of the exponents from the table in each sector shows this result. Namely: 
    \begin{enumerate}
        \item On $S_1$: $\frac{3}{2}a+\frac{\sqrt{3}}{2}b>a>\frac{1}{2}a+\frac{\sqrt{3}}{2}b$, that is $\lvert Z_{R,\varphi}\rvert \gg \lvert X_{R,\varphi}\rvert \gg \lvert Y_{R,\varphi}\rvert$.
        \item On $S_2$: $\frac{3}{2}a+\frac{\sqrt{3}}{2}b>\frac{1}{2}a+\frac{\sqrt{3}}{2}b>a$, that is $\lvert Z_{R,\varphi}\rvert \gg \lvert Y_{R,\varphi}\rvert \gg \lvert X_{R,\varphi}\rvert$. 
        \item   On $S_3$: $\frac{3}{2}a+\frac{\sqrt{3}}{2}b>\frac{1}{2}a+\frac{\sqrt{3}}{2}b>-\frac{1}{2}a+\frac{\sqrt{3}}{2}b$, that is $\lvert Z_{R,\varphi}\rvert \gg \lvert Y_{R,\varphi}\rvert \gg \lvert X_{R,\varphi}\rvert$.

    \item On $S_4$: $-\frac{3}{2}a+\frac{\sqrt{3}}{2}b>-\frac{1}{2}a+\frac{\sqrt{3}}{2}b>\frac{1}{2}a+\frac{\sqrt{3}}{2}b$, that is $\lvert Z_{R,\varphi}\rvert \gg \lvert X_{R,\varphi}\rvert \gg \lvert Y_{R,\varphi}\rvert$.

    \item On $S_5$: $-\frac{3}{2}a+\frac{\sqrt{3}}{2}b>-\frac{1}{2}a+\frac{\sqrt{3}}{2}b>-a$, that is $\lvert Z_{R,\varphi}\rvert \gg \lvert X_{R,\varphi}\rvert \gg \lvert Y_{R,\varphi}\rvert$.

    \item On $S_6$: $-\frac{3}{2}a+\frac{\sqrt{3}}{2}b>-a>-\frac{1}{2}a+\frac{\sqrt{3}}{2}b$, that is $\lvert Z_{R,\varphi}\rvert \gg \lvert Y_{R,\varphi}\rvert \gg \lvert X_{R,\varphi}\rvert$.

    \item On $S_7$: $-\frac{3}{2}a+\frac{\sqrt{3}}{2}b>-a>-\frac{1}{2}a-\frac{\sqrt{3}}{2}b$, that is $\lvert Z_{R,\varphi}\rvert \gg \lvert Y_{R,\varphi}\rvert \gg \lvert X_{R,\varphi}\rvert$.

    \item On $S_8$: $-\sqrt{3}b>-\frac{1}{2}a-\frac{\sqrt{3}}{2}b>-a$, that is $\lvert Z_{R,\varphi}\rvert \gg \lvert X_{R,\varphi}\rvert \gg \lvert Y_{R,\varphi}\rvert$.

    \item On $S_9$: $-\sqrt{3}b>-\frac{1}{2}a-\frac{\sqrt{3}}{2}b> \frac{1}{2}a-\frac{\sqrt{3}}{2}b$, that is $\lvert Z_{R,\varphi}\rvert \gg \lvert X_{R,\varphi}\rvert \gg \lvert Y_{R,\varphi}\rvert$.

    \item On $S_{10}$: $-\sqrt{3}b>\frac{1}{2}a-\frac{\sqrt{3}}{2}b> -\frac{1}{2}a-\frac{\sqrt{3}}{2}b$, that is $\lvert Z_{R,\varphi}\rvert \gg \lvert Y_{R,\varphi}\rvert \gg \lvert X_{R,\varphi}\rvert$.

    \item On $S_{11}$: $-\sqrt{3}b>\frac{1}{2}a-\frac{\sqrt{3}}{2}b>a$, that is $\lvert Z_{R,\varphi}\rvert \gg \lvert Y_{R,\varphi}\rvert \gg \lvert X_{R,\varphi}\rvert$.

    \item On $S_{12}$: $\frac{3}{2}a+\frac{\sqrt{3}}{2}b>a>\frac{1}{2}a-\frac{\sqrt{3}}{2}b$, that is $\lvert Z_{R,\varphi}\rvert \gg \lvert X_{R,\varphi}\rvert \gg \lvert Y_{R,\varphi}\rvert$.
\end{enumerate}
\end{proof}

We need to understand the behaviour of $X_{R,\varphi},Y_{R,\varphi},Z_{R,\varphi}$ as $1\ll R$ is fixed, and $\varphi$ ranges  over $[0,2\pi]$. Lemma~\ref{lem:dominance} shows that $\lvert Z_{R,\varphi}\rvert$ dominates the other two coordinates on each open sector.
\begin{prop}
As $t\rightarrow\infty$ in any of the sectors $S_j$, the image of $\psi$ converges to the point $[0:0:1]$. 
\end{prop}
Notice that this limit point is actually the nodal point of the compactifying curve $C$ determined in Lemma~\ref{lem:compactifying_curve}, namely 
\begin{equation}\label{eq:nodal_curve}
    X^3+Y^3-XYZ=0. 
\end{equation}

\subsection{Proof of Proposition~\ref{prop:main2}}\label{sec:proof_main2}

\subsubsection{Blow-up of the nodal point}\label{sec:blowup}

Let us consider the blow up of the projective compactification of $\mathcal{M}_{\operatorname{B}}$ at the point $[0:0:1]$:
\[
    \widetilde{\mathcal{M}_{\operatorname{B}}} \to \overline{\mathcal{M}_{\operatorname{B}}}. 
\]
Locally around its nodal point $[0:0:1]$, the curve $C$ consists of two components, tangential to the $X$ and $Y$ axes respectively, in the local coordinate system $[X\colon Y\colon Z]$. 
Local affine coordinates near this point are given by $X/Z$ and $Y/Z$. 
Taking $\mathbb{C} P^1$ with homogeneous coordinates $[X'\colon Y']$, the equation of the blown up surface reads as 
\[
\frac XZ Y' - \frac YZ X' = 0.
\]
The inverse image of $(0,0)$ is the exceptional divisor, denoted by $C_2$. 
The proper transform of $C$ is denoted by $C_1$. 
Clearly, $C_1$ and $C_2$ have two intersection points. 
Denote the intersection point of $C_2$ with the proper transform of the $X$ axis by $P_1$, and the one with the $Y$ axis by $P_2$. 
The point $P_1$ is contained in the affine chart $U_1 = \{ Y' \neq 0 \}$.  
On $U_1$, the local affine coordinates are given by 
\[
    \frac XZ, \quad \frac XY. 
\]
Similarly, $P_2$ is contained in the affine chart $U_2 = \{ X' \neq 0 \}$, and local affine coordinates are given by 
\[
    \frac YZ, \quad \frac YX. 
\]

The component $C_i$ gives rise to the vertex $V_i$ in the dual graph, and the two intersection points $P_1, P_2$ correspond to edges $E_1$ and $E_2$ respectively. 
Then the dual boundary graph becomes as on Figure~\ref{figure:dual}. 
\begin{figure}[ht]
\centering
\includegraphics[width=10.0cm]{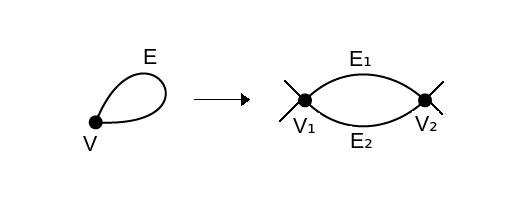}
\caption{The transformation of the graph of the dual boundary complex under the blow up. The vertex $V$ transforms into $V_1$, and $V_2$ corresponds to the appearing exceptional divisor.}
\label{figure:dual}
\end{figure}

\subsubsection{Simpson's map}\label{sec:Simpsons_map}

We are ready to define Simpson's natural map (see~\cite[Section~5,Equations~(41)--(42)]{SzSz2}, also called evaluation map in~\cite[Assumption~4.2.5]{MauMazStev})
\begin{equation}\label{eq:Simpson_map}
    S\colon  \mathcal{M}_{\operatorname{B}}\setminus \psi(h^{-1}(B_R(0)))\to \lvert\mathcal{D}\partial\mathcal{M}_B\rvert    
\end{equation}
as follows.
We consider the compactification $\widetilde{\mathcal{M}}_{\operatorname{B}}$ of $\mathcal{M}_{\operatorname{B}}$ defined in Section~\ref{sec:blowup}. 
Next, we construct a cover 
\begin{equation}\label{eq:cover}
    \mathcal{M}_{\operatorname{B}}\setminus \psi(h^{-1}(B_R(0))) = N_0^*\cup N_1^*\cup N_2^*
\end{equation}
of the range by connected open sets $N_0^*, N_1^*, N_2^*$ satisfying the properties
\begin{enumerate} 
    \item $N_i^*\cap N_j^*$ is non-empty and connected for every $0\leq i,j \leq 2$
    \item $N_0^* \cap N_1^* \cap N_2^* = \emptyset$. 
\end{enumerate}
See Figure~\ref{fig:Simpson}. 
Namely, let $N_0$ be a tubular neighborhood of $C_1$ in  $\widetilde{\mathcal{M}}_{\operatorname{B}}$ and $N_1, N_2$ be suitable neighborhoods of the intersection points $P_1, P_2$ of $C_1, C_2$ defined above. 
Let 
\[
    N_i^* = N_i\cap \mathcal{M}_{\operatorname{B}} = N_i \setminus (C_1 \cup C_2). 
\]
Let $\phi_i$ denote a partition of unity subordinate to the cover~\eqref{eq:cover}. 
We then define the map~\eqref{eq:Simpson_map} mapping $P$ to $(\phi_0 (P), \phi_1 (P), \phi_2 (P))$. 
The range of $S$ is naturally a subset of the unit cube $[0,1]^3$ with coordinates $(\phi_0, \phi_1, \phi_2 )$, and it is easy to see that the conditions imposed on~\eqref{eq:cover} imply that its image $\mathcal{D}\partial\mathcal{M}_B(\textbf{c})$ is the boundary of the standard $2$-simplex 
\[
    \phi_1 + \phi_2 + \phi_3 = 1. 
\]
Its edges are defined by the $\phi_i = 0$ plane sections. 

\subsubsection{Homology basis}

In this section, we will define a basis of $H_1(\widetilde{\Sigma}, \mathbb{Z})$. 
Consider the straight line segment $[r, 1-r]\subset \mathbb{R}$.
Let us temporarily denote its lifts to $\Sigma$ by $f_1, f_2, f_3$, where $f_j$ lies on the sheet $\Sigma_j$ determined by the $Q_j$-eigenspace of $\theta$ (see~\eqref{Q}).  
On the other hand, the three lifts of the loop $\gamma_0$ cyclically connect the three sheets. 
For instance, one of these lifts connects $\Sigma_1$ to $\Sigma_2$. 
Denote this lift by $\tilde{\gamma}_0$. 
Moreover, let $\tilde{\tilde{\gamma}}_0$ denote the lift of $\gamma_0$ connecting the sheet $\Sigma_2$ to $\Sigma_3$.
Similarly, one of the lifts $\tilde{\gamma}_1$ of $\gamma_1$ connects the sheet $\Sigma_1$ to $\Sigma_2$, and another one $\tilde{\tilde{\gamma}}_1$ connects $\Sigma_2$ to $\Sigma_3$. 
It follows that the paths $f_1 * \tilde{\gamma}_0 * (-f_2) * (-\tilde{\gamma}_1)$ and $f_2 * \tilde{\tilde{\gamma}}_0 * (-f_3) * (-\tilde{\tilde{\gamma}}_1)$ represent singular $1$-cycles on $\widetilde{\Sigma}$. 
Let us denote these cycles in order by $A,B$. 
See Figure~\ref{fig:closed_loop}. 

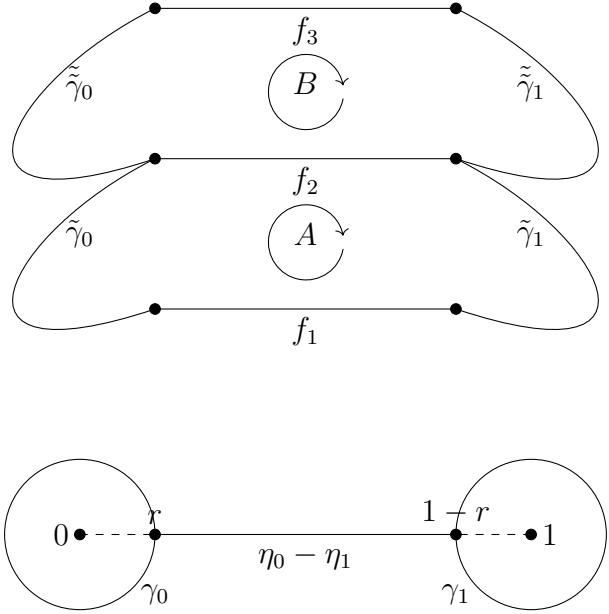
\begin{figure}
 \centering
 \begin{tikzpicture}
 \filldraw [black] (-3,-3) circle (2pt) node [anchor = east] {$0$};
 \filldraw [black] (-2,-3) circle (2pt) node [anchor = south] {$r$}; 
 \filldraw [black] (3,-3) circle (2pt) node [anchor = west] {$1$};
 \filldraw [black] (2,-3) circle (2pt) node [anchor = south] {$1-r$}; 
 \draw (0, -3.3) node {$\eta_0 - \eta_1$};
 \draw (-2, -3.8) node {$\gamma_0$};
 \draw (2, -3.8) node {$\gamma_1$};
 \draw [dashed] (-3,-3) -- (-2,-3);
 \draw (-2,-3) --  (2,-3);
 \draw (-3,-3) circle [radius = 1];
 \draw (3,-3) circle [radius = 1];
 \draw [dashed] (3,-3) -- (2,-3);
 \draw  (-2,0) --  (2,0);
 \draw (0, -.3) node {$f_1$};
 \filldraw [black] (-2,0) circle (2pt);
 \filldraw [black] (2,0) circle (2pt);  
 \draw  (-2,2) --  (2,2);
 \draw (0, 1.7) node {$f_2$};
 \filldraw [black] (-2,2) circle (2pt);
 \filldraw [black] (2,2) circle (2pt);  
 \draw  (-2,4) --  (2,4);
 \draw (0, 3.7) node {$f_3$};
 \filldraw [black] (-2,4) circle (2pt);
 \filldraw [black] (2,4) circle (2pt);  
 \draw (-2,0) .. controls  (-5,-1) and (-4,1) .. (-2,2);
 \draw (-3, 1) node {$\tilde{\gamma}_0$};
 \draw (2,0) .. controls  (5,-1) and (4,1) .. (2,2);
 \draw (3, 1) node {$\tilde{\gamma}_1$};
 \draw (-2,2) .. controls  (-5,1) and (-4,3) .. (-2,4);
 \draw (-3, 3) node {$\tilde{\tilde{\gamma}}_0$};
 \draw (2,2) .. controls  (5,1) and (4,3) .. (2,4);
 \draw (3, 3) node {$\tilde{\tilde{\gamma}}_1$};
 \draw (0,1) node {$A$};
 \draw [->] (.5,.8) arc [radius = .5, start angle = 350, end angle = 10];  
 \draw (0,3) node {$B$};
 \draw [->] (.5,2.8) arc [radius = .5, start angle = 350, end angle = 10];  
\end{tikzpicture}
 \caption{The cycles $A,B$ on the spectral curve.}
\label{fig:closed_loop}
\end{figure}  

\begin{lemma}\label{lem:basis}
The cycles $A,B$ form a $\mathbb{Z}$-basis of $H_1 (\widetilde{\Sigma}, \mathbb{Z})$. 
\end{lemma}

\begin{proof}
This is likely standard, yet we include it for sake of completeness. 
Let us consider straight lines between the marked points $0,1,\infty \in \mathbb{C} P^1 = S^2$. 
They form a triangulation of $S^2$ with $3$ vertices, $3$ edges and $2$ faces.
Up to homeomorphism, $\widetilde{\Sigma}_t$ can be obtained by gluing three copies of such triangulated $S^2$, by identifying their points $0$ with each other, and similarly for their points $1$ and $\infty$. 
The segments $[0,1]$ in the three copies of $S^2$ are (extensions of) the lifts $f_1, f_2, f_3$. 
The resulting triangulation of $T^2$ 
\begin{itemize}
\item admits $3$ vertices $v_1, v_2, v_3$, $9$ edges $e_1, \ldots , e_9$ and $6$ faces 
\item has $6$ edges incident to every vertex,
\item has $3$ edges between each pair of distinct vertices, 
\item has three differently labelled vertices on each face. 
\end{itemize}
An example of such a triangulation is shown on Figure~\ref{fig:triangulation}.
Now, if we label vertices so that $v_1$ corresponds to the identification of the points $0\in S^2$ and $v_2$ to $1\in S^2$, then the three lifts $f_1, f_2, f_3$ correspond to the three edges of the triangulation connecting $v_1, v_2$. 
(On Figure~\ref{fig:triangulation}, these edges are labelled by $e_2, e_4, e_9$. 
The concrete bijection between $f_1, f_2, f_3$ and $e_2, e_4, e_9$ is irrelevant for the argument.) 
The loop $A$ is the concatenation of $f_1, f_2$, and $B$ is the concatenation of $f_2, f_3$ (given suitable orientations so that the concatenations make sense). 
Now, since $f_1$ and $f_3$ are not parallel (i.e., they do not form a bigon), the cycles $A,B$ are linearly independent over $\mathbb{Q}$, and generate $H_1 (\widetilde{\Sigma}, \mathbb{Z})$. 
\end{proof}

\begin{figure}
 \centering
 \begin{tikzpicture}
    \draw (-3,-3) -- (-3,3); 
    \draw (-3.3, -1.5) node {$e_1$};    
    \draw (-3.3, 1.5) node {$e_6$};    
    \draw (3,-3) -- (3,3); 
    \draw (3.3, -1.5) node {$e_1$};    
    \draw (3.3, 1.5) node {$e_6$};    
    \draw (-3,-3) -- (3,-3);
    \draw (-1.5, -3.3) node {$e_2$};    
    \draw (1.5, -3.3) node {$e_4$};    
    \draw (-3,3) -- (3,3);
    \draw (-1.5, 3.3) node {$e_2$};    
    \draw (1.5, 3.3) node {$e_4$}; 
    \filldraw [black] (-3,-3) circle (2pt) node [anchor = north] {$v_1$}; 
    \filldraw [black] (-3,3) circle (2pt) node [anchor = south] {$v_1$}; 
    \filldraw [black] (3,-3) circle (2pt) node [anchor = north] {$v_1$}; 
    \filldraw [black] (3,3) circle (2pt) node [anchor = south] {$v_1$}; 
    \filldraw [black] (0,3) circle (2pt) node [anchor = south] {$v_2$}; 
    \filldraw [black] (0,-3) circle (2pt) node [anchor = north] {$v_2$}; 
    \filldraw [black] (-3,0) circle (2pt) node [anchor = east] {$v_3$}; 
    \filldraw [black] (3,0) circle (2pt) node [anchor = west] {$v_3$}; 
    \draw (-3,0) -- (0,3);
    \draw (-1.2, 1.5) node {$e_8$};    
    \draw (-3,0) -- (0,-3);
    \draw (-1.2, -1.5) node {$e_3$};
    \draw (-3,0) -- (3,3);
    \draw (1.8, -1.5) node {$e_5$};    
    \draw (3,0) -- (0,-3);
    \draw (0, 1.2) node {$e_7$};    
    \draw (1.2, 0) node {$e_9$};    
    \draw (-3,0) -- (0,-3);
    \draw (0,-3) -- (3,3);
 \end{tikzpicture}
 \caption{A possible triangulation of the two-torus.}
\label{fig:triangulation}
\end{figure}
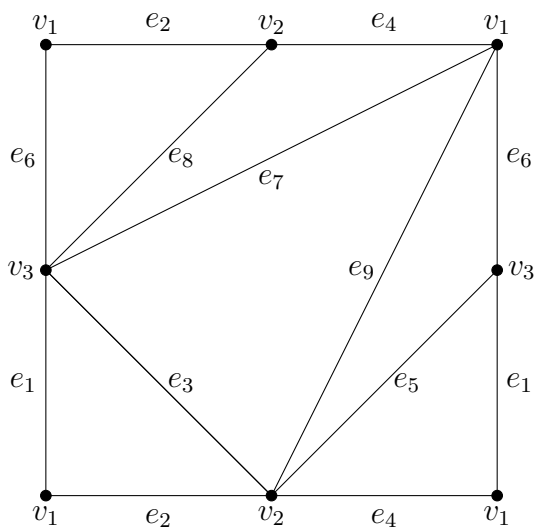

\subsubsection{Phases of the trace coordinates in terms of holonomy of the spectral line bundle}

In Lemma~\ref{lem:dominance}, we determined the dominance properties of the absolute values of the trace coordinates along rays in any open sector $S_j$. 
We have found that $X/Z \to 0$ and $Y/Z \to 0$ as $R\to\infty$. 
Moreover, in sectors $S_j$ with $j\equiv 2,3 \pmod{4}$ the limit in $\widetilde{\mathcal{M}_{\operatorname{B}}}$ is $P_2$ and in the remaining ones it is $P_1$. 
Here, we focus on the angles of the affine coordinates near these limit points found in Section~\ref{sec:blowup}. 
We find that these angles can be interpreted as the holonomy of $\mathcal{L}$ along some loops in $\widetilde{\Sigma}$, depending on the given sector; as a matter of fact, the computation on the other sectors is the same up to a permutation of the sheets. 
Due to the $2\pi / 3$ periodicity of the diagram of Figure~\ref{figure:sectors}, it is sufficient to deal with four consecutive sectors. 
In the sequel we will use the notation of the holonomy isomorphism~\eqref{eq:holonomy}. 
\begin{prop}\label{prop:angles}
\begin{enumerate}
    \item On $S_2$, we have 
    \[
        \frac{Y/X}{|Y/X|} \approx \operatorname{hol}_{B} (\mathcal{L}), \quad \frac{Y/Z}{|Y/Z|} \approx \operatorname{hol}_{2B} (\mathcal{L})
    \] 
    \item On $S_3$, we have 
    \[
        \frac{Y/X}{|Y/X|} \approx \operatorname{hol}_{-A-B} (\mathcal{L}), \quad \frac{Y/Z}{|Y/Z|} \approx \operatorname{hol}_{2B} (\mathcal{L})
    \]
    \item On $S_4$, we have 
    \[
        \frac{X/Z}{|X/Z|} \approx \operatorname{hol}_{A} (\mathcal{L}), \quad \frac{X/Y}{|X/Y|} \approx \operatorname{hol}_{A+B} (\mathcal{L})
    \]
    \item On $S_5$, we have 
    \[
        \frac{X/Z}{|X/Z|} \approx \operatorname{hol}_A (\mathcal{L}), \quad \frac{X/Y}{|X/Y|} \approx \operatorname{hol}_{-A} (\mathcal{L}). 
    \]
\end{enumerate}    
\end{prop}

\begin{proof}
We only prove the first statement of the first item and leave the rest to the readers, because the proofs follow the same line. 
We apply formulas~\eqref{coordinates} combined with Proposition~\ref{prop:dominance}. 
On $S_2$, the dominant term of $X, Y, Z$ are respectively their first, second and first terms. 
The effect of the sub-leading terms on the angles is negligible, and they only amount to a homotopy. 
We find 
\[
 \sqrt{-1} \operatorname{Im}( \ln (Y/X) ) \approx \int_{\eta_0 - \eta_1} u_2 - u_3 + \int_{\gamma_0 - \gamma_1} u_2 
\]
Now, the path $\eta_0 - \eta_1$ is precisely the straight line segment $[r, 1-r]\subset \mathbb{R}$, with opposite orientation (i.e., from $1-r$ to $r$). 
Then, 
\[
    \exp \left( \int_{\eta_0 - \eta_1} u_2 \right) = \exp \left( \int_{f_2} u_2 \right) 
\]
is the holonomy of $\mathcal{L}$ along the path $f_2$, because $u_2$ is its connection $1$-form on the corresponding sheet (the sign comes from the opposite orientation of $\eta_0 - \eta_1$ and $[r, 1-r]$). 
Similarly, 
\[
    \exp \left( - \int_{\eta_0 - \eta_1} u_3 \right) = \exp \left( - \int_{f_3} u_3 \right) 
\]
is the inverse of the holonomy of $\mathcal{L}$ along the path $f_3$. 
We infer that the integral term appearing in the argument of $X/Z$ is equal to the holonomy of $\mathcal{L}$ along the closed loop $B = f_2 * \tilde{\tilde{\gamma}}_0 * (-f_3) * (-\tilde{\tilde{\gamma}}_1)$ on $\widetilde{\Sigma}$. 
\end{proof}

\subsubsection{End of the proof}

The homotopy class of a fiber of $S$ is given by the torus $T_{\varepsilon_1, \varepsilon_2}$ of the form 
\[
    \left\vert \frac XZ \right\vert = \varepsilon_1, \quad \left\vert \frac XY \right\vert = \varepsilon_2
\]
in $U_1$ for some (equivalently, any) $0< \varepsilon_1, \varepsilon_2 \ll 1$. 
Now, if we choose $z_1\in S_4$ of absolute value $1$ then the image of $H^{-1} (R z_1 )$ by $\psi$ lies in a neighbourhood of $T_{\varepsilon_1, \varepsilon_2}$ for suitable values of $\varepsilon_1, \varepsilon_2$. 
Moreover, by Proposition~\ref{prop:angles}, the angles of the above affine coordinates are approximated by the holonomy of the spectral sheaf along $A, A+B$. 
The sub-leading terms of the asymptotic expressions only amount to a homotopy. 
As $A, A+B$ is a $\mathbb{Z}$-basis of $H_1 (\widetilde{\Sigma}, \mathbb{Z})$, we get the assertion. 

\subsection{Conditional proof of Proposition~\ref{prop:main}}\label{sec:proof_main}

Recall that the notation $\mathcal{D}\partial\mathcal{M}_B(\textbf{c})$ stands for the dual boundary complex of the compactifying divisor on the Betti side (see the paragraph following Definition~\ref{def:Betti}). By the results of Section~\ref{sec:char}, the divisor is of type $I_1$, i.e. a so-called fishtail curve, and its dual simplicial complex has a single vertex $V$ with a single loop edge $E$ corresponding to the nodal point of~\eqref{eq:nodal_curve}. 

\subsubsection{Asymptotics over the Stokes rays}

The results of this section are conditional on the following hypothesis, for which we currently have no proof. 

\begin{assn}\label{assn:Stokes}
The asymptotic expansions of~\eqref{coordinates} hold over the Stokes rays too. 
\end{assn}

However, we emphasize that the results of this section are not needed anywhere in the sequel, so the main results of the paper do not rely on this assumption. 
The only reason we include this section is to shed light on the possible behavior of the correspondence in the Stokes directions, which is otherwise not immediate to guess. 

We see from Lemma~\ref{lem:dominance} that when $x$ crosses a critical angle of the second kind that is not of the first kind at the same time, then the dominance properties of the terms \emph{within} a trace coordinate change, but the dominance \emph{between} the coordinates $X,Y$ remain unchanged. 
Therefore, for the purpose of computing the limit of the ratios of the coordinates as $R\to\infty$, it is sufficient to focus on the critical angles of the first kind, which are integer multiples of $\pi$. 

\begin{prop}\label{prop:affine_coordinates}
Let $R\to\infty$ along a ray in the Hitchin base. 
\begin{enumerate}
    \item \label{prop:affine_coordinates1}
    If $\varphi = (2k-1)\pi$ for some $k\in\mathbb{Z}$ then we have 
    \begin{align*}
        \lim \frac{\lvert X_{R,\varphi}\rvert}{\lvert Y_{R,\varphi}\rvert} & =1, \\
        \lim \frac{\lvert X_{R,\varphi}\rvert}{\lvert Z_{R,\varphi}\rvert} & \geq \frac 12, \\
        \lim \frac{\lvert Y_{R,\varphi}\rvert}{\lvert Z_{R,\varphi}\rvert} & \geq \frac 12,
    \end{align*}
    including the possibility that the latter two limits are both simultaneously $\infty$. 
    \item \label{prop:affine_coordinates2}
    If $\varphi = 2k\pi$ then we have 
    \begin{align*}
        \frac{\lvert X_{R,\varphi}\rvert}{\lvert Y_{R,\varphi}\rvert} & \to 1, \\
         \frac{\lvert X_{R,\varphi}\rvert}{\lvert Z_{R,\varphi}\rvert} & \to 0,\\
        \frac{\lvert Y_{R,\varphi}\rvert}{\lvert Z_{R,\varphi}\rvert} & \to 0.
    \end{align*}
\end{enumerate}
\end{prop}

\begin{proof}
 According to~\eqref{eq:argx}, the condition of part~\ref{prop:affine_coordinates1} is equivalent to that $x$ lies on one of the Stokes rays $R_3,R_7, R_{11}$. 
 Similarly, the condition of part~\ref{prop:affine_coordinates2} is equivalent to that $x$ lies on one of the Stokes rays $R_1,R_5, R_9$.

Let us start by treating $R_3$, defined by $a=0, b>0$. 
We use the table obtained in Proposition~\ref{prop:dominance}. 
The analytic functions $X_{R,\varphi},Y_{R,\varphi},Z_{R,\varphi}$ have asymptotic expansions on both sides of $R_3$, i.e. for $\varphi = (2k-1)\pi \pm \varepsilon$ as $R\to\infty$, for any $0< \varepsilon < \frac{\pi}{12}$. 
The term $R^{1/3}\frac{\sqrt{3}b}2$ in the maximal dilation exponents of the coordinates is common (and positive), therefore it does not change the limit of their quotients. 
On the other hand, the dependence of the maximal dilation exponents on $a$ of $X_{R,\varphi},Y_{R,\varphi},Z_{R,\varphi}$ is in order 
\[
    -\frac a2, \quad \frac a2, \quad \left\vert \frac{3a}{2} \right\vert. 
\]
Using Assumption~\ref{assn:Stokes} we see that the maximal dilation exponents of the trace coordinates agree: 
\begin{align*}
   \lvert X_{R,\varphi}\rvert\approx\exp\left( R^{1/3}\frac{\sqrt{3}}{2}b\right) , \\
    \lvert Y_{R,\varphi}\rvert\approx\exp\left(R^{1/3}\frac{\sqrt{3}}{2}b \right) , \\
    \lvert Z_{R,\varphi}\rvert\approx \exp\left(R^{1/3}\frac{\sqrt{3}}{2}b \right) \left\vert e^{\sqrt{-1}\varpi} + e^{\sqrt{-1}\varpi'} \right\vert 
\end{align*}
for some $\varpi, \varpi'\in\mathbb{R}$, coming from the purely imaginary terms in the argument of the first and third terms of the expansion of $Z$ in~\eqref{coordinates}. 
The statement for $R_3$ follows. 
The argument is similar for $R_7, R_{11}$. 

Similarly, $R_9$ is defined by $a=0, b<0$. By virtue of Proposition~\ref{prop:dominance}, we see that the maximal dilation exponents of $X,Y$ agree, but this time they differ from the one of $Z$: 
\begin{align*}
    \lvert X_{R,\varphi}\rvert\approx\exp\left( -R^{1/3}\frac{\sqrt{3}}{2}b\right) , \\
    \lvert Y_{R,\varphi}\rvert\approx\exp\left( -R^{1/3}\frac{\sqrt{3}}{2}b \right) , \\
    \lvert Z_{R,\varphi}\rvert\approx\exp\left(- R^{1/3}\sqrt{3}b \right) . 
\end{align*}
(Here, there is only one dominant term in $Z$, therefore the phase phenomenon of the odd case does not occur.) 
Again, a similar argument works for $R_1,R_5$. 
\end{proof}

\subsubsection{Critical angles and arc decomposition}
Let us now consider $X_{R,\varphi}$ and $Y_{R,\varphi}$ for fixed $R\gg 1$ as functions of $\varphi$, see formulas~\eqref{coordinates}. 
Notice that these formulas depend on $\varphi$ through the factor $e^{i\varphi/3}$. 
As $\varphi$ ranges over $[0,2\pi ]$, the argument of the coordinate $x$ is affected by one third turn in the $(a,b)$-coordinate plane (see Figure~\ref{figure:sectors}). 
The sectors $S_j$ have opening angles $\pi/6$. 
As $\varphi$ ranges over $[0,2\pi ]$, $x$ will cross exactly two rays between white and shaded sectors. 
Remember that $X_{R,\varphi}$ dominates $Y_{R,\varphi}$ over white sectors and vice versa on shaded sectors. 
Therefore there exist two critical angles $\varphi_1=\pi$ 
and $\varphi_2=2\pi$, 
where the dominance order between $\lvert X_{R,\varphi}\rvert$ and $\lvert Y_{R,\varphi}\rvert$ changes. 
Lemma~\ref{lem:dominance} may be reformulated as: 
\begin{lemma}\label{lem:arc_decomposition}
The critical angles $\varphi_1, \varphi_2$ decompose $S^1_R$ into two closed arcs
\begin{equation*}
    S^1_R=I_1\cup I_2,
\end{equation*}
such that (up to relabeling)
\begin{itemize}
    \item for $\varphi\in\operatorname{Int}I_1$, $\lvert X_{R,\varphi}\rvert\gg \lvert Y_{R,\varphi}\rvert$
    \item for $\varphi\in\operatorname{Int}I_2$, $\lvert Y_{R,\varphi}\rvert\gg \lvert X_{R,\varphi}\rvert$.
\end{itemize}
\end{lemma}

\subsubsection{Identifying the images of suitable arcs}

Fix $0 < \varepsilon \ll 1$ and let us decompose 
\begin{align*}
    [0, 2\pi ] & = [0, \varepsilon ] \cup [\varepsilon , \pi - \varepsilon ] \cup [\pi - \varepsilon, \pi + \varepsilon ] \cup [\pi + \varepsilon , 2\pi - \varepsilon ] \cup [2\pi - \varepsilon , 2\pi ] \\
    & = A_1 \cup A_2 \cup A_3 \cup A_4 \cup A_5. 
\end{align*}
It follows from Lemmas~\ref{lem:dominance} and~\ref{lem:arc_decomposition} that 
\[
    A_2 \subset I_2, \qquad A_4 \subset I_1 
\]
after identifying $[0, 2\pi ]$ with $S^1_R$. 

Proposition~\ref{prop:main} is a straightforward consequence of:  
\begin{lemma}
There exists $R_0 = R_0(\varepsilon)>0$ such that for all $R> R_0$, $S\circ \psi \circ \sigma$ maps 
\begin{enumerate}
    \item $A_2$ onto a connected subset of the edge of $\mathcal{D}\partial\mathcal{M}_B(\textbf{c})$ defined by $\phi_1 = 0$, 
    \item $A_3$ onto a connected subset of $\mathcal{D}\partial\mathcal{M}_B(\textbf{c})$ containing $(1,0,0)$, 
    \item $A_4$ onto a connected subset of the edge of $\mathcal{D}\partial\mathcal{M}_B(\textbf{c})$ defined by $\phi_2 = 0$, 
    \item $A_1\cup A_5$ onto a connected subset of the complement of $(1,0,0)$ in $\mathcal{D}\partial\mathcal{M}_B(\textbf{c})$. 
\end{enumerate}
\end{lemma}

For convenience we depict the geometry of the maps in Figure~\ref{fig:Simpson}. 

\begin{figure}
    \centering
    \includegraphics[width=13cm]{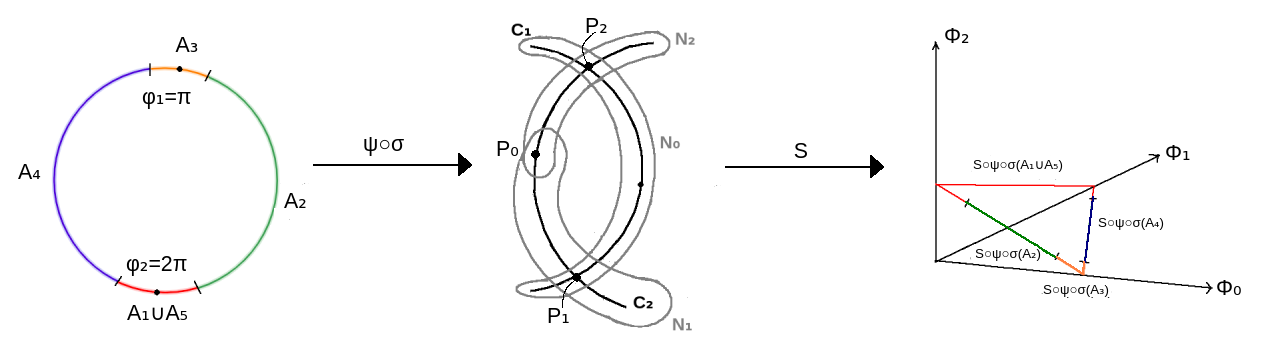}
    \caption{The map on the left is non-abelian Hodge and Riemann--Hilbert, the one on the right is Simpson's map}
    \label{fig:Simpson}
\end{figure}

\begin{proof}
Connectedness of the image follows in each case by continuity. 
Recall from Section~\ref{sec:blowup} the natural pair of homogeneous coordinates $[X'\colon Y']$ on the exceptional curve $C_2$ satisfying 
\[
     \frac XY = \frac{X'}{Y'}. 
\]
By shrinking $N_1, N_2$ we may assume that 
\begin{align}
    [X\colon Y \colon Z ] & \in N_0 \cap N_2 \cap C_2 \Rightarrow \left\vert \frac{X'}{Y'} \right\vert < \frac 1{10} \label{eq:N2} \\
    [X\colon Y \colon Z ] & \in N_0 \cap N_1 \cap C_2 \Rightarrow \left\vert \frac{Y'}{X'} \right\vert < \frac 1{10} \label{eq:N1} .
\end{align}
(The constants $1/10$ on the right hand side could be replaced by any positive constant strictly inferior to $1$.) 
Let us denote the point 
\[
    [X'\colon Y'] = [1\colon 1] \in C_2
\]
by $P_0$. 
We may then assume that 
\[
    P_0 \in N_1\cap N_2 \cap C_2, 
\]
because by inequalities~\eqref{eq:N2}-\eqref{eq:N1}, we have 
\[
    P_0 \notin N_0 .
\]
This formula also implies that 
\[
    \phi_0(P_0 ) = 0 . 
\]
On the other hand, the point $P_1$ introduced in Section~\ref{sec:Simpsons_map} is given by 
\[
    [X'\colon Y'] = [1\colon 0] 
\]
and $P_2$ by 
\[
    [X'\colon Y'] = [0\colon 1] .
\]
After this introduction, we are ready to treat the cases one by one. 
\begin{enumerate}
\item According to Lemma~\ref{lem:arc_decomposition}, $\psi \circ \sigma$ maps $A_2 \subset \operatorname{Int}I_2$ into a tiny neighborhood of the point $P_2$. 
Since $P_2\in N_0 \cap N_2$ and the triple intersection of the $N_i$ vanish, we see that $P_2 \notin N_1$. 
The statement follows because the partition of unity is subordinate to~\eqref{eq:cover}. 
\item 
By Proposition~\ref{prop:affine_coordinates}~\ref{prop:affine_coordinates1}, we have 
\[
    \lim_{R\to \infty} \psi \circ \sigma\left( R e^{\sqrt{-1} \pi} \right) \neq [0\colon 0\colon 1]\in C 
\]
because the limits of $X/Z$ and $Y/Z$ do not vanish. 
Said differently, the limit is not the nodal point of $C$. 
In addition, inequalities~\eqref{eq:N2}-\eqref{eq:N1} show that in the blow-up of the nodal point, the stronger condition 
\[
    \lim_{R\to \infty} \psi \circ \sigma\left( R e^{\sqrt{-1} \pi} \right) \notin N_1 \cup N_2 
\] 
holds too, because the limit satisfies $Y'/X'=1$. 
Since the partition of unity is subordinate to the cover~\eqref{eq:cover}, we infer that $S$ maps the above limit to $(1,0,0)$. 
By continuity, the same holds over $A_3$ too, for sufficiently large $R$. 
\item Completely similar to $A_2$. 
\item 
Proposition~\ref{prop:affine_coordinates}~\ref{prop:affine_coordinates2} implies that $\psi \circ \sigma( R )\to P_0$ as $R\to \infty$. 
Therefore, $S\circ \psi \circ \sigma( R e^{i\varphi_2}) \neq (1,0,0)$ for $R\gg 1$. 
The same then follows for any $\varphi\in A_1 \cup A_5$ by continuity. 
\end{enumerate}
\end{proof}

\bigskip
\textbf{Funding}
\bigskip
\\ The project supported by the Doctoral Excellence Fellowship Programme (DCEP) is funded by the National Research Development and Innovation Fund of the Ministry of Culture and Innovation and the Budapest University of Technology and Economics, under a grant agreement with the National Research, Development and Innovation Office. During the preparation of this manuscript, the authors were supported by the grant K146401 of the National Research, Development and Innovation Office.
The second author was also supported by the grant KKP144148 of the National Research, Development and Innovation Office.

%

\end{document}